\documentclass{amsart}

\usepackage{amssymb}
\usepackage{graphicx, color, xcolor}
\usepackage{graphics}
\usepackage{pinlabel}
\usepackage{bm}
\usepackage{framed}
\usepackage{overpic}
\usepackage{enumitem}
\usepackage{amscd}
\usepackage[title]{appendix}
\usepackage{scalerel}
\usepackage{array} 

\usepackage{tikz, mathtools}
\usepackage{tikz-cd}
\usetikzlibrary{matrix,decorations.pathmorphing,arrows}

\usepackage{hyperref}
\hypersetup{
	colorlinks,
	linkcolor={magenta}, 
	citecolor={cyan},
	urlcolor={blue}
}
\usepackage[]{cleveref} 	

\usepackage[]{todonotes}

\usepackage{thmtools}
\usepackage{thm-restate}
\newtheorem*{morse_lemma}{Morse Lemma}

\numberwithin{equation}{section}

\newtheorem{theorem}[equation]{Theorem}
\newtheorem{lemma}[equation]{Lemma}
\newtheorem{proposition}[equation]{Proposition}
\newtheorem{corollary}[equation]{Corollary}

\newtheorem{fact}[equation]{Fact}

\newtheorem{observation}[equation]{Observation}
\theoremstyle{definition}

\declaretheorem[numbered=no,style=remark, qed=$\lozenge$]{notation}
\declaretheorem[style=remark,qed=$\lozenge$,sibling=theorem]{remark}
\declaretheorem[style=remark,qed=$\lozenge$,sibling=theorem]{example}
\declaretheorem[style=definition,qed=$\lozenge$,sibling=theorem]{definition}
\declaretheorem[style=definition,qed=$\lozenge$,sibling=theorem, name=Construction]{construction}

\newcommand{\bbR}{\mathbb{R}}
\newcommand{\bbZ}{\mathbb{Z}}
\newcommand{\bbH}{\mathbb{H}}

\newcommand{\bbD}{\mathbb{D}}

\newcommand{\acts}{\curvearrowright}
\newcommand{\racts}{\curvearrowleft}

\newcommand{\rquotient}[2]{{\left.\raisebox{.2em}{$#1$}\middle/\raisebox{-.2em}{$#2$}\right.}}
\newcommand{\lquotient}[2]{{\left.\raisebox{-.2em}{$#1$}\middle\backslash\raisebox{.2em}{$#2$}\right.}}

\newcommand{\cD}{\mathcal{D}}

\newcommand{\cP}{\mathcal{P}}

\newcommand{\cI}{\mathcal{I}}
\newcommand{\cF}{\mathcal{F}}

\newcommand{\bcH}{\bm{\mathcal{H}}}

\newcommand{\cL}{\mathcal{L}}
\newcommand{\bcL}{\bm{\mathcal{L}}}
\newcommand{\tcL}{\widetilde{\mathcal{L}}}

\newcommand{\bcD}{\bm{\mathcal{D}}}

\newcommand{\bcF}{\bm{\cF}}

\newcommand{\tcW}{\widetilde{\mathcal{W}}}

\newcommand{\bP}{{\mathbf{P}}}
\newcommand{\bQ}{{\mathbf{Q}}}

\newcommand{\bD}{\mathbf{D}}

\newcommand{\tM}{\widetilde{M}}

\newcommand{\tPhi}{\widetilde{\Phi}}
\newcommand{\tPsi}{\widetilde{\Psi}}

\newcommand{\be}{\mathbf{e}}

\newcommand{\bK}{\bm{K}}
\newcommand{\bL}{\bm{L}}

\newcommand{\ppmm}{{\scaleto{\pm \pm}{4pt}}}
\newcommand{\pp}{{\scaleto{+ +}{4pt}}}
\newcommand{\mm}{{\scaleto{- -}{4pt}}}

\newcommand{\td}{{\widetilde{d}}}
\newcommand{\tS}{\widetilde{S}}

\newcommand{\tG}{{\widetilde{G}}}

\newcommand{\flow}[1]{\left\langle #1 \right\rangle}

\newcommand{\bs}{\mathbf{s}}

\newcommand{\cA}{\mathcal{A}}
\newcommand{\cB}{\mathcal{B}}

\newcommand{\Om}{\mathcal{O}}
\newcommand{\tOm}{\widetilde{\mathcal{O}}}
\newcommand{\bOm}{\mathcal{\bm{\Om}}}
\newcommand{\Sp}{\mathcal{S}}
\newcommand{\tSp}{\widetilde{\mathcal{S}}}
\newcommand{\bSp}{\bm{\mathcal{S}}}

\definecolor{maroon}{RGB}{128,0,0}
\definecolor{navy}{RGB}{0,0,128}
\definecolor{dkgreen}{RGB}{0,128,0}

\DeclareMathOperator{\Hull}{H}

\DeclareMathOperator{\diam}{diam}

\newcommand{\lk}{\mathrm{L}}
\newcommand{\rake}{\mathrm{rake}}

\DeclareMathOperator{\Mon}{Mon}

\newcommand{\capdot}{\mathbin{\mathaccent\cdot\cap}}

\newcommand{\wt}{\widetilde}
\newcommand{\mr}{\mathring}
\newcommand{\del}{\partial}

\newcommand{\tTheta}{\widetilde{\Theta}}
\newcommand{\su}{{s/u}}
\newcommand{\ssuu}{{ss/uu}}

\newcommand{\bcZ}{\boldsymbol{\mathcal{Z}}}
\newcommand{\bZ}{\mathbf{Z}}

\makeatletter
\newcommand{\setword}[2]{%
	\phantomsection
	#1\def\@currentlabel{\unexpanded{#1}}\label{#2}%
}
\makeatother

\makeatletter
\renewcommand\part{%
   \if@noskipsec \leavevmode \fi
   \par
   \addvspace{4ex}%
   \@afterindentfalse
   \secdef\@part\@spart}

\def\@part[#1]#2{%
    \ifnum \c@secnumdepth >\m@ne
      \refstepcounter{part}%
      \addcontentsline{toc}{part}{\thepart\hspace{1em}#1}%
    \else
      \addcontentsline{toc}{part}{#1}%
    \fi
    {\parindent \z@ \center
     \interlinepenalty \@M
     \normalfont
     \ifnum \c@secnumdepth >\m@ne
       \Large\bfseries \partname \nobreakspace\thepart :
     \fi
     \Large \bfseries #2%
     \par}%
    \nobreak
    \vskip 3ex
    \@afterheading}
\def\@spart#1{%
    {\parindent \z@ \raggedright
     \interlinepenalty \@M
     \normalfont
     \Large \bfseries #1\par}%
     \nobreak
     \vskip 5ex
     \@afterheading}
\makeatother

\title{From quasigeodesic to pseudo-Anosov flows}
\author{Steven Frankel}
\address{Department of Mathematics\\Washington University in St. Louis\\St. Louis, MO}
\email{steven.frankel@wustl.edu}
\author{Michael Landry}
\address{Department of Mathematics and Statistics\\Saint Louis University\\St. Louis, MO}
\email{michael.landry@slu.edu}

\begin{document}

\begin{abstract}
    We prove Calegari's conjecture that every quasigeodesic flow on a closed hyperbolic $3$-manifold can be deformed to a flow that is simultaneously quasigeodesic and pseudo-Anosov.
\end{abstract}

\maketitle
\setcounter{tocdepth}{1}
\tableofcontents

\section{Introduction}
A flow on a $3$-manifold is \emph{quasigeodesic} if each flowline is coarsely comparable to a geodesic, and \emph{pseudo-Anosov} if it has a transverse contracting-expanding structure governed by weak stable and unstable singular foliations.
These are very different conditions that regard tangent and transverse behavior, respectively. Nevertheless, Calegari conjectured a close relationship between quasigeodesic and pseudo-Anosov flows in the presence of ambient hyperbolicity: that one can deform any quasigeodesic flow on a closed hyperbolic $3$-manifold into a flow that is simultaneously quasigeodesic and pseudo-Anosov. 

Pseudo-Anosov flows are very rigid objects. On a closed hyperbolic $3$-manifold, for example, every fibration is transverse to a unique pseudo-Anosov flow up to orbit equivalence \cite{Fried_fiberedfaces}, and each pseudo-Anosov flow is determined up to orbit equivalence by its free homotopy classes of periodic orbits \cite{BarthelmeFrankelMann}. A long-standing conjecture asserts that a closed $3$-manifold can admit only finitely many pseudo-Anosov flows up to orbit equivalence.

Quasigeodesic flows are much more flexible, in contrast, since quasigeodesity is a $C^0$-open condition. For example, every flow transverse to a codimension-$1$ fibration of a closed manifold is quasigeodesic \cite{Zeghib}, and one can easily produce uncountably many such flows in different orbit equivalence classes.

Along with this rigidity, the connection between local dynamics and global topology provided by their associated singular foliations makes pseudo-Anosov flows invaluable in the geometric study of $3$-manifolds. For example:
\begin{itemize}

\item A consequence of Fried's result above is that each pseudo-Anosov flow determines a complete fibered face of the unit ball of the Thurston norm. This was generalized by Mosher \cite{Mosher_norm1, Mosher_norm2} (see also \cite{LandryMinskyTaylor_transversesurfaces}), who conjectured that the Thurston norm of a 3-manifold can always be computed using finitely many pseudo-Anosov flows \cite{Mosher}.
 
\item Calegari \cite{Calegari_Rcovered} and Fenley \cite{Fenley_Rcovered} showed that every taut, $\bbR$-covered foliation of a closed atoroidal 3-manifold is transverse to a regulating pseudo-Anosov flow, which Fenley later showed is usually unique in a precise sense \cite{Fenley_Rcoveredrigidity} (see also \cite[Corollary 5.3.22]{Calegari_Rcovered}). These regulating flows have played an important role in the study of partially hyperbolic diffeomorphisms \cite{BFFP1, BFFP2} and in finding compact invariant sets for diffeomorphism with positive ``escape rate'' with respect to an $\bbR$-covered foliation \cite{GMP}.

\item Fenley used pseudo-Anosov flows to understand the asymptotic behavior of leaves of taut foliations in a hyperbolic 3-manifold \cite{Fenley_continuousextension} and the ideal geometry of the ambient manifold itself \cite{Fenley}. 

\item Alfieri and Tsang have recently established a connection between pseudo-Anosov flows and Heegaard Floer homology \cite{AlfieriTsang1, AlfieriTsang2}.
\end{itemize}

To say that a flow is pseudo-Anosov is to make a global statement about how orbits interact with each other in forward and backward time, while the quasigeodesic condition can be checked orbit-by-orbit. Together with their flexibility, this makes it much simpler in principle to tailor a quasigeodesic flow to a geometric problem, and to certify that a given flow is quasigeodesic. For example, Mosher constructs examples of flows transverse to finite depth foliations in \cite{Mosher_examples} and proves directly that they are quasigeodesic. The resulting flows can be made pseudo-Anosov, but this requires the long and currently incomplete machinery of Gabai-Mosher \cite{Mosher}.

Our main result is a proof of Calegari's conjecture:

\begin{restatable}[Main Theorem]{theorem}{maintheorem}\label{theorem_main}
	Let $\Phi$ be a quasigeodesic flow on a closed hyperbolic $3$-manifold $M$. Then there is
	\begin{enumerate}
		\item a quasigeodesic pseudo-Anosov flow $\Psi$ on $M$,
		\item a closed $\Phi$-invariant subset $M_\lk \subset M$, and
		\item a surjective map $M_\lk \to M$, homotopic to the inclusion, that takes each oriented orbit of $\Phi|_{M_\lk}$ monotonically to an oriented orbit of $\Psi$.
	\end{enumerate}
\end{restatable}

We will give an explicit description of the subset $M_\lk$, which we call the \emph{linked subset}. We call the map $M_\lk \to M$ an \emph{orbit semiconjugacy} from $\Phi|_{M_\lk}$ to $\Psi$. Note that even the statement that every closed hyperbolic $3$-manifold supporting a quasigeodesic flow also supports a pseudo-Anosov flow was not previously known.

\begin{remark}
    \Cref{theorem_main} is analogous to Handel's global shadowing theorem \cite{Handel}. This says that for any homeomorphism $g\colon S\to S$ homotopic to a pseudo-Anosov $f$ one can find a closed, $g$-invariant subset $X\subset S$ and a surjective map $\pi\colon X\to S$, homotopic to inclusion, that semiconjugates $g|_X$ to $f$.
\end{remark}

\begin{remark}
    Our main theorem provides a potential alternative construction of pseudo-Anosov flows in closed hyperbolic manifolds with positive first betti number. This result was outlined by Gabai and Mosher but never completed \cite{Mosher} (see also \cite{LandryTsang}).
    
    Given a closed hyperbolic 3-manifold $M$ with positive first betti number, Gabai constructs a taut, finite depth foliation of $M$ (\cite{Gabai}), and Fenley and Mosher construct a quasigeodesic flow transverse to such a foliation \cite{FenleyMosher}. One can now apply our \Cref{theorem_main} to turn such a flow pseudo-Anosov. 
    
    This argument is currently circular: the Fenley-Mosher construction works by showing that the Gabai-Mosher pseudo-Anosov flows are themselves quasigeodesic. However, this may be fixable since the Fenley-Mosher argument really only requires a flow that is transverse to a taut finite-depth foliation with Hausdorff flowspace.
\end{remark}

\begin{remark}
We do not address uniqueness here. It seems likely that for each quasigeodesic flow $\Phi$ there is a unique quasigeodesic pseudo-Anosov flow $\Psi$, up to orbit equivalence, such that the restriction of $\Phi$ to an invariant subset is orbit semiconjugate to $\Psi$. Furthermore, $M_\lk$ should be the maximal such subset.
\end{remark}

\begin{remark}
For $C^1$ flows one can take a deformation of flows to mean a homotopy of nonvanishing vector fields. In this context Calegari's conjecture is that every quasigeodesic flow on a closed hyperbolic $3$-manifold is homotopic, through quasigeodesic flows, to a flow that is simultaneuously quasigeodesic and pseudo-Anosov. See \cite{Calegari_blog}. This provides less dynamical information than \Cref{theorem_main}, which is the original version of Calegari's conjecture. In the $C^1$ context Calegari further conjectures that there is a unique quasigeodesic pseudo-Anosov flow, up to orbit equivalence, in each homotopy class of nonvanishing vector fields containing a quasigeodesic flow. 

In the topological context, a potential extension of \Cref{theorem_main} would be to find a path from $\Phi$ to $\Psi$ in the space of nonstationary $\bbR$-actions on $M$. Then an analogous uniqueness statement would be to show that there is a unique quasigeodesic pseudo-Anosov flow, up to orbit equivalence, in each component of this space that contains a quasigeodesic flow.
\end{remark}

\subsection{Outline}
We will now outline the main ideas in the proof of \Cref{theorem_main}. The paper is organized somewhat nonlinearly with respect to this proof so that \Cref{part:circledisc}, which constructs transverse pairs of singular $1$-dimensional foliations from data at infinity, can be made logically independent. One can read the paper linearly or begin with \Cref{sec:flowbackground} before returning to \Cref{part:circledisc}. The outline roughly follows the latter route.

\bigskip
Consider a closed hyperbolic $3$-manifold $M$, and let $\tM$ denote its universal cover. This may be geometrically identified with $\bbH^3$, from which it inherits a compactification by a $2$-sphere $S^2_\infty$ at infinity.

A \hyperref[sec:flowbackground]{\emph{flow}} on $M$ lifts to a flow on $\tM$. The set of orbits of the lifted flow has a natural topology as a quotient of $\tM$, and the deck action $\pi_1(M) \acts \tM$ induces an action on this \emph{\hyperref[subsection:flowspaces]{flowspace}}. When $\Phi$ is either quasigeodesic or pseudo-Anosov this flowspace is homeomorphic to a plane and denoted by $P$.

For both of these classes of flows there is a natural \emph{universal circle}, a topological circle $S^1_u$ that compactifies the flowspace to a closed disc $P \sqcup S^1_u$ on which the action of $\pi_1(M)$ extends (see \Cref{sec:flowbackground} for details):
\begin{itemize}
    \item Given a \hyperref[sec:pAflows]{\emph{pseudo-Anosov flow}} $\Phi$, the $2$-dimensional weak stable and unstable singular foliations on $M$ (\Cref{figure:2dfoliations}) lift to $2$-dimensional singular foliations on $\tM$ and project to a transverse pair of {$1$-dimensional} singular foliations on $P$ (\Cref{fig:qgneednotlink}, left). We call these the \emph{stable} and \emph{unstable foliations} and denote them by $\cF^s$ and $\cF^u$.
    
    Fenley constructed a natural compactification of the flowspace to a closed disc $P \sqcup S^1_u$, whose boundary circle $S^1_u$ is built out of the ends of leaves of these foliations. The flowspace  action $\pi_1(M) \acts P$ preserves the foliations $\cF^\su$, permuting the ends of leaves, and therefore extends to an action $\pi_1(M) \acts P \sqcup S^1_u$.

    \item A \hyperref[sec:QGbackground]{\emph{quasigeodesic flow}} $\Phi$ does not come with a pair of singular foliations, but Calegari constructed a pair of topological decompositions of $P$ into \emph{positive} and \emph{negative \hyperref[def:WeakOmnileavesLeaves]{leaves}}, denoted by $\cL^+$ and $\cL^-$, that serve an analogous role (\Cref{fig:qgneednotlink}, right). He used the Freudenthal ends of elements of these decompositions to construct a universal circle $S^1_u$ and showed that there is an induced action $\pi_1(M) \acts S^1_u$. Frankel showed that there is a topology making $P \sqcup S^1_u$ a closed disc, and that the actions combine to an action $\pi_1(M) \acts P \sqcup S^1_u$. 
\end{itemize}

\begin{figure}[h]
    \centering
    \includegraphics[height=1.8in]{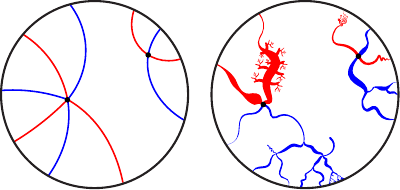}
    \caption{Left: the stable and unstable foliations of the flowspace of a pseudo-Anosov flow form a pair of transverse singular foliations. The ends of the positive and negative leaves through any point alternate in the universal circle. Right: the positive and negative decompositions of the flowspace of a quasigeodesic flow may be less nice. The positive and negative leaves through a point may or may not have ends that link in the universal circle.}
    \label{fig:qgneednotlink}
\end{figure}

A large part of our argument takes place in the flowspace $P$ and universal circle $S^1_u$. Given a quasigeodesic flow $\Phi$, we construct from the flowspace $P$ and decompositions $\cL^\pm$ a new ``straightened flowspace.'' This is a topological plane $Q$ that comes with an action $\pi_1(M) \acts Q$ that preserves transverse pair of singular foliations $\cF^\pm$. 

In fact, the straightened flowspace, along with its action and foliations, is entirely constructed out of data that lives in the universal circle $S^1_u$. The decompositions $\cL^\pm$ of $P$ live in a natural pair of decompositions $\bSp^\pm$ of the compactified flowspace $\bP = P \sqcup S^1_u$ whose elements are called {positive} and negative \hyperref[def:PomnileavesPsprigs]{\emph{sprigs}}. These restrict to a pair of decompositions $\Lambda^\pm$ of the universal circle $S^1_u = \partial \bP$.

Since it is likely of independent interest, we have abstracted the construction of the straightened flowspace from the decompositions $\Lambda^\pm$ in the logically independent \Cref{part:circledisc} of the paper.

\subsubsection*{\Cref{part:circledisc}: Decompositions of the circle and disc}
The decompositions $\Lambda^\pm$ on the universal circle of a quasigeodesic flow have special properties which we abstract in the definition of an \hyperref[definition:spS]{\emph{especial pair}}. We provide a straightforward hull-and-quotient construction that turns any especial pair $\Lambda^\pm$ on a circle $S^1$ into a pair of decompositions $\bcF^\pm$ of a closed disc. One first takes the convex hull of each element of $\Lambda^+$ and $\Lambda^-$, and shows that this yields a pair of decompositions $\bcH^\pm$ of the Euclidean disc with appropriate properties. Each nontrivial intersection of a positive hull with a negative hull is a convex polygon (\Cref{figure:Sectors}), and one can quotient each of these polygons to a point to obtain a new disc $\bQ$, on which the hull decompositions collapse to  decompositions $\bcF^\pm$ that intersect \hyperref[def:muuemuu]{\emph{efficiently}} (\Cref{fig:collapsingexamples}).

We then leverage this efficient intersection to show that the $\bcF^\pm$ on $\bQ$ restrict to a transverse pair of \hyperref[sec:emuupairsgivefoliations]{\emph{singular foliations}} $\cF^\pm$ on the interior $Q = \mathring{\bQ}$. 

If the decompositions $\Lambda^\pm$ come from restricting a pair of decompositions $\bSp^\pm$ of a disc $\bP$, then there is a natural identification between the boundary circles of $\bQ$ and $\bP$ (both are identified with $S^1_u$ in our case). Under this identification, the new decompositions $\bcF^\pm$ on $\bQ$ and the original decompositions $\bSp^\pm$ on $\bP$ restrict to the same especial pair $\Lambda^\pm$. This provides a natural map $\bs$ that associates certain points in $\bP$ to points in $\bQ$: a point $p \in \bP$ should be taken to a point $\bs(p) \in \bQ$ such that the elements of $\bSp^\pm$ through $p$ and the elements of $\bcF^\pm$ through $\bs(p)$ restrict to the same elements of $\Lambda^\pm$ on the boundary circle. This is only possible, however, if these elements are \emph{\hyperref[subsection:linking]{linked}} in the boundary circle, so that the corresponding convex hulls intersect nontrivially. See \Cref{fig:qgneednotlink}.
This yields a \hyperref[sec:Straightening]{\emph{straightening map}}
\[ \bs: \bP_\lk \to \bQ \]
that takes an appropriately defined \emph{linked subset} $\bP_\lk \subset \bP$ surjectively onto $\bQ$. The decompositions $\bSp^\pm$, restricted to $\bP_\lk$, are taken to the singular foliations $\bcF^\pm$.

In the case of a quasigeodesic flow, the decompositions $\bSp^\pm$ have some additional intersection properties that ensure $\bs$ takes interior points to interior points. It therefore restricts to a surjective straightening map
\[ s: P_\lk \to Q \]
that takes the decompositions $\cL^\pm$, restricted to $P_\lk = \bP_\lk \cap P$, to the singular foliations $\cF^\pm$.

All of this can be done in a way that respects group actions. In our case, the action $\pi_1(M) \acts S^1_u$ preserves $\Lambda^\pm$ and hence determines a unique action $\pi_1(M) \acts \bQ$ that preserves the singular foliations $\cF^\pm$. We now have have a potential flowspace for a pseudo-Anosov flow: $Q$.

\subsubsection*{\Cref{part:flowstraightening}: Straightening quasigeodesic flows}

To see that $Q$, together with the given action of $\pi_1(M)$, is the flowspace of a pseudo-Anosov flow, it suffices to construct a free and properly discontinuous action of $\pi_1(M)$ on $Q \times \bbR$ that projects to the given action under the projection to the first factor. Then the quotient
\[ N = \lquotient{\pi_1(M)}{Q \times \bbR} \]
would be a closed $3$-manifold (with fundamental group $\simeq \pi_1(M)$), and the foliation of $\wt N =Q \times \bbR$ by vertical lines $\{q\} \times \bbR$ would project to a foliation on $N$ that can be parametrized to form a flow. In other words, the action $\pi_1(M) \acts Q$ tells us the ``horizontal'' part of an action $\pi_1(M) \acts Q \times \bbR$, or how such an action should permute flowlines. We need to need to construct the ``vertical'' part of the action by deciding how flowlines are mapped to each other.

Let $\tM_\lk \subset \tM$ be the union of all flowlines that lie in the linked subset $P_\lk$. We have a surjective, $\pi_1(M)$-equivariant straightening map
\[ s\colon P_\lk \to Q. \]
This is in fact a quotient map, and one can think of the action $\pi_1(M) \acts Q$ as being induced by the action $\pi_1(M) \acts P_\lk$ via $s$. The universal cover $\tM$ may be identified with $P \times \bbR$ in such a way that flowlines of $\Phi$ correspond to vertical lines, but there are many choices of such identifications or \emph{\hyperref[section:productstructures]{product structures}}. Our approach is to find such a product structure, at least for the linked subset $\tM_\lk \simeq P_\lk \times \bbR$, such that the map
\begin{align*}
    \tS\colon P_\lk \times \bbR &\to Q\times \bbR \\
    (p, h) &\mapsto (s(p), h)
\end{align*}
is a quotient map that respects the action $\pi_1(M) \acts P_\lk$ and hence induces an action $\pi_1(M) \acts Q$.

The key to finding such a product structure is to construct ``strong positive and negative decompositions" for quasigeodesic flows analogous to the strong stable and unstable foliations of pseudo-Anosov flows. In \Cref{section:comparisonmaps} we construct well-behaved \emph{\hyperref[def:comparisonMap]{comparison maps}} that send the orbits of our quasigeodesic flow on $M$ to orbits of the geodesic flow on the unit tangent bundle $T^1 M$. This is an Anosov flow whose strong stable and unstable leaves correspond to horospheres. In \Cref{section:strongdecompositions} we pull back the strong stable and unstable foliations of the geodesic flow to obtain strong stable and unstable decompositions of our quasigeodesic flow. 

In \Cref{section:productstructures} we show that the topology of the strong decompositions on $\tM$ is simple, at least over the linked subset. In particular, we construct product structures $\tM_\lk \simeq P_\lk \times \bbR$ in which the strong positive leaves lie in horizontal slices $P_\lk \times \{h\}$. One can also make the strong negative leaves horizontal, but typically not at the same time.

In \Cref{section:StraighteningFlow} we show that these product structures are the right ones for our purposes. We construct a map
\[ \tS\colon P_\lk \times \bbR \to Q \times \bbR\]
as above and a $\pi_1(M)$-action on $Q \times \bbR$ with respect to which this map is equivariant. The quotient
\[ M_s = \lquotient{\pi_1(M)}{Q \times \bbR} \]
is a $3$-manifold. This comes with a flow, the quotient of the vertical flow on $Q\times \bbR$, and the map $\tS$ descends to a map
\[ S\colon M_\lk\to M_s \]
that semiconjugates the two flows. We then produce a homeomorphism $M_s \to M$ such that the induced map $M_\lk\to M$ is homotopic to inclusion.

Finally, it remains to show that the induced flow $\Psi$ on $M$ is quasigeodesic and pseudo-Anosov. Of these, checking the pseudo-Anosov property is more involved. This is carried out in \Cref{section:StraightenedFlowsPA}.

\subsection*{Acknowledgments}
Thanks to Ian Agol, Thomas Barthelme, Danny Calegari, Larry Conlon, Sérgio Fenley, David Gabai, Jeremy Kahn, Anatole Katok, Rafael Potrie, Rachel Roberts, Matthew Stover, Samuel Taylor, Chi Cheuk Tsang, Bob Williams, and Juliana Xavier for many helpful discussions.

SF was supported by the Eric and Wendy Schmidt Foundation, and the National Science Foundation under DMS-1128155,
DMS-1611768, and DMS-2045323. ML was supported by NSF DMS-2405453.

\section{Topological background}\label{sec:TopologicalBackground}

In this preliminary section we review the notions of Kuratowski and Hausdorff convergence, upper semicontinuous decompositions, and some topology of the disc and circle. The reader may wish to skip ahead and return as required.

\subsection{Kuratowski and Hausdorff convergence}
\subsubsection{Kuratowski convergence}\label{subsection:Kuratowskiconvergence}
Let $X$ be a topological space, and let $(A_i)_{i=1}^\infty$ be a sequence of subsets of $X$. The \emph{limit superior} $\limsup A_i$ is the set of all points $x \in X$ such that every neighborhood of $x$ intersects infinitely many of the $A_i$. The \emph{limit inferior} $\liminf A_i$ is the set of all points $x \in X$ such that every neighborhood of $x$ intersects all but finitely many $A_i$.

Equivalently, $p \in \limsup A_i$ if and only if $p = \lim p_{j_i}$ for a sequence of points $p_{j_i} \in A_{j_i}$ with $j_1 < j_2 < \cdots$, and $p \in \liminf A_i$ if and only if $p = \lim p_i$ for a sequence of points $p_i \in A_i$.
The limits superior and inferior are always closed. In fact, $\overline{\limsup A_i} = \limsup A_i = \limsup \overline{A_i}$, and $\overline{\liminf A_i} = \liminf A_i = \liminf \overline{A_i}$, where bars denote closure in $X$.

The sequence $(A_i)$ is \emph{Kuratowski convergent} if $\limsup A_i = \liminf A_i$. The \emph{Kuratowski limit}, or simply \emph{limit}, is $\lim_K A_i = \limsup A_i = \liminf A_i$.
Note that we allow $\lim_K A_i = \emptyset$.

\begin{lemma}[\cite{HockingYoung}, Lemma~2-101]\label{lem:connectedlimit}
	If $X$ is compact and Hausdorff, and $(A_i)$ is a sequence of connected subsets with $\liminf A_i \neq \emptyset$, then $\limsup A_i$ is connected.
\end{lemma}

\begin{lemma}[\cite{HockingYoung}, Theorem~2-102]\label{lem:subseqhauscon}
	If $X$ is compact and metrizable, then any sequence of subsets has a Kuratowski convergent subsequence.
\end{lemma}

\subsubsection{Hausdorff convergence}\label{subsection:Hausdorffmetric}

    Let $(X, d)$ be a metric space. Given subsets $A, B \subset X$, we define the \emph{Hausdorff distance}
\[ 
d_{H}(A, B) = \inf \{\epsilon > 0 \mid A \subset N_\epsilon(B), B \subset N_\epsilon(A) \}, 
\]
where $N_\epsilon(C)$ denotes the $\epsilon$-neighborhood of a subset $C \subset X$.

This number is not always finite. However, if $X$ is compact, then $d_H$ defines a complete metric on the space of its compact subsets. Whether or not $X$ is compact, we write $\lim_H A_n=A$ to mean that $A$ is closed and $\lim_{n\to \infty} d_H(A_n, A)=0$.

We will later use the following basic facts about Hausdorff and Kuratowski convergence.

\begin{lemma}\label{lemma:compactequivKuratowskiHausdorff}
    In compact metric spaces, Hausdorff and Kuratowski convergence are equivalent. That is, if $X$ is a compact metric space, and $(A_i)$ is a sequence of compact subsets, then $\lim_K A_n=A$ if and only if $\lim_H A_n=A$.
\end{lemma}

\begin{lemma}\label{lemma:continuoushausdorfflim}
    Uniformly continuous maps respect Hausdorff convergence. That is, if $f\colon X\to Y$ is a uniformly continuous map between metric spaces, and $\lim_H A_i=A$, then $\lim_H f(A_i)=f(A)$.
\end{lemma}

\subsection{Quotients and (upper semicontinuous) decompositions}\label{sec:Decompositions}

Let us recall some material on quotients and decompositions of topological spaces.

A \emph{partition} of a space $X$ is a collection of pairwise disjoint, nonempty subsets whose union is $X$.
Given a partition $\mathcal{P}$ we denote by $\mathcal{P}(x)$ the unique element of $\mathcal{P}$ that contains $x \in X$. This may be thought of as a function
\[ \mathcal{P}(\cdot)\colon X \to \mathcal{P}, \]
called the \emph{quotient function}. Given a subset $\cA$ of $\mathcal{P}$, we let $|\cA|$ denote the corresponding subset of $X$, i.e.
\[ |\cA| = \bigcup_{A \in \cA} A = \mathcal{P}^{-1}(A). \]

A subset $A \subset X$ is said to be \emph{$\mathcal{P}$-saturated} or \emph{saturated by $\mathcal{P}$} if it is a union of elements of $\mathcal{P}$. Equivalently, $A \subset X$ is saturated if and only if $A = |\mathcal{P}(A)|$.

The set $\mathcal{P}$ has a natural topology, the \emph{quotient topology}, with which $\mathcal{P}$ is called the \emph{quotient space}. This is defined by declaring $\cA \subset \mathcal{P}$ to be open if and only if $|\cA|$ is open. With this topology the quotient function $\mathcal{P}(\cdot)$ is a map, i.e. it is continuous. It has the universal property that any map $f\colon X\to Y$ that is constant on partition elements can be factored as $f=g\circ \mathcal P$ for a unique map $g$.

A partition of $X$ is equivalent to an equivalence relation on $X$. It is common to use the notation $\rquotient{X}{\cP}$ to refer to $\cP$ with the quotient topology, emphasizing the idea of taking $X$ and collapsing each element of $\cP$ to a point. 

A partition $\cD$ of a space $X$ whose elements are closed is called a \emph{decomposition}. This means that points in the corresponding quotient space are closed, i.e. the quotient space is $T^1$. The quotient topology associated to a decomposition is also called the \emph{decomposition topology}, and $\cD$ with this topology is called the \emph{decomposition space}.

In general, a surjection $f: X \to Y$ is said to be a \emph{quotient map} if for all $A \subset Y$, $A$ is open if and only if $f^{-1}(A)$ is open. In this situation the point preimages of $f$ partition $X$ and $Y$ is homeomorphic to the quotient of $X$ by this partition.
We have the following basic fact (see e.g. \cite{Munkres}): 

\begin{fact}\label{lem:quotientfact}
    Let $X,Y,Z$ be topological spaces. If $f\colon X\to Y$ and $g\colon X\to Z$ are quotient maps, and $h\colon Y\to Z$ is a bijection (not a priori continuous) such that $h\circ f=g$, then $h$ is a homeomorphism.    
\end{fact}

An equivalent definition of quotient map is the following: a surjection is a quotient map if and only if it maps saturated open sets to open sets, if and only if it maps saturated closed sets to closed sets. Hence any continuous surjection that is either open or closed is a quotient map. While a quotient map need neither be open nor closed, the closed case will be of central importance to us.

\begin{definition}
	A decomposition $\cD$ of a topological space $X$ is \emph{upper semicontinuous} if the quotient map $X\to \cD$ is closed.
\end{definition}

An equivalent formulation of upper semicontinuity is: for every open set $U\subset X$ that contains a decomposition element $K$, there is an open subset $V\subset U$ that contains $K$ such that every decomposition element that intersects $V$ is contained in $U$. This is in fact the more commonly given definition (e.g. \cite[\S 3-6]{HockingYoung}); the equivalence of the two conditions is a standard exercise.

\begin{example}\label{example:PointPreimages}
    A basic example of an upper semicontinuous decomposition to have in mind is the set of point preimages of a continuous map
    \[ f\colon X \to Y \]
    from a compact space $X$ to a Hausdorff space $Y$. That is,
    \[ \cD := \{f^{-1}(y) \mid y\in Y\}-\{\emptyset\}. \]
    If $f$ is onto, then $Y$ is homeomorphic to the decomposition space $\cD$ by \Cref{lem:quotientfact}.
\end{example}

In compact metric spaces, upper semicontinuity has a few convenient reformulations:

\begin{lemma}\label{theorem:USCConditions}
	Let $\cD$ be a decomposition of a compact metric space $X$. The following are equivalent.
	\begin{enumerate}[label=(\arabic*)]
		\item $\cD$ is upper semicontinuous.
		
		\item For any sequence of decomposition elements $A_i$ for which $\liminf A_i\ne\emptyset$, $\limsup A_i$ is contained in a decomposition element.
		
		\item \label{USClimitcontainment}The Hausdorff limit of any sequence of elements in $\cD$ is contained in an element of $\cD$.

        \item As a decomposition space, $\cD$ is Hausdorff.
	\end{enumerate}
\end{lemma}
\begin{proof}
    The equivalence of (1), (2), and (3) is \cite[Ch. IV, \S43, Th. IV.2]{Kuratowskiv2}. The implication (1)$\Rightarrow$(4) is \cite[Ch. IV, \S43, Th. III.1]{Kuratowskiv2}.

    Finally, suppose $C\subset X$ is closed, and hence is compact. The image $\cD(C)$ of $C$ under the quotient map is compact, so if $\cD$ is Hausdorff then $\cD(C)$ is closed. This shows (4)$\Rightarrow$(1), regardless of whether $X$ is a metric space.
\end{proof}

\subsubsection{Restriction and intersection}

\begin{definition}\label{def:restriction}
	Let $\cD$ be a decomposition of a space $X$. Given a subspace $Y \subset X$, we define $\cD \cap Y := \{A \cap Y \mid A \in \cD\}$.
	
	If some decomposition element $A \subset \cD$ does not intersect $Y$, then $\cD$ will contain $\emptyset$ and hence not technically be a decomposition. Thus we define $\cD \capdot Y := (\cD \cap Y) - \{\emptyset\}$, and call this the \emph{restriction} of $\cD$ to $Y$.
\end{definition}

On the other hand, let $\cD_Y=\{A\in \cD\mid A\cap Y\ne \emptyset\}$, and endow $\cD_Y$ with the  topology it inherits as a subspace of $\cD$. Note that there is a natural bijection $h\colon \cD_Y\to \cD \capdot Y.$

\begin{lemma}\label{lemma:restrictionUSC}
	Let $\cD$ be an upper semicontinuous decomposition of a space $X$, and let $Y \subset X$ be a closed subset. Then the bijection $h\colon \cD_Y\to\cD \capdot Y$ is a homeomorphism, and $\cD\capdot Y$ is upper semicontinuous. 
\end{lemma}
\begin{proof}
	 The inclusion $Y\hookrightarrow X$ is a closed map, as is $X\to \cD$ by upper semicontinuity. Composing the two maps and restricting to the range, we get a closed map $Y\to \cD_Y$, which is in particular a quotient map. We have a commutative triangle
     \begin{center}
     \begin{tikzcd}
    Y \arrow{rd}{}\arrow{d}& \\
    \cD_Y\arrow{r}{h}&\cD\capdot Y,
    \end{tikzcd}
    \end{center}
so $h$ is a homeomorphism by \Cref{lem:quotientfact}. Since $Y\to \cD_Y$ is closed, $Y\to \cD\capdot Y$ is closed. Hence $\cD\capdot Y$ is upper semicontinuous.
\end{proof}

The converse of the upper semicontinuity statement is not true:

\begin{remark}
    Using the notation of \Cref{lemma:restrictionUSC},
    it is possible that $\cD\capdot Y$ is upper semicontinuous while $\cD$ is not. For example, consider the decomposition $\cD$ of the square $X = [-1, 1] \times [-1, 1] \subset \bbR^2$ consisting of three kinds of elements:
	\begin{enumerate}
		\item the vertical full intervals $R_x := \{x\} \times [-1, 1]$, for $x \in [-1, 1]$, on the open right half;
		\item the horizontal half-intervals $L_y := [-1,0] \times \{y\}$, for $y \in (-1, 1)$, on the closed left half; and
		\item the union $L_\star := [-1, 0] \times \{-1, 1\}$ of the two remaining horizontal half-intervals.
	\end{enumerate}
    
	This is not upper semicontinuous, since the Hausdorff limit of $R_{1/i}$ as $i \to \infty$ is $\{0\} \times [-1, 1]$, which is not contained in a decomposition element. However, it is straightforward to see that the restriction $\Lambda:=\cD\capdot Y$ of $\cD$ to the boundary $Y$ of the square is upper semicontinuous. In particular, the Hausdorff limit of $R_{1/i} \cap Y$ as $i \to \infty$ is $\{0\} \times \{-1,1\}$, which is a subset of $L_\star \cap Y$. As decomposition spaces, then, $\Lambda$ is Hausdorff but $\cD$ is not.
\end{remark}

\begin{definition}
	Given two decompositions $\cD$ and $\cD'$ of a space $X$, we define the \emph{intersection} decomposition
	\[ 
    \cD \capdot \cD' := \{ K \cap K' \mid K \in \cD,  K' \in \cD' \} - \{\emptyset\}. \qedhere
    \]
\end{definition}

\begin{lemma}\label{lemma:intersectionUSC}
	If $\cD$ and $\cD'$ are upper semicontinuous decompositions of a compact metric space $X$, then $\cD \capdot \cD'$ is upper semicontinuous.
\end{lemma}

\begin{proof}
	Let $L_i$ be an arbitrary sequence of elements of $\cD \capdot \cD'$ such that $\liminf L_i$ intersects some element $L_\infty \in \cD \capdot \cD'$ nontrivially. By Lemma~\ref{theorem:USCConditions}, it suffices to show that $\limsup L_i \subset L_\infty$.
	
	For each $i = 1, 2, \dots, \infty$, let $K_i$ and $K'_i$ be the elements of $\cD$ and $\cD'$ for which $L_i = K_i \cap K'_i$. Then $\liminf K_i$ and $\liminf K'_i$ contain $\liminf L_i$, so they intersect $K_\infty$ and $K'_\infty$ nontrivially. By upper semicontinuity of $\cD$ and $\cD'$, we have $\limsup K_i \subset K_\infty$ and $\limsup K'_\infty \subset K'_\infty$. Hence $\limsup L_i \subset L_\infty$ as desired.	
\end{proof}

\subsubsection{Monotonization}\label{sec:monotonization}
A decomposition is called \emph{monotone} if each of its elements is connected.

\begin{definition}
	The \emph{monotonization} of a decomposition $\cD$ of a space $X$ is the decomposition $\Mon(\cD)$ whose elements are connected components of elements of $\cD$.
\end{definition}

\begin{lemma}\label{lemma:monotonizationUSC}
	If $\cD$ is an upper semicontinuous decomposition of a compact metric space $X$, 
    then $Mon(\cD)$ is upper semicontinuous.
\end{lemma}
\begin{proof}
 	Let $A'_i$ be a Hausdorff convergent sequence of elements of $\Mon(\cD)$. Then $\lim A'_i$ is contained in an element $A$ of $\cD$ by \Cref{theorem:USCConditions}. It is also connected by \Cref{lem:connectedlimit}, so it is contained in a component $A'$ of $A$, i.e. an element of $\Mon(\cD)$. Hence $\Mon(\cD)$ is upper semicontinuous by \Cref{theorem:USCConditions}.
\end{proof}

\subsubsection{Collapsing a decomposition}

In some nice cases, the quotient of a space by a decomposition is homeomorphic to the space itself.

Using a topological characterization of a compact interval, one can show:

\begin{theorem}\label{theorem:MooreForArcs}
	Let $\cD$ be a nontrivial decomposition of a compact interval $I$ into closed sub-intervals. Then $\rquotient{I}{\cD}$ is a compact interval.
\end{theorem}
Moving up a dimension, one has the following theorem of Moore.
\begin{theorem}[\cite{Moore}]\label{theorem:Moore}
	Let $\cD$ be a nontrivial upper semicontinuous decomposition of a topological plane $P$ such that each decomposition element is compact, connected, and nonseparating. Then $\rquotient{X}{\cD}$ is homeomorphic to the plane.
\end{theorem}
\begin{corollary}\label{corollary:Mooredisc}
	Let $\cD$ be a nontrivial upper semicontinuous decomposition of a closed disc $\bD=D\sqcup S$ such that each decomposition element is connected and nonseparating. Then $\rquotient{\bD}{\cD}$ is homeomorphic to a disc, and $S$ maps to the boundary of this disc under the quotient mapping. 
\end{corollary}

\begin{notation}
Here and throughout, ``a disc $\bD=D\sqcup S$" means that $\bD$ is a closed $2$-dimensional disc with interior $D$ and boundary circle $S$.
\end{notation}

\begin{proof}[Proof of \Cref{corollary:Mooredisc}]
	Simply think of $\bD$ as the unit disc in $\bbR^2$, and take the decomposition $\cD'$ consisting of elements of $\cD$, together with singleton elements for the points in $\bbR - \bD$. The unit circle maps to a simple closed curve in the decomposition space, which bounds a closed disc by the Schoenflies theorem.
\end{proof}

These results are false in higher dimensions. In particular, there are upper semicontinuous decompositions of $\bbR^3$ whose decomposition space is not a manifold (see \cite{Bing_Dogbone}).

\subsection{Topology of the disc and circle}

Let $\bD = D \sqcup S$ be a disc. 

We say a subset of $\bD$ is \emph{unbounded} if it intersects the boundary circle$S$.

When an orientation on $S$ is fixed, we let $(a, b)$, $[a, b]$, $[a, b)$, and $(a, b]$ 
denote the positively oriented open, closed, and half-open intervals between distinct points $a, b \in S$. We can extend this to $a = b$ by taking $(a, a) = S^1 - \{a\}$, and $[a,a] = [a, a) = (a, a] = S^1$.

Given a subset $A \subset \bD$, the connected components of $\bD - A$ will be called \emph{complementary regions} of $A$.

Given a subset $\lambda \subset S$, the connected components of $S - \lambda$ will be called \emph{complementary intervals} of $\lambda$. They are indeed intervals of the form $(a, b)$, $[a, b]$, $[a, b)$, or $(a, b]$ for $a, b \in S$ (or $S$ when $\lambda = \emptyset$.

Given a subset $A \subset \bD$, we will use the notation 
\[
\del A:=A\cap S.
\]

\subsubsection{Linking and separation in the circle and disc}\label{subsection:linking}
\begin{definition}
    Two unordered pairs of points $\{a, a'\}$ and $\{b, b'\}$ in a circle $S$ are said to be \emph{linked} if they are linked as $S^0$'s in $S^1$. That is if $\{a, a'\}$ separates $b$ from $b'$, or equivalently if $\{b, b'\}$ separates $a$ from $a'$. Otherwise they are \emph{unlinked}.

    Two subsets $A, B \subset S$ are \emph{linked} if there are pairs $\{a, a'\} \subset A$ and $\{b, b'\} \subset B$ that are linked. Otherwise they are \emph{unlinked}.
\end{definition}

For disjoint closed subsets one can quantify linking as follows:

\begin{lemma}\label{lemma:finitelinking}
    Let $A, B \subset S$ be disjoint, closed, and nonempty. Fix an orientation on $S$. Then there is the same, finite number $n \geq 1$ of open intervals of each of the following types:
    \begin{enumerate}[label=(\roman*)]
        \item complementary intervals of $A$ that intersect $B$ nontrivially,
        \item complementary intervals of $B$ that intersect $A$ nontrivially,
        \item complementary intervals of $A \cup B$ of the form $(a, b)$ for $a \in A$ and $b \in B$, and
        \item complementary intervals of $A \cup B$ of the form $(b, a)$ for $b \in B$ and $a \in A$.
    \end{enumerate}

    Disjoint closed subsets $A, B \subset S$ will be said to be \emph{$n$-linked} where $n$ is this number.
\end{lemma}
\begin{proof}
    Since $A$ and $B$ are nonempty, there is at least one interval of type (i).
    
    There can only be finitely many intervals of type (i). Indeed, if there were infinitely many distinct such intervals $(a_i, a'_i)$ then their diameters must tend to $0$ so a subsequence would converge to a single point $s$. Since the endpoints are in $A$, and each contains a point in $B$, this would mean $s \in A \cap B$, contradicting the assumption that $A$ and $B$ are disjoint.

    Each type (i) interval has an initial segment that is a type (iii) interval and a terminal segment that is a type (iv) interval. This determines bijections between type (i), (iii) and (iv) intervals. Similarly, each type (ii) interval has an initial segment that is a type (iv) interval and a terminal segment that is a type (iii) interval.
\end{proof}

Note that disjoint closed subsets are unlinked iff they are $1$-linked and linked iff they are $n$-linked for $n \geq 2$.

Linking in the circle can force intersection in the disc:

\begin{lemma}\label{lemma:iflinkthenintersect}
    Let $A, B$ be closed, connected subsets of a disc $\bD = D \sqcup S$. If $\del A$ and $\del B$ are linked in $S$ then $A$ and $B$ intersect.

    In particular, if $\del A$ and $\del B$ are disjoint and linked then $A$ and $B$ intersect in the interior $D$.
\end{lemma}

This follows immediately from:
\begin{lemma}
    Let $A$ be a closed, connected, unbounded subset of a disc $\bD = D \sqcup S$. Then for each complementary region $U$ of $A$, $\partial U$ is either $\emptyset$ or a complementary interval of $\del A$.
\end{lemma}
\begin{proof}
    Let $U$ be a complementary region of $A$ with $\partial U \neq \emptyset$. Then $\partial U$ must be a union of complementary intervals of $\partial A$. If $U$ contains two distinct complementary intervals $I, J$ of $\partial A$, then it contains a path from $I$ to $J$ since $\bbD$ is locally path connected. This path separates the endpoints of $I$, hence separates $A$, contradicting that $A$ is connected. Hence $U$ must be a single complementary interval of $\del A$. (This is the same argument as \cite[Lem. 6.2]{Frankel_closing}.)
\end{proof}

This also implies:

\begin{corollary}\label{corollary:SeparatingBasedOnEnds}
    Let $A, B, C$ be closed, connected, unbounded, and pairwise disjoint subset of a disc $\bD = D \sqcup S$. Then $B$ separates $A$ from $C$ in $\bD$ if and only if $\partial B$ separates $\partial A$ from $\partial C$ in $S$.
\end{corollary}

The following is \cite[Lemma 6.9]{Frankel_closing}, but repackaged to be more general.

\begin{lemma}\label{lemma:RegionBetween}
    Let $A, B$ be closed, connected, and disjoint subsets of a disc $\bD = D \sqcup S$. Then there is a unique complementary region $U$ of $A \cup B$ that accumulates on both $A$ and $B$, called the \emph{region between $A$ and $B$}. If $A$ and $B$ are unbounded, then $\partial U = (a, b) \sqcup (b', a')$ for $a, a' \in A$ and $b, b' \in B$.
\end{lemma}
 
\begin{proof}
    Let $U(A,B)$ be the complementary region of $A$ containing $B$, and let $U(B,A)$ be the complementary region of $B$ containing $A$. Let
    \[U=U(A,B)\cap U(B,A).\]
    Any complementary region of $A\cup B$ that accumulates on both $A$ and $B$ must be contained in $U$. Hence for the first claim it suffices to show that $U$ is connected.
    
    Using superscript $c$ to denote complement, we claim that $U(A,B)^c\cup U(B,A)^c$ is nonseparating. This holds because $\bD$ has the Phragmen-Brouwer property (i.e. the union of any two disjoint, nonseparating sets is nonseparating, see \cite[\S II.4]{Wilder}). Since $U=\bD-(U(A,B)^c\cup U(B,A)^c)$, $U$ is connected. (in fact, since $U$ is open and $\bbD$ is locally path connected, $U$ is path connected).

    Now assume $A$ and $B$ are unbounded. We have that $U(A,B) \cap S=(a, a')$ and $U(B,A)\cap S=(b,b')$ for $a,a'\in A$ and $b,b'\in S$ where $\{a,a'\}$ and $\{b,b'\}$ are unlinked by \Cref{lemma:iflinkthenintersect}. Hence 
    \[
    U\cap S= U(A,B)\cap U(B,A)\cap S=(a,b)\sqcup (b', a'). \qedhere
    \] 
\end{proof}

\part{Decompositions of the circle and disc}\label{part:circledisc}

A closed collection of pairwise disjoint geodesics in the hyperbolic plane, called a \emph{geodesic lamination}, corresponds naturally to a closed collection of pairwise unlinked $2$-point subsets of the circle at infinity. This correspondence goes in both directions, earning the latter the title of \emph{lamination} of the circle. This simple idea appears throughout $2$- and $3$-dimensional geometry, topology, and dynamics, where $\pi_1$-invariant laminations of a circle are used to produce laminations of manifolds.

With additional conditions, laminations of a circle may also be used to represent certain singular foliations of the plane, obtained by ``blowing down'' complementary regions of the corresponding geodesic laminations. They can be used, for example, to construct stable and unstable foliations for the pseudo-Anosov representative of an aperiodic irreducible homotopy class of homeomorphisms of a surface.

In this self-contained first part of the paper we will build a theory of ``especial pairs.'' These are pairs of decompositions of the circle designed to more generally and directly represent transverse pairs of singular foliations of the plane. We will apply this in \Cref{part:flowstraightening} to transform the positive and negative leaf decompositions of a quasigeodesic flow into a transverse pair of singular foliations.

The reader who is less familiar with quasigeodesic and pseudo-Anosov flows may wish to read \Cref{sec:flowbackground} now as motivation for \Cref{part:circledisc}, but this is not logically necessary.

\section{Definitions and results}\label{section:part1definitions}

\subsection{Definitions}
The fundamental objects of \Cref{part:circledisc} are ``especial pairs'' on the circle, which are designed to represent the data at infinity of a transverse pair of essential singular foliations of the plane.

\begin{definition}\label{definition:spS}
	A decomposition $\Lambda$ of a circle $S$ is \emph{special} if it is upper semicontinuous and has the following two properties.
	\begin{enumerate}[label=(\arabic*)]
		\item\label{it:spSUnlinked} \emph{unlinking:} distinct $\lambda, \lambda' \in \Lambda$ are unlinked, and 
		
		\item\label{it:spSNesting} \emph{nesting}: given any $\lambda \in \Lambda$, and a disjoint compact interval $I \subset S$, there is a decomposition element $\lambda' \in \Lambda$ that separates $I$ from $\lambda$.
	\end{enumerate}
	
	A pair of decompositions $\Lambda^\pm$ of $S$ has \emph{efficient intersection} if $\lambda^+ \cap \lambda^-$ is at most one point for any $\lambda^+ \in \Lambda^+$ and $\lambda^- \in \Lambda^-$.

    An \emph{especial pair} is pair of special decompositions of a circle having efficient intersection.
\end{definition}

\begin{remark}
Fixing an orientation on $S^1$, each complementary component of an element $\lambda \in \Lambda$ can be written as an oriented interval $(a, b)$ for $a, b \in \Lambda$. We may have $a = b$, which occurs exactly when $\lambda = \{a\}$ is a single point.

The nesting condition \ref{it:spSNesting} is equivalent to the following:
\begin{enumerate}[label=(\arabic*$'$), start=2]
    \item\label{it:spSNestingSequences} Given any $\lambda \in \Lambda$, and any complementary interval $(a, b)$ of $\lambda$, there is a sequence of elements $\lambda_i \in \Lambda$ contained in $(a, b)$ that converge to the ends of this interval. 
\end{enumerate}
By ``the $\lambda_i$ converge to the ends of this interval'' we mean that the $\lambda_i$ converge to the endpoints of the natural two-point compactification of $(a, b)$ obtained by adding a least upper bound and greatest lower bound. 

When $a \neq b$ this is equivalent to $\lim \lambda_i = \{a, b\}$. When $a=b$ it says that $\lim \lambda_i=a$ and that the limiting is from both sides of $a$. 
\end{remark}

Given an especial pair on a circle, we will construct an essentially canonical ``emuu pair'' of decompositions of the disc. Recall that when we refer to a disc $\bD=D\sqcup S$ we mean that $\bD$ is a closed disc with interior $D$ and boundary $S$.

\begin{definition} \label{def:muuemuu}
    A decomposition $\bSp$ of a disc $\bD = D \sqcup S$ is \emph{unbounded} if element $\bK \in \bSp$ intersects the boundary circle $S$ nontrivially, and \emph{monotone} if each of its elements is connected.

    A pair of decompositions $\bSp^\pm$ of $\bD$ is said to have \emph{efficient intersection} if $\bK^+ \cap \bK^-$ is at most one point for any $\bK^+ \in \bSp^+$ and $\bK^- \in \bSp^-$.

    We will often deal with decompositions of a disc that are monotone, upper semicontinuous, and unbounded, which we abbreviate \emph{muu}. A pair of muu decomositions with efficient intersection will be called an \emph{emuu pair}.
\end{definition}

Recall that we use the notation $\partial A := A \cap S$ for each $A \subset \bD = D \sqcup S$. Given an unbounded decomposition $\bSp$ of $\bD$, we define
\[ \partial \bSp := \{ \partial \bK \mid \bK \in \bSp \}. \]
This is a decomposition of the boundary circle $S$, where $\bK \mapsto \partial \bK$ is a bijection $\bSp \to \partial \bSp$. It is the same as the restriction decomposition $\partial \bSp = \bSp \cap S = \bSp \capdot S$ in \Cref{def:restriction}, where the second equality comes from the assumption that $\bSp$ is unbounded.

\subsection{Results of \Cref{part:circledisc}}

The primary goal of \Cref{part:circledisc} is to prove:

\begin{theorem}\label{theorem:effExt}
	Let $\Lambda^\pm$ be an especial pair on a circle $S$. Then there are decompositions $\bcF^\pm$ of a disc $\bD = D \sqcup S$ with the following properties:
	
	\begin{enumerate}[label=(\arabic*)]
        \item $\bcF^\pm$ is an emuu pair with $\partial \bcF^\pm = \Lambda^\pm$, and is the unique such pair up to a homeomorphism of $\bD$ that fixes $S$.

        \item Every group action $\Gamma \acts S$ that preserves $\Lambda^+$ and $\Lambda^-$ extends uniquely to an action $\Gamma \acts \bD$ that preserves $\bcF^+$ and $\bcF^-$.
        
		\item The decompositions $\bcF^\pm$ determine a transverse pair of essential singular foliations of $D$ given by
        \begin{align*}
            \cF^\pm &:= \Mon(\bcF \capdot D) \\
                    &= \{ \text{components of } \bK \cap D \neq\emptyset\mid \bK \in \bcF^\pm\}.
        \end{align*}
	\end{enumerate} 
\end{theorem}

Item (1) is \Cref{proposition:effExtExistence} and \Cref{theorem:emuuUniqueness}. Item (2) is \Cref{corollary:EfficientExtensionsAction}. Item (3) follows from \Cref{prop:emuutofoliations}, which says that in fact any emuu pair on $\bD$ induces a pair of transverse essential singular foliations in this way.

\begin{remark}
   See \Cref{sec:singularfoliations} for our definition of an essential singular foliation of the plane. It is more general than some in that it allows leaves to have multiple singularities.
\end{remark}

To conclude \Cref{part:circledisc}, in \Cref{sec:Straightening} we introduce the concept of ``straightening maps."
Starting with a pair $\bSp^\pm$ of muu decompositions of a disc $\bD$ whose elements intersect in a nice way (a so-called ``proper" muu pair), 
we identify a certain closed subset $\bD_\lk$ of $\bD$, and build a surjective ``straightening map"
\[\bs\colon\bD_\lk\to \bD\]
that carries $\bSp^\pm$ to an emuu pair $\bcF^\pm$ on $\bD$. We prove in particular that if $\Gamma$ is a group acting on $\bD$ preserving $\bSp^\pm$, then
\begin{itemize}
    \item $\bD_\lk$ is $\Gamma$-invariant,
    \item there is a corresponding action of $\Gamma$ on $\bD$ preserving $\bcF^\pm$, and
    \item the straightening map $\bs$ intertwines the actions.
\end{itemize}

In \Cref{part:flowstraightening}, a straightening map will be used to build a pseudo-Anosov flow from a quasigeodesic one. There, the proper muu pair is the positive and negative sprig decompositions of the compactified flowspace (see \Cref{sec:flowbackground} for definitions).

\section{From especial pairs to emuu pairs}\label{section:StraighteningFlowspace}

\subsection{Special and muu decompositions}\label{subsec:SpecialAndMUU}
We begin by showing that any especial decomposition of a circle extends to a muu decomposition of a disc, and any muu decomposition of a disc restricts to a special decomposition of the boundary circle.

Here, we will say that an unbounded decomposition $\bcD$ of a disc $\bD$ \emph{extends} or \emph{restricts to} a decomposition $\Lambda$ of $\partial \bD$ if $\partial \bcD = \Lambda$. 

\subsubsection{Extending special decompositions}
Given a decomposition $\Lambda$ of a circle $S$, identify $S$ with the boundary of the Euclidean unit disc $\bD = D \sqcup S$ and let $\Hull(\cdot)$ denote the Euclidean convex hull operation. We define
\[ \Hull(\Lambda) := \{ \Hull(\lambda) \mid \lambda \in \Lambda  \}.\]

\begin{proposition}\label{proposition:HullDecomposition}
    If $\Lambda$ is a special decomposition of $S$ then $\Hull(\Lambda)$ is a muu decomposition of $\bD$ that extends $\Lambda$.
\end{proposition}

The following is fundamental:

\begin{theorem}[Carath\'eodory's Theorem]
    In $\bbR^n$, each point $p \in \Hull(X)$ is contained in $\Hull(\{x_0, x_1, \cdots, x_n\})$ for $(n+1)$ points $x_0, \cdots, x_n \in X$. 
    
    Equivalently, if $\Delta^n$ is the standard $n$-simplex in $\bbR^{n+1}$, the natural map $X^{n+1}\times \Delta^{n}\to \Hull(X)$ taking $((x_0, \cdots, x_{n}),(t_0, \cdots, t_n)) \mapsto \sum{t_i x_i}$ is onto.
\end{theorem}

We will also use the following  observations.

\begin{lemma}\label{lemma:Hulls}
Let $S^1$ be the unit circle in $\bbR^2$, and $A\subset S^1$.
    \begin{enumerate}[label=(\arabic*)]
        \item\label{it:HullClosedIffClosed} If $A$ is closed (resp. open) in $S^1$ then $\Hull(A)$ is closed (resp. open).
        \item\label{it:HullIntersectS1} $\Hull(A) \cap S^1 = A$.
        \item\label{it:HullDisjThenIntIffLink} If $A, B \subset S^1$ are disjoint, then $\Hull(A)\cap \Hull(B)\ne \emptyset$ if and only if $A$ and $B$ are linked.
    \end{enumerate}
\end{lemma}

\begin{remark}
Regarding \ref{it:HullClosedIffClosed}, it is not true in general that the convex hull of a closed set is closed. However, the convex hull of a compact set is compact by Carath\'eodory's Theorem.
\end{remark}

These observations imply, in particular, that the elements of $\Hull(\Lambda)$ are closed, and that they are disjoint when $\Lambda$ has the unlinking property. To see that $\Hull(\Lambda)$ fills $\bD$ and is upper semicontinuous we will need the following lemmas.

\begin{lemma}\label{lemma:HullHausdorffDistances}
            Let $A, B \subset \bbR^n$. Then the Hausdorff distance $d_H$ satisfies 
            \[d_H(\Hull(A), \Hull(B) ) \leq d_H(A, B).\]
\end{lemma}
For the reader's convenience we include a nice proof of this due to A. Blumenthal on Mathematics StackExchange. 

\begin{proof}
    
    Let $\alpha\in \Hull(A)$, so by Carath\'eodory's Theorem $\alpha=\sum_{i=0}^{n+1} t_i a_i$ for points $a_0,\dots, a_{n+1}\in A$ and constants $t_1,\dots, t_{n+1}\ge0$ with $\sum_i t_i=1$. Fix $\lambda>1$. Then there are $b_0, \cdots, b_{n+1}\in B$ such that $d(a_i,b_i)<\lambda d_H(A,B)$ for all $i$. Letting $\beta=\sum_{i=0}^{n+1}t_ib_i$, we see $d(\alpha,\beta)<\lambda d_H(A,B)$. Now observe that the choice of $\lambda$ was arbitrary and the roles of $A$ and $B$ are symmetric.
\end{proof}

\begin{lemma}\label{lemma:SequencesOfHulls}
	Let $A_1, A_2, \dots$ be closed subsets of $S=\del \bD$.
	\begin{enumerate}
		\item $A_1, A_2,\dots$ converges if and only if $\Hull(A_1), \Hull(A_2), \dots$ converges.
		\item If these converge then $\lim \Hull(A_i) = \Hull(\lim A_i)$.
	\end{enumerate}
\end{lemma}
\begin{proof}
	Note that for a point $s \in S$, the open intervals in $S$ containing $s$ form a neighborhood base of $s$ in $S$, while the hulls of such intervals form a neighborhood base for $s$ in $\bD$. Moreover, a set $A \subset S$ intersects an interval $I \subset S$ if and only if $\Hull(A)$ intersects $\Hull(I)$ by \Cref{lemma:Hulls} (\ref{it:HullIntersectS1} \& \ref{it:HullDisjThenIntIffLink}). Consequently, a sequence of subsets $A_1, A_2, \dots \subset S$ respectively limits or accumulates on a point $s \in S$ if and only if the corresponding sequence of hulls $\Hull(A_1), \Hull(A_2), \dots \subset \bD$ limits or accumulates on $s$. It follows that $\Hull(A_i)$ has a Kuratowski limit if and only if $\lim A_i$ does. 

	Recall that in a compact metric space, Kuratowski convergence is equivalent to convergence with respect to the Hausdorff metric (\Cref{lemma:compactequivKuratowskiHausdorff}). If $A_i$ has a limit $\lim A_i$, then $d_H(A_i, \lim A_i) \to 0$. Then $d_H(\Hull(A_i), \Hull(\lim A_i)) \to 0$ by \Cref{lemma:HullHausdorffDistances}, and hence $\lim \Hull(A_i) = \Hull(\lim A_i)$.
\end{proof}

\begin{proof}[Proof of \Cref{proposition:HullDecomposition}]
	Fix a special decomposition $\Lambda$ of $S^1$. The elements of $\Hull(\Lambda)$ are closed and pairwise disjoint by \Cref{lemma:Hulls} \ref{it:HullClosedIffClosed} \& \ref{it:HullDisjThenIntIffLink}, so they form a decomposition of some subset $|\Hull(\Lambda)| \subset \bbD$. This decomposition is unbounded by construction, it is monotone because convex sets are connected, and $\partial\Hull(\Lambda) = \Lambda$ by \Cref{lemma:Hulls} \ref{it:HullIntersectS1}. It remains to show that this is upper semicontinuous and $|\Hull(\Lambda)| = \bbD$.

	\textbf{Claim:}	$|\Hull(\Lambda)|$ is closed. 

	Let $p$ be a point in the closure of $ |\Hull(\Lambda)|$. Then we can find a sequence of elements $\Hull(\lambda_i) \in \Hull(\Lambda)$ that accumulate on $p$. After passing to a subsequence we can assume that the $\Hull(\lambda_i)$ converge and $p \in \lim \Hull(\lambda_i)$. Then the $\lambda_i$ converge and $\lim \Hull(\lambda_i) = \Hull(\lim \lambda_i)$ by \Cref{lemma:SequencesOfHulls}. But $\Lambda$ is upper semicontinuous, so $\lim \lambda_i$ is contained in some $\lambda \in \Lambda$. Thus $p \in \Hull(\lambda) \in \Hull(\Lambda)$, so $p \in |\Hull(\Lambda)|$.

	\textbf{Claim:}	$|\Hull(\Lambda)| = \bbD$.

        Since the elements of $\Lambda$ cover $S$, it is clear that $S \in |\Hull(\Lambda)|$.	Suppose that some point $p \in \mr\bbD$ is not contained in $|\Hull(\Lambda)|$. Let $B$ be the largest metric open ball centered at $p$ contained in the complement of $|\Hull(\Lambda)|$. Then some element $\Hull(\lambda) \in \Hull(\Lambda)$ must intersect $\overline{B}$.
	
	Let $U$ be the component of $\bbD - \Hull(\lambda)$ that contains $B$. Then $U \cap S$ is an open interval $(a, b)$ with $a, b \in \lambda$ (note that we may have $a = b$). By the nesting property (\Cref{definition:spS}~\ref{it:spSNesting} and \ref{it:spSNestingSequences}) we can find elements $\lambda_i \in \Lambda$ that lie in $(a, b)$ and converge to the endpoints of this interval. Then $\Hull(\lambda_i)$ must eventually intersect $B$, contradicting the fact that $B$ is disjoint from $|\Hull(\Lambda)|$. Thus $|\Hull(\Lambda)| = \bbD$ as desired.

    \textbf{Claim:}	$\Hull(\Lambda)$ is upper semicontinuous.

	Let $\Hull(\lambda_i)$ be a convergent sequence of elements of $\Hull(\Lambda)$. Then the $\lambda_i$ converge and $\lim \Hull(\lambda_i) = \Hull(\lim \lambda_i)$ by \Cref{lemma:SequencesOfHulls}. But $\Lambda$ is upper semicontinuous, so $\lim \lambda_i \subset \lambda$ for some $\lambda \in \Lambda$, so $\lim \Hull(\lambda_i)$ is contained in $\Hull(\lambda) \in \Hull(\Lambda)$. By \Cref{theorem:USCConditions}, $\Hull(\Lambda)$ is upper semicontinuous.
\end{proof}

\subsubsection{Restricting muu decompositions}
As a counterpart to the extension result \Cref{proposition:HullDecomposition}, let us show that any muu decomposition of a disc restricts to a special decomposition of its boundary circle.

\begin{proposition}\label{prop:muuRestrictsToSpecial}
    If $\bSp$ is a muu decomposition of a disc $\bD = D \sqcup S$ then $\partial \bSp$ is a special decomposition of the boundary circle $S$.
\end{proposition}

We will use the following basic properties of muu decompositions.

\begin{lemma}\label{lem:FaceOfSaturatedContinuum}
    Let $\bSp$ be a muu decomposition of a disc $\bD$. Let $A \subset \bD$ be closed, connected, and $\bSp$-saturated. Then for each complementary component $U$ of $A$, $\overline{U} \cap A$ is contained in a single element of $\bSp$.
\end{lemma}
\begin{proof}
    The proof of \cite[Corollary 6.7]{Frankel_closing} works without modification. Although that article considers muu decompositions with the additional property that $\partial \bK$ is totally disconnected for each decomposition element $\bK$, the proof does not use this property.
\end{proof}

\begin{corollary}\label{cor:SeparatingSaturatedSets}
    Let $\bSp$ be a muu decomposition of a disc $\bD$. If $A, B \subset \bD$ are disjoint, closed, connected, $\bSp$-saturated sets then some $\bK \in \bSp$ separates $A$ from $B$.
\end{corollary}
\begin{proof}
    Let $U$ be the region between $A$ and $B$ (see \Cref{lemma:RegionBetween}). Then $\partial U = (a, b) \sqcup (b', a')$ with $a, a' \in A$ and $b, b' \in B$. Let $\gamma$ be an arc from a point in $(a, b)$ to a point in $(b', a')$ and let $C$ be the $\bSp$-saturation of $\gamma$. Then $C$ separates $A$ from $B$. Let $V$ be the complementary region of $C$ that contains $A$. Then $\overline{V} \cap C$ separates $A$ from $B$ and is contained in an element $\bK \in \bSp$ by \Cref{lem:FaceOfSaturatedContinuum}. See \Cref{fig:SeparatingSprigs}. 
    \end{proof}

    \begin{figure}[h]
		\centering
		\begin{overpic}[scale=0.8]{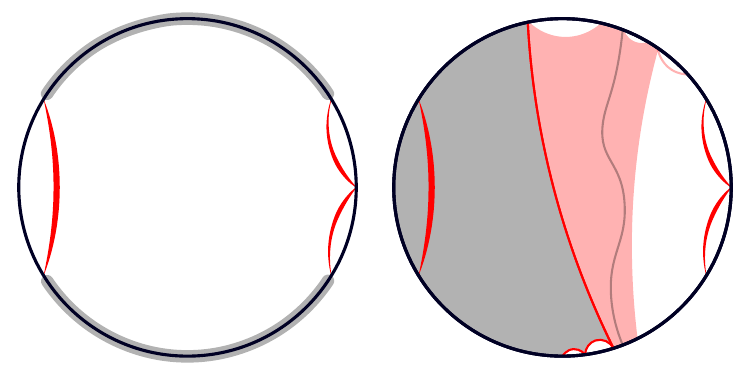}
			\tiny
			
			\put(9, 24){\textcolor{red}{$B$}}
			\put(43, 24){\textcolor{red}{$A$}}
			
			\put(45, 37){$a$}
            \put(3.5, 37){$b$}
            \put(3.5, 11){$b'$}
            \put(45, 11){$a'$}

            \put(64, 19){$V$}
            \put(71, 24){\textcolor{red}{$\bK$}}
            \put(76, 28){\textcolor{red}{$C$}}
            \put(81, 24){$\gamma$}
		\end{overpic}
		\caption{Separating saturated sets (\Cref{cor:SeparatingSaturatedSets}).} \label{fig:SeparatingSprigs}
    \end{figure}

\begin{proof}[Proof of \Cref{prop:muuRestrictsToSpecial}]
    Fix a muu decomposition $\bSp$ of $\bD$. Then $\del\bSp$ is upper semicontinuous by \Cref{lemma:restrictionUSC}.
    
    Distinct elements of $\partial \bSp$ are of the form $\partial \bK, \partial \bL$ for distinct elements $\bK, \bL \in \bSp$. Such $\bK$ and $\bL$ are disjoint, closed, and connected, so $\partial \bK$ and $\partial \bL$ are unlinked by \Cref{lemma:iflinkthenintersect}. Thus $\partial \bSp$ has the unlinking property (\Cref{definition:spS} \ref{it:spSUnlinked}).

    To see that $\partial\bSp$ satisfies the nesting property (\Cref{definition:spS} \ref{it:spSNesting}), let $I \subset S$ be a compact interval disjoint from an element $\partial \bK \in \partial \bSp$. Let $A$ be the $\bSp$-saturation of $I$. Since $\bSp$ is monotone and upper semicontinuous it follows that $A$ is connected and closed. Then \Cref{cor:SeparatingSaturatedSets} provides an element $\bL \in \bSp$ that separates $\bK$ from $A$, and hence $\partial \bL$ separates $\partial \bK$ from $I$ as desired (see \Cref{corollary:SeparatingBasedOnEnds}).
\end{proof}

\subsection{Separation intervals}

As an application of what we have proved so far, let us show that the decomposition space of a muu decomposition contains an interval between any two elements. This will be useful when we show that emuu decompositions restrict to singular foliations in \Cref{sec:emuupairsgivefoliations}. First, an observation:

\begin{lemma}\label{lemma:circlediscrestrictionhomeo}
    Let $\bSp$ be a muu decomposition of a disc $\bD = D \sqcup S$. Then
    \begin{align*}
        \partial_{\bSp}\colon \bSp &\to \partial \bSp\\
        \partial_{\bSp}(\bK) &= \partial \bK
    \end{align*}
    is a homeomorphism between the decomposition spaces.
\end{lemma}
\begin{proof}
    Every element of $\bSp$ intersects $S$, so we may simply apply \Cref{lemma:restrictionUSC}.
\end{proof}

In particular, this means that any two muu decompositions of a disc $\bD = D \sqcup S$ that restrict to the same special decomposition of $S$ have the same decomposition space. Since any special decomposition can be filled in with hulls, one can therefore treat any muu decomposition as a hull decomposition, which makes some arguments simpler and more geometric. In particular:

\begin{definition}\label{definition:separationinterval}
    Let $\bSp$ be a muu decomposition of a disc $\bD = D \sqcup S$. Given distinct $\bK_0, \bK_1 \in \bSp$, we define the corresponding \emph{open separation interval}
    \[ (\bK_0, \bK_1) = \{\bK \in \bSp \mid \bK \text{ separates } \bK_0 \text{ from }\bK_1 \}  \] 
    and \emph{closed separation interval}
    \[ [\bK_0, \bK_1] = \{\bK_0\} \cup (\bK_0, \bK_1) \cup \{\bK_1\}. \] 

    We define separation intervals for special decompositions of a circle $S$ in exactly the same way.
\end{definition}

\begin{lemma}\label{lemma:separationintervals}
    Let $\bSp$ be a muu decomposition of a disc $\bD$. Then every closed separation interval $[\bK_0, \bK_1]$ is homeomorphic to a compact interval as a subspace of the decomposition space, with endpoints corresponding to $\bK_0$ and $\bK_1$.
\end{lemma}
\begin{proof}
    Given a muu decomposition $\bSp$ of $\bD = D \sqcup S$, let $\Lambda = \partial \bSp$ be the corresponding special decomposition of $S$. The bijection $\bSp \to \Lambda$ defined by $\bK \mapsto \partial \bK$ is a homeomorpism by \Cref{lemma:circlediscrestrictionhomeo}. By \Cref{corollary:SeparatingBasedOnEnds} this homeomorphism takes each separation interval $[\bK_0, \bK_1]$ in $\bSp$ to the separation interval $[ \partial\bK_0, \partial \bK_1]$ in $\Lambda$. It therefore suffices to prove the lemma for separation intervals in $\Lambda$ or for any other muu decomposition that restricts to $\Lambda$. We will prove it for the hull decomposition $\bcH := \Hull(\Lambda)$. 

    Fix distinct $\bK_0, \bK_1 \in \bcH$ and choose a straight line segment $\gamma = \Hull(\{p_0, p_1\})$ from a point $p_0 \in \bK_0$ to a point $p_1 \in \bK_1$.

    Let $\bcH_\gamma$ be the set of elements of $\bcH$ that intersect $\gamma$, endowed with the topology as a subset of the decomposition space $\bcH$. Let $\cI = \{ \bK \cap \gamma \mid \bK \in \bcH_\gamma \}$, endowed with the topology as a decomposition of $\gamma$. By \Cref{lemma:restrictionUSC} the bijection $h: \bcH_\gamma \to \cI$ defined by $h(\bK) = \bK \cap \gamma$ is a homeomorphism. Hence it suffices to show that $\bcH_\gamma = [K_0, K_1]$ and that $\cI$ is homeomorphic to an interval, with endpoints corresponding to $\gamma\cap \bK_0$ and $\gamma \cap \bK_1$.
    
    Since $H_0 \cup \gamma \cup H_1$ is connected, any element of $\bcH$ that separates $\bK_0$ from $\bK_1$ must intersect $\gamma$, and $\bK_0$ and $\bK_1$ intersect $\gamma$ by construction. Thus $[\bK_0, \bK_1] \subset \bcH_\gamma$. On the other hand, if $\bK \in \bcH$ does not separate $\bK_0$ from $\bK_1$ then $\partial \bK$ does not separate $\partial \bK_0$ from $\partial \bK_1$ (\Cref{corollary:SeparatingBasedOnEnds}), hence $\partial \bK$ is unlinked from $\partial \bK_0 \cup \partial \bK_1$. Then $\Hull(\partial \bK) = \bK$ is disjoint from $\Hull(\partial \bK_0 \cup \partial \bK_1)$ (\Cref{lemma:Hulls} \ref{it:HullDisjThenIntIffLink}), and therefore also from $\gamma$, which is a subset of $\Hull(\partial \bK_0\cup \partial \bK_1)$. Thus $\bcH_\gamma = [\bK_0, \bK_1]$ as desired.

    The elements of $\cI$ are intersections of convex compact sets with the convex compact set $\gamma$, and are therefore convex compact. That is, they are compact subintervals of $\gamma$. By \Cref{theorem:MooreForArcs}, the decomposition space $\cI$ is a compact interval. The endpoints of this interval are clearly $\gamma\cap\bK_0$ and $\gamma\cap \bK_1$ (they are the only nonseparating elements).
\end{proof}

\subsection{Especial and emuu pairs}
Now that we can move back and forth between muu decompositions of a disc and special decompositions of a circle, we would like to do the same with emuu pairs and especial pairs. 

One direction is easy. 

\begin{corollary}\label{cor:emuuRestrictsToEspecial}
    If $\bSp^\pm$ is an emuu pair on a disc $\bD = D \sqcup S$ then $\partial \bSp^\pm$ is an especial pair on the boundary circle $S$.
\end{corollary}
\begin{proof}
    Each of $\partial \bSp^+$ and $\partial \bSp^-$ are special by \Cref{prop:muuRestrictsToSpecial}, and the efficient intersection property for $\bSp^\pm$ implies the efficient intersection property for $\partial \bSp^\pm$.
\end{proof}

It remains to show that any especial pair on a circle can be extended to an emuu pair on a disc. We give the construction below, followed by a proposition verifying that it satisfies the required properties.

\begin{construction}[Extending an especial pair to an emuu pair]\label{construction:emuuextension}
    Let $\Lambda^\pm$ be an especial pair on a circle $S$.
    Identify $S$ with the boundary of the Euclidean unit disc $\bD$, and let
    \[ \bcH^\pm := \Hull(\Lambda^\pm) \]
    be the hull decompositions constructed above. Consider the collection
    \[ \bcH^\cap := \{ K^+ \cap K^- \mid K^\pm \in \bcH^\pm, K^+ \cap K^- \neq \emptyset\} \]
    of all nonempty pairwise intersections of elements of $\bcH^+$ and $\bcH^-$.

    Observe:
    \begin{itemize}	
        \item $\bcH^\cap$ is a monotone decomposition of $\bD$ into convex compact subsets: The elements are intersections of convex compact sets and are therefore convex (hence connected) and compact. They fill $\bD$ since the elements of $\bcH^+$ and $\bcH^-$ fill $\bD$. 
        \item The elements of $\bcH^\cap$ are nonseparating: Since $\Lambda^\pm$ intersect efficiently, each element of $\bcH^\cap$ intersects the boundary circle in at most one point. Convex compact subsets of the disc that intersect the boundary in at most one point are nonseparating. 
        \item $\bcH^\cap$ is upper semicontinuous by \Cref{lemma:intersectionUSC}.
    \end{itemize}
    
	Thus, Moore's Theorem (\Cref{corollary:Mooredisc}) implies that 
    \[ \bQ := \rquotient{\bD}{\bcH^\cap} \]
	is homeomorphic to a closed disc. Letting
    \[ \tau: \bD \to \bQ \] 
    denote the quotient map, \Cref{corollary:Mooredisc} also gives that $\tau(S)$ is the boundary of $\bQ$.
    Since $\Lambda^+$ and $\Lambda^-$ intersect efficiently, $\tau$ restricts to a injection (necessarily closed) $S\to \bQ$, which is in particular a homeomorphism onto its image. 
    We will use $\tau$ to identify $S$ with $\partial \bQ$, and write $\bQ = Q \sqcup S$.

	We claim that the images of elements of $\bcH^+$ and $\bcH^-$ under $\tau$,
	\[ \bcF^\pm = \{ \tau(K) \mid K \in \bcH^\pm \} \]
    form an emuu pair on $\bQ$. This is verified in \Cref{proposition:effExtExistence}.
\end{construction}

\begin{proposition}\label{proposition:effExtExistence}
	Let $\Lambda^\pm$ be an especial pair on a circle $S$. Then the pair $\bcF^\pm$ constructed in \Cref{construction:emuuextension} form an emuu pair $\bcF^\pm$ on the disc $\bQ = Q \sqcup S$ such that $\partial \bcF^\pm = \Lambda^\pm$.
\end{proposition}
\begin{proof}
    First, note that each element of $\bcF^\pm$ is the continuous image of a compact connected set and is therefore compact and connected. Distinct elements of $\bcF^+$ or $\bcF^-$ are disjoint since they are images of disjoint $\bcH^\cap$-saturated sets, and $\bcF^+$ and $\bcF^-$ fill $\bQ$ because $\bcH^+$ and $\bcH^-$ fill $\bD$. 
    Efficient intersection of $\bcF^\pm$ holds by construction.
    Since $\tau$ identifies $S$ with the boundary of $\bQ$ it follows that $\bcF^\pm$ are unbounded and $\partial \bcF^\pm = \Lambda^\pm$. 
	
	It remains to show that $\bcF^\pm$ are upper semicontinuous. 
    We have commutative squares
    \begin{center}
    \begin{tikzcd}
        \bD \arrow{r}{\tau}\arrow{d}{} & \bQ\arrow{d}{}\\
        \bcH^\pm\arrow{r}{}&\bcF^\pm
    \end{tikzcd}
    \end{center}
    where the vertical arrows are the quotient maps and the bottom maps are the corresponding bijections, which are homeomorphisms by \Cref{lem:quotientfact}. 
    Surjectivity of $\tau$ and closedness of $\bD\to \bcH^\pm$ now imply that $\bQ\to\bcF^\pm$ are closed maps.
\end{proof}

    \Cref{fig:collapsingexamples} illustrates some examples of emuu pairs that may come from \Cref{construction:emuuextension}.
    Observe that while the quotient map $\tau: \bbD \to \bQ$ takes $S^1$ homeomorphically to the boundary of $\bQ$, the preimage $\tau^{-1}(\partial \bQ)$ can be larger than $S^1$. This can happen if an element of $\Lambda^\pm$ contains an interval, or if some element of $\Lambda^+$ both intersects and links with some element of $\Lambda^-$. See \Cref{fig:collapsingexamples}(c), (d), and (e).

    Note that the decompositions in the figure appear to be foliations on the interior of the disc, and they are. We will show in  \Cref{sec:emuupairsgivefoliations} that every emuu pair in the disc restricts to a transverse pair of singular foliations of the interior. The singular leaves of these foliations may have more than one singular point, as illustrated in \Cref{fig:collapsingexamples}(b). Leaves of the same or different sign may meet at infinity as illustrated in \Cref{fig:collapsingexamples}(c). There may also be ``infinite product regions" (unbounded strips with product bifoliations) as in \Cref{fig:collapsingexamples}(d), which come from elements of $\Lambda^+$ or $\Lambda^-$ that contain intervals. Such an infinite product region might contain leaves limiting on nonseparated leaves as in \Cref{fig:collapsingexamples}(e).

\begin{figure}
    \centering
   \includegraphics{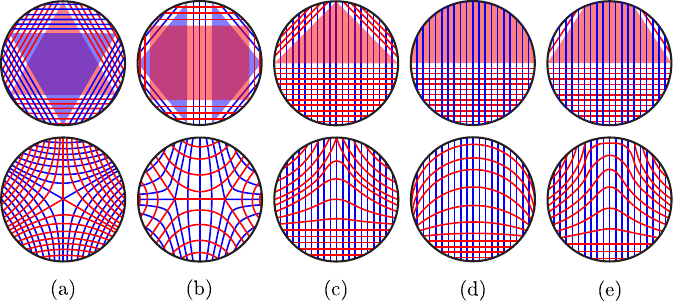}

    \caption{Top row: some decompositions of the disc by convex hulls of decompositions of the circle. Bottom row: the results of collapsing the intersections of the hull decompositions.}
    \label{fig:collapsingexamples}
\end{figure}

\section{Comparing emuu extensions}\label{sec:ComparingEmuu}
We we will now build some tools to compare different emuu decompositions of the disc that extend the same especial decomposition of a circle. We will see that an emuu pair $\bcF^\pm$ extending a given especial pair $\Lambda^\pm$ is unique (\Cref{theorem:emuuUniqueness}) up to a homeomorphism of the ambient disc that fixes its boundary, and any group action on the circle that preserves $\Lambda^\pm$ extends uniquely to an action on the disc that preserves $\bcF^\pm$ (\Cref{corollary:EfficientExtensionsAction}).

The idea is simple: consider emuu pairs $\bcF^\pm_1, \bcF^\pm_2$ living on discs $\bD_1 = D_1 \sqcup S$, $\bD_2 = D_2 \sqcup S$, that restrict to the same especial pair $\partial\bcF^\pm_1 = \partial\bcF^\pm_2$ on $S$. Construct a function $\bD_1 \to \bD_2$ by sending each $p \in \bD_1$ to a point $q \in \bD_2$ such that $\partial\bcF^+_1(p) = \partial\bcF^+_2(q)$ and $\partial\bcF^-_1(p) = \partial\bcF^-_2(q)$, and show that this is a well-defined homeomorphism. 

We will take a more indirect approach, and show that for an especial pair $\Lambda^\pm$ there is a natural topological disc $\bZ$ that lives in the product of the decomposition spaces $\Lambda^+ \times \Lambda^-$ called the ``especial disc.'' The way this disc lies in $\Lambda^+ \times \Lambda^-$ endows it naturally with a pair of decompositions $\bcZ^\pm$.

We will show that any muu pair $\bSp^\pm$ on $\bD$ that extends $\Lambda^\pm$ determines a natural map $\delta_{\bSp^\pm}\colon \bD \to \Lambda^+ \times \Lambda^-$ called the ``double boundary map.'' In the special case of an emuu pair $\bcF^\pm$, the double boundary map $\delta_{\bcF^\pm}$ is a homeomorphism onto $\bZ$ that identifies $\bcF^\pm$ with $\bcZ^\pm$.

The double boundary map will appear again in \Cref{sec:Straightening} where we use it to compare a more general (non-e) muu pair that extends a given especial pair to an emuu pair.

\subsection{The double boundary map and especial disc}\label{sec:doubleboundaryespecialdisc}

\begin{definition}\label{definition:doubleboundarymap}
    Given a unbounded pair $\bSp^\pm$ on a disc $\bD$, we define the corresponding \emph{double boundary function}
    \begin{align*}
        \delta\colon \bD &\to \partial \bSp^+ \times \partial \bSp^-\\
       p &\mapsto (\partial \bSp^+(p), \partial \bSp^-(p)).\qedhere
    \end{align*}
\end{definition}

\begin{lemma}\label{lemma:doubleboundarycontinuous}\label{cor:emuudoubleboundaryhomeo}
    If $\bSp^\pm$ is a muu pair then the corresponding double boundary function $\delta$ is continuous.

    If $\bSp^\pm$ is an emuu pair, then $\delta$ is a homeomorphism onto its image.
\end{lemma}
\begin{proof}
    Given a muu pair $\bSp^\pm$, let
    \begin{align*}
        \bSp^\pm(\cdot)\colon \bD &\to \bSp^\pm\\
        p &\mapsto\bSp^\pm(p)
    \end{align*}
    be the quotient maps, which are continuous. Let
    \begin{align*}
        \partial^\pm(\cdot)\colon \bSp^\pm&\to \partial\bSp^\pm\\
        \bK &\mapsto\partial\bK
    \end{align*}
    be the natural bijections from $\bSp^+$ to $\partial \bSp^+$ and from $\bSp^-$ to $\partial \bSp^-$, which are homeomorphisms by \Cref{lemma:restrictionUSC}. Then $\delta$ is continuous because $\delta = (\partial^+ \circ \bSp^+) \times (\partial^- \circ \bSp^-)$. (This uses only unboundedness and upper semicontinuity).

    If $\bSp^\pm$ is emuu, then the efficient intersection property implies $\delta$ is injective. The image of $\delta$ is Hausdorff because it lives in $\partial \bSp^+ \times \partial \bSp^-$, which is Hausdorff by upper semicontinuity (\Cref{theorem:USCConditions}). Any continuous bijection from a compact space to a Hausdorff space is a homeomorphism, so $\delta$ is a homeomorphism onto its image.
\end{proof}

Fix an especial pair $\Lambda^\pm$ on a circle $S$, and let $\bcF^\pm$ be the emuu pair on $\bQ = Q \sqcup S$ from \Cref{construction:emuuextension}. By \Cref{cor:emuudoubleboundaryhomeo}, the corresponding double boundary map can be thought of as a homeomorphism 
\[ \delta\colon \bQ \to \bZ \]
where $\bZ := \delta(\bQ) \subset \Lambda^+ \times \Lambda^-$. Let us set the notation $Z := \delta(Q)$ and $\partial Z:= \delta(S)$, and note that $\bZ = Z \sqcup\partial Z$. 

In the construction of $\bQ$ (\Cref{construction:emuuextension}), points in $S = \partial Q$ correspond exactly to pairs of elements of $\Lambda^+$ and $\Lambda^-$ that intersect nontrivially. Hence $\partial Z$ can be characterized directly as
\[ \partial Z = \{(\lambda^+, \lambda^-) \mid \lambda^+ \text{ and }\lambda^- \text{ intersect}\}.\]
The points in $Q$ arise precisely from pairs of elements of $\Lambda^+$ and $\Lambda^-$ that are disjoint and linked, so
\[ Z = \{(\lambda^+, \lambda^-) \mid \lambda^+ \text{ and }\lambda^- \text{ are disjoint and linked}\},\]
and hence
\[ \bZ = Z \sqcup \partial Z = \{(\lambda^+, \lambda^-) \mid \lambda^+ \text{ and }\lambda^- \text{ intersect or link}\}. \]
The homeomorphism $\delta: \bQ \to \mathbf{Z}$ takes the decompositions $\bcF^\pm$ to decompositions 
\[ \bcZ^\pm :=\delta(\bcF^\pm) = \{ \delta(\bK) \mid \bK \in \bcF^\pm \}, \]
which may be also be characterized directly. Let $\pi_+\colon \bZ \to \Lambda^+$ and $\pi_-\colon \bZ \to \Lambda^-$ be the compositions of inclusion $\bZ \hookrightarrow \Lambda^+ \times \Lambda^+$ with the projections of $\Lambda^+ \times \Lambda^-$ onto its factors. Then
\begin{align*}
    \bcZ^+ = \delta(\bcF^+) &= \{ (\pi_+)^{-1}(\lambda) \mid \lambda \in \Lambda^+ \}\\
    \bcZ^- = \delta(\bcF^-) &= \{ (\pi_-)^{-1}(\lambda) \mid \lambda \in \Lambda^- \}.
\end{align*}
We record these characterizations, which depend only on $\Lambda^\pm$, as definitions:
\begin{definition}\label{def:especialdecomps}
    Let $\Lambda^\pm$ be an especial pair on $S^1$. We define the following subsets of $\Lambda^+ \times \Lambda^-$:
    \begin{align*}
        Z &:= \{(\lambda^+, \lambda^-) \mid \lambda^+ \text{ and }\lambda^- \text{ are disjoint and linked}\}\\
        \partial Z &:= \{(\lambda^+, \lambda^-) \mid \lambda^+ \text{ and }\lambda^- \text{ intersect}\}\\
        \bZ &:= \{(\lambda^+, \lambda^-) \mid \lambda^+ \text{ and }\lambda^- \text{ intersect or link}\}= Z \cup \partial Z
    \end{align*}
    and call $\bZ$ the \emph{especial disc} associated to $\Lambda^\pm$.

    We define the following decompositions of the especial disc:
    \begin{align*}
        \bcZ^+ &:= \{ (\pi_+)^{-1}(\lambda) \mid \lambda \in \Lambda^+ \}\\
        \bcZ^- &:= \{ (\pi_-)^{-1}(\lambda) \mid \lambda \in \Lambda^- \}.\qedhere
    \end{align*}
\end{definition}

The preceding discussion yields:
\begin{proposition}
    Let $\Lambda^\pm$ be an especial pair on $S^1$. Then:
    \begin{itemize}
        \item The corresponding especial disc $\bZ \subset \Lambda^+ \times \Lambda^-$ is a topological disc with interior $Z$ and boundary $\partial Z$.

        \item $\bcZ^\pm$ are an emuu pair on $\bZ$.

        \item The map $S^1 \to \partial Z$ defined by $s \mapsto (\Lambda^+(s), \Lambda^-(s))$ is a homeomorphism that identifies $\Lambda^\pm$ with $\partial \bcZ^\pm$.
    \end{itemize}
\end{proposition}

\subsection{Uniqueness and actions}\label{subsec:uniqueness}
The following result is tantamount to uniqueness of emuu pairs extending a given especial pair.

\begin{proposition}\label{lemma:doubleboundaryimage}
    Let $\bcF^\pm$ be an emuu pair on a disc $\bD=D \sqcup S$ and let $\Lambda^\pm := \partial \bcF^\pm$ be the corresponding especial pair on $S$. Then the double boundary map
    \[ \delta\colon \bD \to \Lambda^+ \times \Lambda^-\]
    is a homeomorphism onto its image, which is the especial disc $\bZ = Z \sqcup \partial Z$. This takes $D$ and $S$ homeomorphically to $Z$ and $\partial Z$, respectively. Moreover, $\delta(\bcF^+) = \bcZ^+$ and $\delta(\bcF^-) = \bcZ^-$.
\end{proposition}

Note that the discussion in \Cref{sec:doubleboundaryespecialdisc} shows that \Cref{lemma:doubleboundaryimage} holds for the specific emuu pair from \Cref{construction:emuuextension}; the point of the proposition is that it applies to \emph{any} emuu pair $\bcF^\pm$.
To prove the proposition we will need the following lemma.

\begin{lemma}\label{lemma:linkingandintersecting}
    Let $\bSp^\pm$ be a muu pair on $\bD=D\sqcup S$ with \emph{connected intersection}, i.e. such that $\bK^+ \cap \bK^-$ is empty or connected for all $\bK^\pm \in \bSp^\pm$. Then for every pair $\bK^\pm \in \bSp^\pm$ that intersect, the sets $\del \bK^+, \del\bK^-$ either intersect or are linked.
\end{lemma}

The proof of this will use the following topological property of the 2-sphere:

  \begin{fact}[name={\cite[Thm II.5.29]{Wilder}}]\label{Wilderfact}
        If $x$ and $y$ are points in $S^2$ which are not separated by either of the closed sets $K$ and $L$, and $K \cap L$ is connected, then $x$ and $y$ are not separated by $K \cup L$.
    \end{fact}

\begin{proof}[Proof of \Cref{lemma:linkingandintersecting}]
 
    Suppose there are $\bK^\pm \in \bSp^\pm$ that intersect, with $\del \bK^+, \del \bK^-$ disjoint and unlinked.

    Let $A := \bK^+ \cap \bK^-$. This is compact, it is connected by the assumption of connected intersection, and it is contained in $D$ because $\del A \subset\del \bK^+ \cap \del \bK^- = \emptyset$.

    Let $(a^+, b^+)$ be the complementary interval of $\del\bK^+$ that contains $\del\bK^-$. Since $\partial \bSp^+$ is special (\Cref{prop:muuRestrictsToSpecial}), we can use the nesting property (\Cref{definition:spS} \ref{it:spSNestingSequences}) to find $\bL^+ \in \bSp^+$ that separates $\bK^+$ from $\partial \bK^-$. Similarly we can find $\bL^- \in \bSp^-$ that separates $\bK^-$ from $\partial \bL^+$ (and hence also from $\bK^+$). See \Cref{fig:linkingandintersecting}.

    \begin{figure}
        \centering
        \begin{overpic}{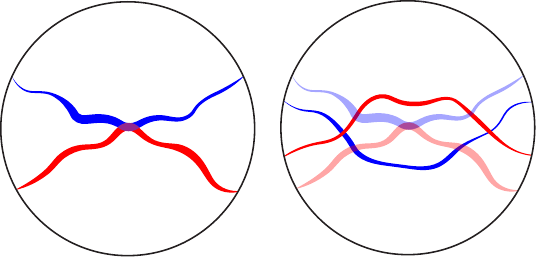}
        \small
        \put(7, 10){$\bK^+$}
        \put(38, 32){$\bK^-$}
        \put(70, 32){$\bL^+$}
        \put(79, 12){$\bL^+$}
        \end{overpic}
        \caption{From the proof of \Cref{lemma:linkingandintersecting}.}
        \label{fig:linkingandintersecting}
    \end{figure}

    Let $U^+$ be the complementary region of $\bL^+$ that contains $\bK^+$ and let $U^-$ be the complementary region of $\bL^-$ that contains $\bK^-$. By construction, $\partial U^+$ and $\partial U^-$ are disjoint intervals. 

    If we identify $\bD$ with the upper hemisphere in $S^2$, then neither $\bL^+$ nor $\bL^-$ separate $A$ from the south pole $s$, but $\bL^+ \cup \bL^-$ does. To see this, note that any path from $A$ to $s$ that avoids $\bL^+$ must have first intersection with the equator $S^1$ contained in $\partial U^+$, while any path from $A$ to $s$ that avoids $\bL^-$ must have first intersection with the equator $S^1$ contained in $\partial U^-$. Since these are disjoint, there is no path from $A$ to $s$ that avoids $\bL^+ \cup \bL^-$. By \Cref{Wilderfact} this implies that $\bL^+ \cap \bL^-$  is disconnected, contradicting the connected intersection assumption.
\end{proof}

\begin{proof}[Proof of \Cref{lemma:doubleboundaryimage}]
\Cref{cor:emuudoubleboundaryhomeo} implies that $\delta$ is a homeomorphism onto its image, so it remains to study its image.

We first show that $\delta(S) = \partial Z$. If $s \in  S$, then $\delta(s) = (\partial\bcF^+(s), \partial \bcF^-(s))$. Since $\partial\bcF^+(s)$ and $\partial\bcF^-(s)$ intersect at $s$, $\delta(s)\in \del Z$; hence $\delta(S) \subset \partial Z$. On the other hand, if $(\partial \bK^+, \partial \bK^-) \in \partial Z$, then $(\partial \bK^+, \partial \bK^-) = \delta(s)$ for $s$ in $\partial \bK^+ \cap \partial \bK^-$, which is nonempty by the definition of $\bZ$ (in fact $\partial \bK^+ \cap \partial \bK^- = \{s\}$ by efficient intersection). Thus $\partial Z \subset \delta(S^1)$.

Next we show that $\delta(D) = Z$. Let $(\del \bK^+,\del \bK^-) \in Z$. By the definition of $Z$,  $\bK^+$ and $\bK^-$ are disjoint and linked. By \Cref{lemma:iflinkthenintersect}, they must intersect at some point $p \in D$. Then $(\del \bK^+,\del \bK^-) = \delta(p)$, and it follows that $Z \subset \delta (D)$. On the other hand, let $x \in D$. Then the elements $\del\bcF^\pm(x)$ are disjoint by efficient intersection, and they are linked by \Cref{lemma:linkingandintersecting}, so $\delta(D)\subset Z$.

Finally we show that $\delta(\bcF^\pm)=\bcZ^\pm$. For $\bK\in \bcF^+$, we claim that $\delta(\bK)=\pi_+^{-1}(\del \bK)$. Note that
\[
\delta(\bK)=\{(\del \bK, \del \bcF^-(p))\mid p\in \bK\}.
\]
It follows that $\delta(\bK)\subset\pi_+^{-1}(\del \bK)$. On the other hand, if $(\del \bK, \del \bL)\in \pi_+^{-1}(\del \bK)$, then $\bL$ must intersect $\bK$ nontrivially in a point $p$ by \Cref{lemma:iflinkthenintersect}, since $\del\bK$ and $\del \bK$ intersect or link. Then $(\del \bK, \del \bL)=\delta(p)$, so $\pi_+^{-1}(\del \bK)\subset \delta(\bK)$. This shows that $\delta(\bcF^+)=\bcZ^+$, and the argument for $\bcZ^-$ is similar. 
\end{proof}

\begin{theorem}\label{theorem:emuuUniqueness}
    Let $\bcF_i^\pm$ be emuu pairs on discs $\bD_i = D_i \sqcup S_i$, for $i = 1, 2$. If there is a homeomorphism $f: S_1 \to S_2$ taking $\partial \bcF_1^\pm$ to $\partial \bcF_2^\pm$, then there is a unique homeomorphism $F: \bD_1 \to \bD_2$ extending $f$ that takes $\bcF_1^\pm$ to $\bcF_2^\pm$.
\end{theorem}
\begin{proof}    
    For $i=1,2$, let $\Lambda_i^\pm:=\del\bcF_i^\pm$, $\bZ_i = Z_i \cup \partial Z_i$ be the especial disc associated to $\Lambda_i^\pm$, and let
    $\delta_i\colon \bD_i\to \bZ_i$
    be the double boundary map associated to $\bcF_i^\pm$. 
    
    If $f\colon S_1\to S_2$ is a homeomorphism taking $\Lambda_1^\pm$ to $\Lambda_2^\pm$, it induces a homeomorphism 
    \begin{align*}
        G\colon \bZ_1 &\to \bZ_2\\
        (\lambda^+, \lambda^-)&\mapsto (f(\lambda^+), f(\lambda^-)).
    \end{align*}
     Then $F=\delta_2^{-1}\circ G\circ \delta_1$ is a homeomorphism $\bD_1\to \bD_2$ that extends $f$ and takes $\bcF_1^\pm$ to $\bcF_2^\pm$.

    To see that $F$ is the unique extension of $f$ sending $\bcF_1^\pm$ to $\bcF_2^\pm$, note that any such extension must send $p_1\in \bD_1$ to the unique intersection point of the decomposition elements $\bcF_2^+(f(\del\bcF_1^+(p)))$ and $\bcF_2^-(f(\del\bcF_1^-(p)))$, which is $F(p_1)$.
\end{proof}	

This immediately implies that emuu pairs that extend the same especial pair on the circle may be uniquely identified (taking $f = \mathrm{id}_S$). 
\begin{corollary}
    Let $\bcF_i^\pm$ be emuu pairs on discs $\bD = D_i \sqcup S$, for $i = 1, 2$, with the same boundary circle. Then there is a unique homeomorphism $F: \bD_1 \to \bD_2$ that takes $\bcF_1^\pm$ to $\bcF_2^\pm$ and restricts to the identity on $S$.
\end{corollary}

It also has an immediate application to group actions:
\begin{corollary}\label{corollary:EfficientExtensionsAction}
    Let $\bcF^\pm$ be an emuu pair on a disc $\bD = D \sqcup S$. Then any group action $\Gamma \acts S$ that preserves $\partial \bcF^\pm$ extends uniquely to an action $\Gamma \acts \bbD$ that preserves $\bcF^\pm$. 
\end{corollary}

Alternatively, \Cref{corollary:EfficientExtensionsAction} follows from the observation that any action $\Gamma \acts S$ that preserves an especial pair $\Lambda^\pm$ induces an action on the decomposition spaces $\Lambda^+$ and $\Lambda^-$, and hence a product action on $\Lambda^+ \times \Lambda^-$. Since $\bZ$ is defined by intersection and linking properties, which are preserved by homeomorphisms, it is obviously preserved by $\Gamma \acts \Lambda^+ \times \Lambda^-$, as are the decompositions $\bcZ^\pm$.

\section{Emuu pairs and singular foliations}\label{sec:emuupairsgivefoliations}
To finish the proof of \Cref{theorem:effExt}, we will prove that every emuu pair on a closed disc restricts to a transverse singular foliations of the disc's interior (\Cref{prop:emuutofoliations}).

\subsection{Singular foliations}\label{sec:singularfoliations}

For each $n \ge 2$, consider the \emph{standard foliated $2n$-gon} and \emph{standard bifoliated $2n$-gon} as depicted (for $n=3$) in \Cref{figure:StandardFoliatedPolygons}. We think of the foliations on these spaces as decompositions.

A \emph{singular foliation} of a surface $\Sigma$ can be thought of as a partition $\cF$ of $\Sigma$ such that each point has a neighborhood $U$ identified with a standard foliated $2n$-gon ($n\ge 2$) in such a way that $\Mon(\cF \capdot U)$ is identified with the standard foliation. Given a leaf $L \in \cF$, the components of $L \cap U$ are called the \emph{plaques} of $L$ in $U$.

A \emph{pair of transverse singular foliations} of a surface is defined similarly using standard bifoliated $2n$-gons.

\begin{figure}[ht]
	\centering
	\includegraphics[scale=0.8]{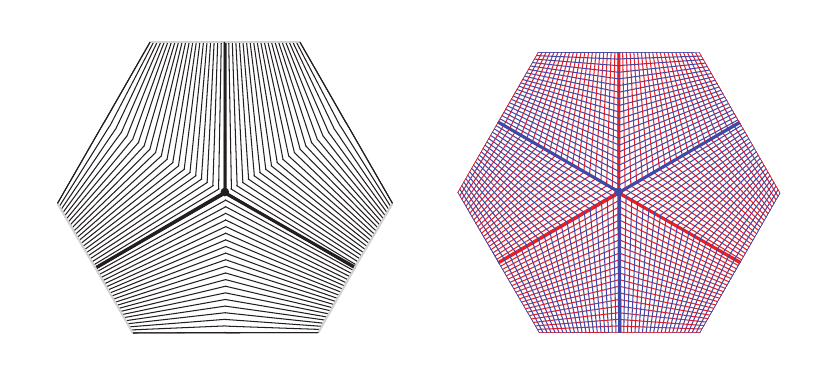}
	\caption{The standard foliated and bifoliated $6$-gon.} \label{figure:StandardFoliatedPolygons}
\end{figure}

\begin{remark}
    Our definition allows for leaves with multiple singular points (or equivalently ``saddle connections"). Cf. the ``bifoliated planes'' in \cite[Def. 2.1]{BarthelmeFrankelMann}, which allow only one singular point per leaf.
\end{remark}

\Cref{proposition:effExtExistence}, \Cref{theorem:emuuUniqueness}, and \Cref{corollary:EfficientExtensionsAction} together prove most of \Cref{theorem:effExt}. To complete the theorem, we will prove that an emuu pair in $\bbD=D\sqcup S$ induces a pair of transverse essential singular foliations in $D$. Note that the decompositions $\bcF^\pm\capdot D$ may not be monotone, but we can monotonize them by taking components as in \Cref{sec:monotonization}.

\begin{proposition}\label{prop:emuutofoliations}
    Let $\bcF^\pm$ be an emuu pair on $\bD=D\sqcup S$. Then $\Mon(\bcF^\pm \capdot D)$ is a transverse pair of singular foliations of $D$.
\end{proposition}
The remainder of this section will work toward proving \Cref{prop:emuutofoliations}.

\subsection{The proof of \Cref{prop:emuutofoliations}}

Fix $\Lambda^\pm$ and $\bcF^\pm$ as in \Cref{prop:emuutofoliations}, and let $\cF^\pm := \Mon(\bcF^\pm \capdot D)$.

Now, construct the decompositions $\bcH^\pm$ and $\bcH^\cap$ of $\bD$ as in the proof of \Cref{construction:emuuextension}. By the uniqueness of emuu pairs (\Cref{theorem:emuuUniqueness}), we may identify 
the decompositions $\bcF^\pm$ with the images of $\bcH^\pm$ under the quotient map $\tau: \bD \to \bD/\bcH^\cap \cong \bD$.

For convenience, let us set some notation.
\begin{notation}
	Recall that the elements of $\bcF^\pm$ through a point $p \in \bD$ are denoted by $\bcF^\pm(p)$. We will abuse notation and denote the corresponding elements of $\Lambda^\pm$ and $\bcH^\pm$ by
	\begin{align*}
	\Lambda^\pm(p) 	&:= \del\bcF^\pm(p) \text{ and }\\
	\bcH^\pm(p) 	&:= \Hull(\Lambda^\pm(p)),
	\end{align*}
	and the corresponding element of $\bcH^\cap$ by
	\[ \bcH^\cap(p) := \bcH^+(p) \cap \bcH^-(p).\qedhere \] 
\end{notation}

\subsubsection{Segments in sprigs} We now show that the elements of $\bcF^+$ and $\bcF^-$ are uniquely arc-connected.

\begin{lemma}\label{lemma:SprigSegments}
	Let $\bK$ be an element of $\bcF^+$ (or $\bcF^-$), and let $p, q \in \bK$. Then there is a subset $\bK[p, q] \subset \bK$ containing $p$ and $q$ such that:
	\begin{enumerate}[label=(\arabic*)]
		\item\label{it:SprigSegments1} $\bK[p, q]$ is homeomorphic to an arc with endpoints $p$ and $q$,
		\item\label{it:SprigSegments2} $r \in \bK[p, q] - \{p, q\}$ if and only if $r \in \bK$ and $\Lambda^-(r)$ separates $\Lambda^-(p)$ from $\Lambda^-(q)$ in $S$ (respectively, $\Lambda^+(r)$ separates $\Lambda^+(p)$ from $\Lambda^+(q)$), and
        \item $\bK[p,q]$ is the unique minimal connected subset of $\bK$ with this property.
	\end{enumerate}
\end{lemma}

We will view $\bK[p, q]$ as an \emph{oriented} arc, and call it the \emph{$\bK$-segment from $p$ to $q$}.

\begin{proof}
    We assume $\bK\in \bcF^+$; the other case is symmetric.
    Let $H^+=\tau^{-1}(\bK)$ ($=\bcH^+(p)=\bcH^+(q)$), and let $A=\bcH^-(p)$, $B=\bcH^-(q)$. Let $[A,B]$ be the separation interval from $A$ to $B$ in $\bcH^-$ (see \Cref{definition:separationinterval}).
    Let $Y=H^+\cap|[A,B]|$, and note that every $C\in [A,B]$ intersects $Y$. Let 
    \[
    \bK[p,q]:=[A,B]\capdot Y=\tau(Y).
    \]
    Since $[A,B]$ is an upper semicontinuous decomposition of $|[A,B]|$, \Cref{lemma:restrictionUSC} tells us that $[A,B]$ is homeomorphic to $[A,B]\capdot Y$ via a homeomorphism carrying $A$ to $p$ and $B$ to $q$. Since $[A,B]$ is homeomorphic to a compact interval with endpoints $A$ and $B$ by \Cref{lemma:separationintervals}, this establishes (1).
    Property (2) holds by construction, and property (3) follows from (2).
\end{proof}

\begin{remark}
	If $p, q$ are contained in a single leaf $K \in \cF^+$ within $\bK \in \bcF^+$, then $\bK[p, q] \subset K$. This follows from the fact that $K$ is connected and $\bK[p, q]$ is minimal among connected subsets containing $p$ and $q$. In this case, we will write $K[p, q] = \bK[p, q]$ and call this the \emph{$K$-segment from $p$ to $q$}. In general, $\bK[p, q]$ may intersect the boundary circle. Both these cases are illustrated in \Cref{fig:leafsegments}.
\end{remark}

\begin{figure}
    \centering
    \begin{overpic}{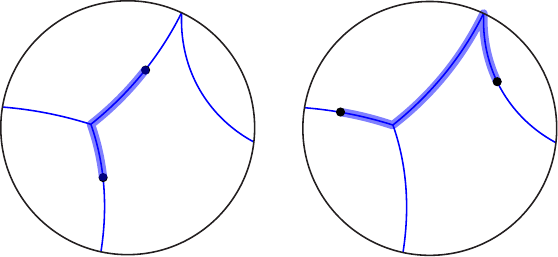}
    \small
    \put (20,13){$p$}
    \put (27,31){$q$}
    \put (61,28){$p$}
    \put (91, 30){$q$}
    \end{overpic}
    \caption{A $\bK$-segment $\bK[p,q]$ (highlighted) is contained in a single element $K\in\cF^\pm$ if $p,q\in K$. Otherwise $\bK[p,q]$ will intersect $S^1$.}
    \label{fig:leafsegments}
\end{figure}
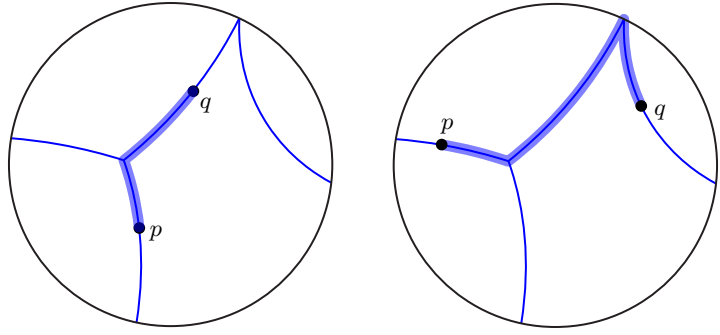

\subsubsection{Sectors and charts}
To show that $\cF^\pm$ form a transverse pair of singular foliations, we will show that each point $p \in D $ is contained in a standard bifoliated $2n$-gon.

Fix a point $p \in D$, and let $\lambda^\pm = \Lambda^\pm(p)$, $H^\pm = \bcH^\pm(p)$, and $\bK^\pm = \bcF^\pm(p)$. By efficient intersection, $\lambda^\pm$ are disjoint and hence $n$-linked for some finite $n$ (cf. \Cref{lemma:finitelinking}), hence there are $n$ complementary intervals of $\lambda^+ \cup \lambda^-$ of the form $(s^+, s^-)$ for $s^+ \in \lambda^+$ and $s^- \in \lambda^-$ and $n$ of the form $(s^-, s^+)$ for $s^- \in \lambda^-$ and $s^+ \in \lambda^+$. We will call these intervals \emph{$\pm$-} and \emph{$\mp$-sectors}, respectively. See \Cref{figure:Sectors}.

\begin{figure}[ht]
	\begin{overpic}{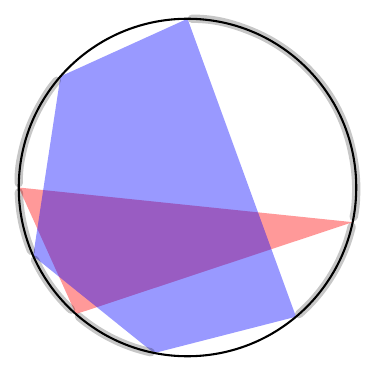}
		\put(10, 90){$\bD$}
		
		\tiny
		\put(26, 35){$ H^+ \cap H^- = \bcH^\cap(p)$}
		\put(35, 65){$H^-$}
		\put(76, 38){$H^+$}
		
		\put(87, 22){$\mp$-sector}
		\put(80, 85){$\pm$-sector}

	\end{overpic}
	\caption{Sectors.} \label{figure:Sectors}
\end{figure}

Consider a single $\pm$-sector $I = (s_1, s_2)$, and label points $s_0 \in \Lambda^-(p)$ and $s_3 \in \Lambda^+(p)$ so that $(s_0, s_2)$ is the complementary interval of $\Lambda^-(p)$ that contains $s_1$ and $(s_1, s_3)$ is the complementary interval of $\Lambda^+(p)$ that contains $s_2$. See \Cref{figure:FoliationCharts}. 

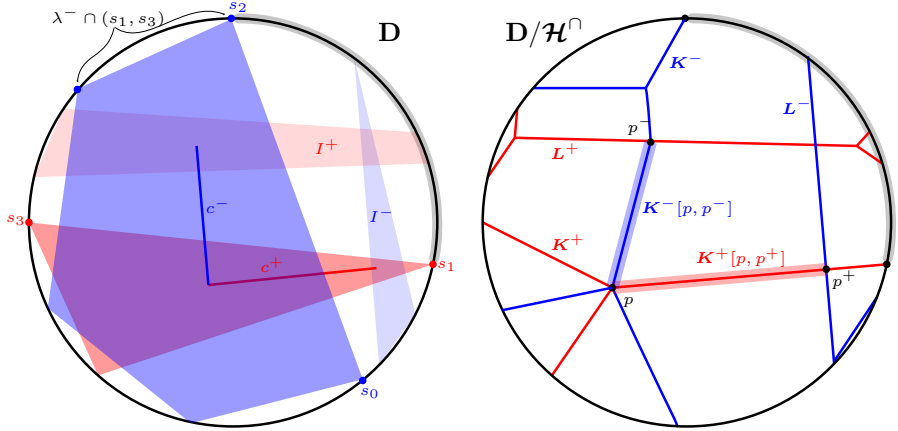
\begin{figure}[ht]
	\begin{overpic}[scale=.95]{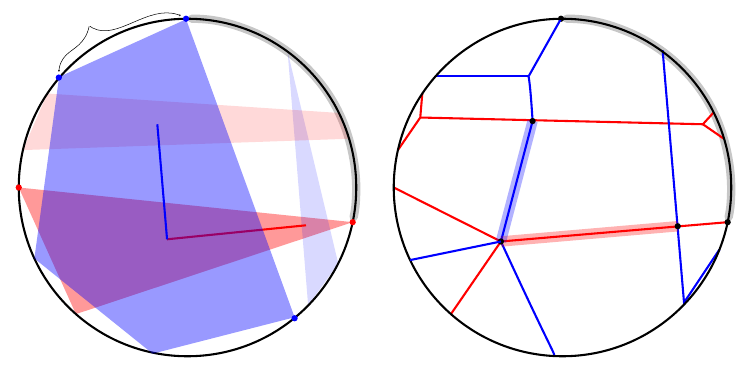}
    		\put(41, 45){$\bD$}
		\put(55, 45){$\bD/\bcH^\cap$}

		\tiny
		\put(39, 6){\textcolor{blue}{$s_0$}}
		\put(47.5, 20){\textcolor{red}{$s_1$}}
		\put(24.5, 48.5){\textcolor{blue}{$s_2$}}
		\put(0, 25){\textcolor{red}{$s_3$}}
		
		\put(68, 16){$p$}
		\put(91, 18){$p^+$}
		\put(68.5, 35){$p^-$}
		
		\put(34, 32.5){\textcolor{red}{$I^+$}}
		\put(40, 25){\textcolor{blue}{$I^-$}}
		
		\put(22, 26){\textcolor{blue}{$c^-$}}
		\put(28, 19.5){\textcolor{red}{$c^+$}}
		
		\put(73, 42){\textcolor{blue}{$\bK^-$}}
		\put(60, 22){\textcolor{red}{$\bK^+$}}
		
		\put(70, 26){\textcolor{blue}{$\bK^-[p, p^-]$}}
		\put(76, 20.5){\textcolor{red}{$\bK^+[p, p^+]$}}
		
		\put(85.5, 37){\textcolor{blue}{$\bL^-$}}
		\put(60, 32){\textcolor{red}{$\bL^+$}}
		
		\put(5, 47){$\lambda^- \cap (s_1, s_3)$}
		  
	\end{overpic}
	\caption{Building charts on sectors.} \label{figure:FoliationCharts}
\end{figure}

Using the nesting property (\Cref{definition:spS}~\ref{it:spSNesting}), we can find an element $\mu^+ \in \Lambda^+$ that separates $\lambda^+$ from $\lambda^- \cap (s_1, s_3)$. The corresponding hull $I^+ \in \bcH^+$ intersects $H^-$ in a convex polygon contained in $D$, which corresponds to a point $p^- \in \bK^- \cap \bL^+$, where $\bL^+$ is the element of $\bcF^+$ corresponding to $I^+$. Note that the segment $\bK^-[p, p^-]$ is contained entirely in $Q$---this is because $\mu^+$ was chosen so that no point in $\lambda^-$ lies between $\lambda^+$ and $\mu^+$. In \Cref{figure:FoliationCharts}, $c^-$ is a line segment corresponding to $\bK^-[p, p^-]$.

Choose $\mu^-$ similarly, except that we require it to separate $\lambda^-$ from both $\lambda^+ \cap (s_0, s_2)$ \emph{and} $\mu^+ \cap (s_0, s_2)$. Let $I^-$, $L^-$, and $p^+$ be defined similarly. The additional requirement for $\mu^-$ ensures that $H^+$, $I^+$, $H^-$, and $I^-$ cobound a quadrilateral in $D$. 

Let $R=|[K^+, I^+]|\cap|[K^-, I^-]|$, which is a compact $\bcH^\cap$-saturated set. Further, $R$ is contained in $D$ by our careful choice of $\mu^+$ and $\mu^-$. 

Taking a product of the quotient maps for $\bcH^+$ and $\bcH^-$ and restricting to $R$, we have a continuous closed map
\[
f\colon R\to [K^+,I^+]\times[K^-,I^-].
\]
Note that there is a bijection 
\[
g\colon\tau(R)\to [K^+,I^+]\times[K^-,I^-]
\]
carrying $x\in \tau(R)$ to the pair $(\bcH^+(x),\bcH^-(x))$, and that $g\circ \tau=f$. By \Cref{lem:quotientfact}, $g$ is a homeomorphism. By construction $g$ carries $\cF^+\cap \tau(R)$ and $\cF^-\cap \tau(R)$ to the foliations of $[K^+,I^+]\times[K^-,I^-]$ by horizontal and vertical lines.

One can do this for each of the finitely many sectors at $p$ and combine the results to construct a single chart to a standard bifoliated $2n$-gon.
This completes the proof of \Cref{prop:emuutofoliations}, and thus that of  \Cref{theorem:effExt}.

\section{Straightening maps}\label{sec:Straightening}

We now have a bijective correspondence between especial pairs in the circle and emuu pairs in the disc up to homeomorphism, where the latter determine transverse pairs of singular foliations of the interior. 

\subsection{Intersection properties}\label{subsec:IntersectionProperties}

Consider the following property that may apply to a pair $\bSp^\pm$ of muu decompositions of a disk $\bD=D\sqcup S$, governing the way they interact with the boundary circle:
\begin{enumerate}[label=($\cap_0$)]
    \item \label{it:boundaryefficient} For each $s \in S$, $\partial \bSp^+(s) \cap \partial\bSp^-(s) = \{s\}$.
\end{enumerate}
Note that \ref{it:boundaryefficient} allows elements $\bK^+ \in \bSp^+$ and $\bK^- \in \bSp^-$ to intersect in both $D$ and $S$, simply requiring that their intersection in $S$ is at most one point.
Note also that \ref{it:boundaryefficient} suffices to ensure that $\partial \bSp^\pm$ is an especial pair:

\begin{corollary}\label{cor:whenEmuuRestrictsToEspecial}
    If $\bSp^\pm$ is an muu pair on $\bbD$ with property \ref{it:boundaryefficient} then $\partial \bSp^\pm$ is an especial pair on $S^1$.
\end{corollary}

This is immediate corollary of \Cref{prop:muuRestrictsToSpecial}, since \ref{it:boundaryefficient} is equivalent to $\partial \bSp^\pm$ having the efficient intersection property (cf \Cref{cor:emuuRestrictsToEspecial}). 
Although trivial, this observation is important in our main application, in which we will want to turn a pair of decompositions $\bcL^\pm$ of the plane, associated to a quasigeodesic flow, into a transverse pair of singular foliations. It follows from \cite{Frankel_closing} that these decompositions naturally determine a muu pair $\bSp^\pm$ with property \ref{it:boundaryefficient}. Then \Cref{cor:whenEmuuRestrictsToEspecial} shows that $\partial \bSp^\pm$ is especial, and \Cref{proposition:effExtExistence} provides an emuu pair $\bcF^\pm$ with $\partial \bSp^\pm = \partial \bcF^\pm$ which restricts to a transverse pair of singular foliations of the plane by  \Cref{prop:emuutofoliations}.

In fact, the decompositions we are concerned with have a stronger intersection property:

\begin{definition}\label{def:propermuu}
    A muu pair $\bSp^\pm$ in a disc $\bD=D\sqcup S$ is called \emph{proper} if it satisfies the following intersection property:
    \begin{enumerate}
    \item[($\cap_1$)] For each $s \in S$, $\bSp^+(s) \cap \bSp^-(s) = \{s\}$. 
\end{enumerate}
    Equivalently: for any $\bK^+\in\bSp^+$ and $\bK^-\in \bSp^-$, the intersection $\bK^+\cap \bK^-$ is either $\emptyset$, a subset of $D$, or a single point in $S$.
\end{definition}

Properness will be important in \Cref{subsection:StraighteningMaps}, where we build ``straightening maps'' that relate $\bSp^\pm \capdot D$ and $\bcF^\pm \capdot D$.

\subsection{Straightening maps}\label{subsection:StraighteningMaps}
Consider a pair of muu decomposition $\bSp^\pm$ of a disc $\bD = D \sqcup S$ with the intersection property \ref{it:boundaryefficient}. Then $\Lambda^\pm :=\partial \bSp^\pm$ is an especial pair (\Cref{cor:whenEmuuRestrictsToEspecial}), so one can construct an emuu pair $\bcF^\pm$ on a disc $\bQ= Q \sqcup S$ with $\partial \bcF^\pm = \Lambda^\pm :=\partial \bSp^\pm$.

\begin{notation}
    Given a pair of muu decompositions $\bSp^\pm$ of a disc $\bD = D \sqcup S$, set
    \begin{align*}
        D_\mathrm{I} &:=  \{ p \in D \mid \partial \bSp^+(p) \text{ and } \partial \bSp^-(p) \text{ intersect} \} \\
        D_\mathrm{L} &:=  \{ p \in D \mid \partial \bSp^+(p) \text{ and } \partial \bSp^-(p) \text{ are disjoint and linked} \}.\qedhere
    \end{align*}
\end{notation}

\begin{theorem}\label{bigstraighteningtheorem}
    Let $\bSp^\pm$ be a muu pair on a disc $\bD = D \sqcup S$ with intersection property \ref{it:boundaryefficient}, and let $\bcF^\pm$ be an emuu pair on a disc $\bQ = Q \sqcup S$ with $\partial \bSp^\pm = \partial \bcF^\pm$.
    
    Then $D_\mathrm{I} \sqcup D_\mathrm{L}$ is a closed, nontrivial subset of $D$. For each $p \in D_\mathrm{I} \sqcup D_\mathrm{L}$ there is a unique $s(p) \in \bQ$ such that $\partial\bSp^\pm(p) = \partial \bcF^\pm(s(p))$, where $s(p) \in S \Leftrightarrow p \in D_\mathrm{I}$ and $s(p) \in Q \Leftrightarrow p \in D_\mathrm{L}$.
    
    Moreover, the assignment $p \mapsto s(p)$ defines a continuous map
    \[ s\colon D_\mathrm{I} \sqcup D_\mathrm{L} \to \bQ \]
    which restricts to a surjection
    \[ s|_{D_\mathrm{L}}\colon D_\mathrm{L} \twoheadrightarrow Q \]
    and extends continuously to a  surjection
    \[ \bs\colon S \sqcup D_\mathrm{I} \sqcup D_\lk \twoheadrightarrow \bQ \]
    by setting $\bs|_S = \mathrm{id}_S$.
\end{theorem}
\begin{proof}
    Let $\bcZ^\pm$ be the emuu pair associated to the especial disc $\bZ = Z \sqcup \partial Z$ (see \Cref{def:especialdecomps}). Since any emuu pair $\bcF^\pm$ on a disc $\bQ = Q \sqcup S$ can be uniquely identified with $\bcZ^\pm$ (\Cref{lemma:doubleboundaryimage}) it suffices to prove the theorem with $\bZ$ and $\bcZ^\pm$ replacing $\bQ$ and $\bcF^\pm$.

    Let
    \[ \delta: \bD \to \Lambda^+ \times \Lambda^- \]
    be the double boundary map (\Cref{definition:doubleboundarymap}). From the definition of $\delta$ and $\bZ = Z \sqcup \partial Z$ it follows immediately that
    \begin{align*}
        \delta^{-1}(\partial Z) &= \{ p \in \bD \mid \partial \bSp^+(p) \text{ and } \partial \bSp^-(p) \text{ intersect} \} = S \sqcup D_\mathrm{I},\\
        \delta^{-1}(Z) &= \{ p \in \bD \mid \partial \bSp^+(p) \text{ and } \partial \bSp^-(p) \text{ are disjoint and linked}\} = D_\mathrm{L},\\
        \delta^{-1}(\bZ) &= \delta^{-1}(\partial Z) \sqcup \delta^{-1}(Z) = S \sqcup D_\mathrm{I} \sqcup D_\mathrm{DL}.
    \end{align*}

    Note that $s$ as defined is simply the restriction of $\delta$ to $D_\mathrm{I} \sqcup D_\mathrm{L}$. Since on the boundary circle $S$, $\delta$ is simply the canonical identification $S \simeq \partial Z$, it follows that $\bs$ is the restriction of $\delta$ to $S \sqcup D_\mathrm{I} \sqcup D_\mathrm{L}$. Hence $\bs$ and $s$ are continuous.

    Since $\bZ$ is closed and $\delta$ is continuous, $\delta^{-1}(\bZ) = S \sqcup D_\mathrm{I} \sqcup D_\mathrm{L}$ is closed in $\bD$, hence $D_\mathrm{I} \sqcup D_\mathrm{L} = (S \sqcup D_\mathrm{I} \sqcup D_\mathrm{L}) \cap D$ is closed in $D$.

    To see that $s|_{D_\mathrm{L}}$ is a surjection it suffices to show that $\delta(D_\mathrm{L}) = Z$. Let $z \in Z$. Then by the definition of $Z$, $z = (\lambda^+, \lambda^-)$ for $\lambda^\pm \in \Lambda^\pm$ disjoint and linked. Since $\Lambda^\pm = \partial \bSp^\pm$, we have $\lambda^\pm = \partial \bK^\pm$ for $\bK^\pm \in \bSp^\pm$. The $\bSp^\pm$ are monotone decompositions, so $\bK^\pm$ are closed, connected subsets with $\partial \bK^\pm$ linked, and must therefore intersect by \Cref{lemma:iflinkthenintersect}. Since $\partial \bK^\pm = \bK^\pm \cap S$ are disjoint, they must intersect at some point $p \in Q$. Then $s(p) = \delta(p) = (\lambda^+, \lambda^-) = z$. Thus $\delta(D_\mathrm{L}) = Z$. 
    It follows that $D_\mathrm{L}$ is nontrivial and $\bs$ is surjective.
\end{proof}

A simple example of straightening is shown in \Cref{fig:straightening}.

Note that $s$ can take points in the interior to the boundary; equivalently, one may have $D_\mathrm{I} \neq \emptyset$. For example, this happens to the hull decompositions in \Cref{fig:collapsingexamples}(c)--(e). Here is another example: take an emuu pair $\bcF^\pm$ in a disc $\bD=D\sqcup S$ and containing elements $\bK^\pm\in \bcF^\pm$ that meet $D$ and intersect in a single point in $S$. Now perturb $\bcF^-$ by an isotopy that fixes $S$ so that $\bK^\pm$ intersect in $D$.

For our applications of this theory in \Cref{part:flowstraightening}, it would be undesirable for points in the interior to map to the boundary under straightening. However, this does not happen when $\bSp^\pm$ is a \emph{proper} muu pair, where we have the following simpler version of the preceding theorem.

\begin{corollary}\label{prop:properstraightening} 
    Let $\bSp^\pm$ be a proper muu pair (\Cref{def:propermuu}) on a disc $\bD = D \sqcup S$, and let $\bcF^\pm$ be an emuu pair on a disc $\bQ = Q \sqcup S$ with $\partial \bSp^\pm = \partial \bcF^\pm$.

    Then $D_\mathrm{L}$ is a closed, nontrivial subset of $D$. For each $p \in D_\mathrm{L}$ there is a unique $s(p) \in Q$ such that $\partial\bSp^\pm(p) = \partial \bcF^\pm(s(p))$, and the assignment $p \mapsto s(p)$ defines a continuous {closed} surjection
    \[ s\colon D_\mathrm{L} \twoheadrightarrow Q, \]
    which extends to a continuous {closed} surjection
    \[ \bs\colon D_\lk \sqcup S \twoheadrightarrow \bQ \]
    by setting $\bs|_S = \mathrm{id}_S$.
\end{corollary}
\begin{proof}
    Properness implies that $\partial \bSp^+(p)$ and $\partial \bSp^-(p)$ are disjoint for each $p \in D$, so we have $D_\mathrm{I} = \emptyset$. Hence \Cref{bigstraighteningtheorem} implies that $s$ and $\bs$ are continuous surjections $D_\lk\to Q$ and $D_\lk\sqcup S\to \bQ$ respectively.

    It remains only to check $\bs$ and $s$ are closed.
    By \Cref{bigstraighteningtheorem}, $S\sqcup D_\lk$ is closed in $\bD$, hence compact. Since $\bQ$ is Hausdorff, $\bs$ is closed. If $C$ is a closed subset of $\bD$, then $s(C\cap D_\lk)=\bs(C)\cap Q$, a closed subset of $Q$. Hence $s$ is closed.
\end{proof}

\begin{remark}
    When $\bSp^\pm$ is a proper muu pair on a disc $\bD = D \sqcup S$, one can define $D_\lk$ by
    \[ D_\mathrm{L} = \{ p \in D \mid \partial \bSp^+(p) \text{ and } \partial \bSp^-(p) \text{ are linked} \}.\qedhere\]
\end{remark}

\begin{figure}
    \centering
    \includegraphics[height=1.5in]{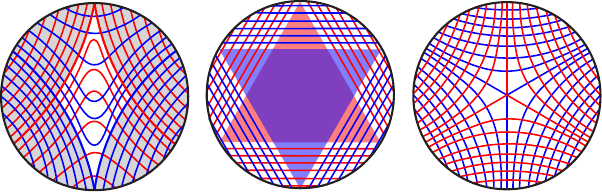}
    \caption{Left: a (proper) muu pair in the disc, with the associated linked subset shaded. Center: the corresponding hull decompositions. Right: the  emuu pair associated to the restrictions to $S^1$ of the previous two pairs.
    In this example, the straightening map collapses the boundary of the unshaded diamond-shaped region in the leftmost disk by identifying points in the red (blue) portion of the boundary that lie in the same blue (red) decomposition element. The result is the emuu pair on the right.}
    \label{fig:straightening}
\end{figure}

Straightening maps respect group actions:

\begin{lemma}
Let $\bSp^\pm$ be a proper muu pair on a disc $\bD = D \sqcup S$, and let $\bcF^\pm$ be an emuu pair on a disc $\bQ = Q \sqcup S$ with $\partial \bSp^\pm = \partial \bcF^\pm$. Then $\bs$ and $s$ are equivariant with respect to any two group actions $\Gamma \acts \bD$, $\Gamma \acts \bQ$ that agree on $S$.
\end{lemma}
\begin{proof}
    We simply unwind the definitions.
    Let $\Lambda^\pm=\del\bSp^\pm = \partial \bcF^\pm$, and observe that the actions $\Gamma\acts \bD$ and $\Gamma\acts \bQ$ induce the same action on $\Lambda$.

    As in the proof of \Cref{bigstraighteningtheorem}, we can identify $\bQ$ with the especial disk $\bZ$ associated to $\Lambda^\pm$. Hence it suffices to show that the double boundary map $\delta\colon \bD\to \bZ$ intertwines the action $\Gamma\acts \bD$ with the action $\Gamma\acts \bZ$ given by $g(\lambda^+,\lambda^-)=(g \lambda^+, g \lambda^-)$. We have
    \begin{align*}
    \delta(g p)&=(\del \bSp^+(gp), \del \bSp^-(gp))\\
    &=(g\del \bSp^+(p), g\del \bSp^-(p))\\
    &=g(\del \bSp^+(p), \del \bSp^-(p))\\
    &=g\delta(p). \qedhere
    \end{align*}
\end{proof}

\part{Straightening quasigeodesic flows}\label{part:flowstraightening}

Having established that especial pairs in the circle can be filled in to transverse singular foliations of the disc, and developed the technology of straightening maps, we will now focus on turning quasigeodesic flows into pseudo-Anosov ones. The rest of the paper will be devoted to proving our main theorem, \Cref{theorem_main}.

We will begin with some necessary background on pseudo-Anosov flows, quasigeodesic flows, and their associated universal circles.

\section{Quasigeodesic and pseudo-Anosov flows}\label{sec:flowbackground}
Throughout this article, a \emph{flow} on a manifold $M$ will be considered topologically, as a continuous map 
\[ \Phi^{(\cdot)}(\cdot): \bbR \times M \to M \]
such that 
\begin{itemize}
	\item $\Phi^0(x) = x$ for all $x \in M$, and 
	\item $\Phi^s(\Phi^t(x)) = \Phi_{t + s}(x)$ for all $x \in M$ and $t, s \in \bbR$.
\end{itemize}
For each $t \in \bbR$, the \emph{time-$t$ map} $\Phi^t(\cdot): M \to M$ is a homeomorphism, since it has $\Phi^{-t}(\cdot)$ as an inverse, so we can think of a flow as a continuous action $M \racts \bbR$, writing $x \cdot t = \Phi^t(x)$. 

A flow is \emph{nonsingular} if it has no global fixed points, i.e. for each point $x \in M$ there is a time $t \in \bbR$ at which $x \cdot t \neq x$. 

As with smooth flows, the leaves of a nonsingular continuous flow form a $1$-dimensional foliation so one has a local flowbox around any point, i.e. a neighborhood homeomorphic to $\bbD^2 \times [-1, 1]$ on which the flow is conjugate to the ``vertical flow'' $(x, h) \cdot t = (x, h + t)$. See \cite[\S 20]{Whitney}.

\subsection{Flowlines and flowspaces}\label{subsection:flowspaces}

Any flow $\Phi$ on an closed $3$-manifold $M$ lifts to a flow $\tPhi$ on the universal cover $\tM$. We will call the orbits of the lifted flow \emph{flowlines}. The collection of flowlines
\[ P := \{ x \cdot \bbR \mid x \in \tM \} \]
partitions $\tM$, which we endow with the quotient topology induced by the function
\begin{align*}
    \nu\colon \tM &\to P\\
    x &\mapsto x \cdot \bbR
\end{align*}
that takes each point to the flowline containing it.

Since the action $\pi_1(M) \acts \tM$ by deck transformations takes flowlines to flowlines, it induces an action $\pi_1(M) \acts P$. By the \emph{flowspace} of a flow $\Phi$ we mean this space $P$ together with this action $\pi_1(M) \acts P$. 

We use the symbol $P$ in this paper because our flows of interest will have flowspaces that are topological planes. 

\begin{theorem}
    Let $\Phi$ be a flow on a closed $3$-manifold $M$ with $\tM \simeq \bbR^3$. Assume that the flowspace $P$ is Hausdorff. Then $P \simeq \bbR^2$, there is a homeomorphism $\tM \simeq P \times \bbR$ that conjugates $\tPhi$ to the vertical flow on $P \times \bbR$, and the action $\pi_1(M) \acts P$ is by orientation-preserving homeomorphisms.
\end{theorem}

Flows with planar flowspace are therefore called \emph{product-covered}. The orientation-preserving property of the action on $P$ comes from the fact that the $\pi_1$-action on $\tM$ sends oriented flowlines to oriented flowlines.

This is a special case of \cite[Theorem 3]{MontgomeryZippin} which shows that a continuous action of $\bbR^k$ on $\bbR^n$, $(k \geq n - 2)$ is conjugate to a $k$-parameter translation group if and only if it is ``dispersive,'' and Hausdorff flowspace immediately implies dispersive.

\subsection{Pseudo-Anosov flows}\label{sec:pAflows}
Let $\phi: \Sigma \to \Sigma$ be a pseudo-Anosov homeomorphism of a closed surface, and let $\Phi$ be the corresponding suspension flow. This is the flow on the mapping torus $M_\phi:=\Sigma\times \bbR/(x,h+1)\sim (\phi(x),h)$ that is induced by the ``vertical flow'' $(x, h) \cdot t = (x, h+t)$ on $\Sigma \times \bbR$. A good reference for the basics of pseudo-Anosov suspension flows is \cite{Fried_fiberedfaces}. 

Identify $\Sigma$ with the image of $\Sigma \times \{0\}$ in $M_\phi$. The stable and unstable singular foliations of $\phi$ suspend to $2$-dimensional singular foliations
\begin{align*}
    \cF^s &= \{ k \cdot \bbR \mid k \text{ a stable leaf of } \phi \}\\
    \cF^u &= \{ l \cdot \bbR \mid l \text{ an unstable leaf of } \phi \}
\end{align*}
called the \emph{weak stable} and \emph{weak unstable} foliations of $\Phi$. All of the flowlines in a weak stable leaf are mutually forward asymptotic, while all of the flowlines in a weak unstable leaf are mutually backward asymptotic.

This can be seen directly in the $1$-dimensional \emph{strong stable} and \emph{strong unstable foliation}
\begin{align*}
    \cF^{ss} &= \{ k \cdot t \mid k \text{ a stable leaf of } \phi, t \in \bbR \} \\
    \cF^{uu} &= \{ l \cdot t \mid l \text{ an unstable leaf of } \phi, t \in \bbR \}
\end{align*}
which are $1$-dimensional singular foliations that refine the weak stable and weak unstable singular foliations. The strong stable leaves are uniformly exponentially contracted in forward time while strong unstable leaves are uniformly exponentially contracted in backward time. That is, there is a uniform constant $\lambda > 1$ such that if $x$ and $y$ are points in a single strong stable leaf, then $d_{\mathrm{leaf}}(x \cdot t, y \cdot t) < \lambda^{-t} d_{\mathrm{leaf}}(x, y)$, where $d_{\mathrm{leaf}}$ is the leafwise path metric. The same statement holds for strong unstable leaves with time reversed. Note that by compactness, this does not depend on the choice of metric on $M_\phi$ up to rechoosing $\lambda$.

In general, a {pseudo-Anosov flow} is one that looks locally like a pseudo-Anosov suspension flow. Different authors have taken this to mean slightly different things. Here are three properties generalizing the suspension flow case that one might ask of a flow $\Phi$ on a 3-manifold $M$:

\begin{enumerate}[label=(PA\arabic*)]
        \item \label{it:weakcontraction}(weak contraction/expansion) There is a transverse pair of $2$-dimensional singular foliations $\cF^s$ and $\cF^u$, leafwise preserved by $\Phi$, called the \emph{weak stable} and \emph{weak unstable} foliations. Any two flowlines in the same leaf of $\cF^s$ (resp. $\cF^u$) are forward (backward) asymptotic up to reparametrization. See \Cref{figure:2dfoliations}.
        
        By saying $\cF^{s/u}$ are transverse, we mean that every point in $M$ has a neighborhood which is a \emph{flow box}, i.e. a neighborhood $A\times I$ where $A$ is a standard bifoliated $2n$-gon for $n\ge 2$ (see \Cref{figure:StandardFoliatedPolygons}) such that for every leaf $\lambda$ of one of the two foliations of $A$, $\lambda\times I$ is contained in a leaf of $\cF^s$ or $\cF^u$.

        \item \label{it:strongcontraction}(strong contraction/expansion) There are $1$-dimensional singular foliations $\cF^{ss}$ and $\cF^{uu}$, preserved by $\Phi$, called the \emph{strong stable} and \emph{strong unstable} foliations, that refine the weak stable and weak unstable foliations in the sense that every weak leaf is a union of strong leaves, and each strong leaf intersects every orbit contained in the corresponding weak leaf. 
        Further, there exist constants $ a, C>0, \lambda>1$ such that:
        \begin{itemize}
        \item If $k$ is the lift of a strong stable leaf to the universal cover $\tM$ of $M$, and if $x,y\in k$ satisfy $d(x,y)\le C$, then
        \[
        d(x\cdot t, y\cdot t)\le a\lambda^{-t}d(x,y) \text{ for all } t\ge0.
        \]
        \item If $k$ is the lift of a strong unstable leaf to $\tM$, and if $x,y\in k$ satisfy $d(x,y)\le C$, then
        \[
        d(x\cdot t, y\cdot t)\le a\lambda^{t}d(x,y) \text{ for all } t\le0.
        \]
        \end{itemize}
        Here $d$ is some metric on $\tM$ coming a choice of metric on $M$; since $M$ is compact, the definition does not depend on the metric (up to choosing new constants). In contrast with the suspension flow example, we are not phrasing this condition in terms of a path metric on strong leaves, because we are not assuming the strong foliations are rectifiable.

        \item \label{it:markovpartition}$\Phi$ admits a Markov partition. This is a certain nice decomposition of $M$ into flowboxes that allows one to code the orbits of $\Phi$ by the boxes they traverse. Since we will not need to work directly with Markov partitions, we do not define them.\qedhere
    \end{enumerate}

If we require only (PA1) and (PA3), we obtain the definition of a ``topological pseudo-Anosov flow" in \cite{Mosher} or \cite{AgolTsang}. In those definitions there is an additional requirement that each singular orbit has a neighborhood of a standard form, but Agol--Tsang show that this follows from the weak contraction/expansion of the singular leaves of $\cF^{s/u}$ (\cite[p. 3440]{AgolTsang}). 

Requiring only (PA1) and (PA2) is another natural definition for a topological pseudo-Anosov flow that perhaps more directly generalizes suspension pseudo-Anosov flows. Moreover, we believe that (PA1) and (PA2) should together allow one to construct a Markov partition using the techniques of \cite{Ratner}.

We will use what we believe is the most standard definition:

\begin{definition}\label{definition:pseudoAnosov}
    A flow $\Phi$ on a $3$-manifold $M$ is \emph{pseudo-Anosov} if it satisfies properties (PA1) and (PA3) above.
\end{definition}

We will leave no room for ambiguity in our results: we will prove in \Cref{section:StraightenedFlowsPA} that the flow $\Psi$ in the statement of our main theorem, \Cref{theorem_main}, satisfies (PA1), (PA2), and (PA3).

\begin{figure}[ht]
	\centering
	\includegraphics{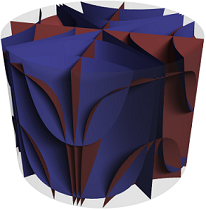}
	\caption{The weak stable (red) and unstable (blue) foliations near a 4-pronged singular orbit, where the flow is upward.} \label{figure:2dfoliations}
\end{figure}

\begin{remark}
    There is also a definition of ``smooth pseudo-Anosov flow." Using results of Brunella \cite{Brunella}, Agol--Tsang adapted an argument of Shannon \cite{Shannon} to prove that for transitive flows (i.e. those possessing a dense orbit), smooth pseudo-Anosov is the same as topologically pseudo-Anosov, up to orbit equivalence by a homeomorphism of the ambient space isotopic to the identity. Note that in a hyperbolic 3-manifold, every pseudo-Anosov flow is transitive; this was proved first in the smooth case in \cite{Mosher_norm1} and more recently in general in \cite{BarthelmeBonattiMann_nontransitive}.
\end{remark}

\subsubsection{Flowspaces, singular foliations, and universal circles}
Fenley and Mosher showed that every pseudo-Anosov flow on a closed $3$-manifold is product-covered \cite[Prop. 4.2]{FenleyMosher}, so it has a planar flowspace $P$.

The $2$-dimensional weak stable and unstable singular foliations on $M$ lift to singular foliations of $\tM$ that are preserved by the deck action $\pi_1(M) \acts \tM$, so they project to a transverse pair of $1$-dimensional singular foliations of $P$ that are preserved by the flowspace action $\pi_1(M) \acts P$.

Fenley \cite{Fenley} showed that there is a natural way to collate the topological ends of leaves of these foliations to produce a \emph{universal circle} $S^1_u$ that compactifies the flowspace to a closed disc $\bP = P\sqcup S^1_u$. The action of $\pi_1(M)$ on $P$ permutes the leaves of the $1$-dimensional singular foliations, and extends to an action $\pi_1(M) \acts \bP$.

\subsection{Quasigeodesic flows}\label{sec:QGbackground}
A flow $\Phi$ on a manifold $M$ is \emph{quasigeodesic} if each orbit lifts to a quasigeodesic in the universal cover $\tM$. That is, for each $x \in \tM$ there are constants $k>0,\epsilon\ge0$ such that
\[ 
\frac{1}{k} \cdot d(x, x \cdot t) - \epsilon \leq |t| \leq k \cdot d(x, x \cdot t) + \epsilon 
\]
for all $t\in\bbR$.

\begin{remark}
    This is the definition used in \cite{FenleyMosher}. For $C^1$ flows this is equivalent to requiring that each flowline be quasigeodesic when parametrized by arclength (\cite[Remark 3.6]{Calegari}). 
\end{remark}

The constants $k,\epsilon$ can clearly be chosen uniformly over each flowline $x \cdot \bbR$. The flow $\Phi$ is \emph{uniformly quasigeodesic} if they can be chosen uniformly for the whole lifted flow.

\begin{theorem}[name={Calegari, \cite[Lemma 3.10]{Calegari}}]
    Every quasigeodesic flow on a closed hyperbolic $3$-manifold is uniformly quasigeodesic.
\end{theorem}

The simplest examples of quasigeodesic flows come from fibrations.

\begin{example}
	Zeghib showed that any flow on a closed $3$-manifold $M$ (not necessarily hyperbolic) that is transverse to a fibration is quasigeodesic \cite{Zeghib}. The idea is to lift such a flow to the infinite cyclic cover dual to a fiber $\Sigma \subset M$, which may be identified with $\Sigma \times \bbR$ in such a way that the lifts of $\Sigma$ are of the form $\Sigma \times \{i\}$ for $i \in \bbZ$. Quasigeodesity follows from the observation that there are upper and lower bounds on the distance between adjacent lifts, as well as the time it takes for the flow to move points from one lift to the next.	
\end{example}

On the other hand, there are many quasigeodesic flows that are not transverse to fibrations, even virtually (i.e. after passing to a finite cover).

\begin{example}
    Fenley and Mosher showed that any product-covered flow transverse to a taut, finite-depth foliation is quasigeodesic \cite{FenleyMosher}, and there are many such flows that are not suspension flows or even virtually suspension flows. 
    
    Indeed, Gabai showed that any nontrivial second cohomology class on a closed $3$-manifold is represented by a union of depth-zero leaves of a taut, finite-depth foliation \cite{Gabai}. Gabai proved that this foliation is transverse to an ``almost pseudo-Anosov flow," and such flows are product-covered. Mosher began writing this argument in \cite{Mosher}, and Landry-Tsang are currently finishing it (see \cite{LandryTsang}). 
    
    If one starts with a cohomology class that is not in the closure of a fibered cone, then the associated quasigeodesic flow is not a suspension flow or virtually a suspension flow. Indeed, if it were virtually a suspension flow then an embedded surface in the lift that intersects every orbit positively would project to an immersed surface downstairs that intersects every orbit positively. This would imply that the homology directions of the flow lie in a half-space, and the flow would be a suspension by Fried \cite{Fried}.
\end{example}

\subsubsection{Endpoint maps}
Consider a closed hyperbolic $3$-manifold $M$. The universal cover $\tM$ is identified with $\bbH^3$, so it has a natural compactification to a closed $3$-ball $\tM \sqcup S^2_\infty$. Here, the \emph{sphere at infinity} $S^2_\infty$ is identified with the boundary of hyperbolic space in the unit ball model. The deck action of $\pi_1(M)$ on the universal cover is by hyperbolic isometries, so it extends to the sphere at infinity.

A general flow $\Phi$ on $M$ lifts to a flow $\tPhi$ on $\tM$, but the orbits of the lifted flow need not behave well with respect to the sphere at infinity. They may remain in bounded subsets of $\tM$, for example, or accumulate on arbitrary closed subsets of $S^2_\infty$. When $\Phi$ is quasigeodesic, however, the following so-called \emph{Morse Lemma} implies that each lifted orbit has well-defined and distinct endpoints in $S^2_\infty$. See \cite{Gromov}, \cite[Corollary~3.44]{Kapovich}, or \cite[\S III.H]{BridsonHaefliger}.

\begin{morse_lemma}\label{lemma:Morse}
	Every quasigeodesic in $\bbH^3$ lies at a bounded distance from a unique geodesic. Furthermore, there are constants $C(k, \epsilon)$ such that every $(k, \epsilon)$-quasigeodesic in $\bbH^3$ lies in the $C(k, \epsilon)$-neighborhood of its associated geodesic. 
\end{morse_lemma}

\begin{remark}\label{rmk:uniformconstant}
    For a quasigeodesic flow on a closed hyperbolic $3$-manifold, uniform quasigeodesity implies that there is a single constant $C$ such that every lifted orbit lies in the $C$-neighborhood of the associated geodesic in the universal cover.
\end{remark}

In addition, the endpoints of lifted orbits of a quasigeodesic flow vary continuously, and this behavior characterizes the quasigeodesic flows on a closed hyperbolic $3$-manifold: they are exactly the flows that can be studied ``from infinity'' in the following sense.

\begin{proposition}[\hspace{1sp}{\cite[Theorem~B]{FenleyMosher} \& \cite[Lemma~4.3]{Calegari}}]\label{proposition:QGFlowCharacterization}
	Let $\Phi$ be a flow on a closed hyperbolic $3$-manifold $M$, and let $\tPhi$ be the lifted flow on the universal cover $\tM$. Then $\Phi$ is quasigeodesic if and only if
	\begin{enumerate}
		\item each orbit of $\tPhi$ has well-defined and distinct endpoints in $S^2_\infty$, and
		\item the positive and negative endpoints of $x \cdot \bbR$ vary continuously with $x \in \tM$.
	\end{enumerate}
\end{proposition}

As a consequence, a quasigeodesic flow $\Phi$ on a closed hyperbolic $3$-manifold $M$ determines a pair of continuous maps
\[ E^\pm: \tM \to S^2_\infty, \]
taking each point $x \in \tM$ to the positive/negative endpoints of the corresponding flowline, i.e. $E^\pm(x) := \lim_{t \to \pm \infty}\tPhi^t(x)$.

\subsubsection{Omnileaves and leaves}
For the remainder of the section we fix a quasigeodesic flow $\Phi$ on a closed hyperbolic $3$-manifold $M$.

Using the endpoint maps $E^\pm$ one can construct objects analogous to the weak stable and unstable leaves of a pseudo-Anosov flow:

\begin{definition}\label{def:WeakOmnileavesLeaves}
    Let $\Phi$ be a quasigeodesic flow on a closed hyperbolic $3$-manifold $M$.
    
    For each ideal point $z \in S^2_\infty$, the set $(E^+)^{-1}(z)$, when nonempty, is called a \emph{weak positive omnileaf}, and each connected component of $(E^+)^{-1}(z) \neq \emptyset$ is called a \emph{weak positive leaf}.
    
    For each ideal point $z \in S^2_\infty$, the set $(E^-)^{-1}(z)$, when nonempty, is called a \emph{weak negative omnileaf rooted at $z$}, and each connected component of $(E^+)^{-1}(z) \neq \emptyset$ is called a \emph{weak negative leaf}.
    
    In other words, the weak positive/negative omnileaves are the maximal flow-invariant subsets of $\tM$ that are forwards/backwards asymptotic to a single point. The weak positive/negative leaves are the maximal \emph{connected} flow-invariant subsets of $\tM$ that are forwards/backwards asymptotic to a single point.

    Given a positive/negative weak leaf or omnileaf $\wt K \subset \tM$, we will call the point $z = E^\pm(\wt K)$ the \emph{root} of $\wt K$. 
    
    The collections of all weak positive and negative omnileaves will be denoted
    \[
    \tOm^\pm:=\{(E^\pm)^{-1}(z) \neq \emptyset \mid z \in S^2_\infty\},
    \]
    and the collections of all positive and negative leaves will be denoted
    \[
    \tcL^\pm:=\{\text{components of } (E^\pm)^{-1}(z) \mid z \in S^2_\infty\}.\qedhere
    \]
\end{definition}

Each weak positive/negative omnileaf is closed, being a continuous preimage of a point in the sphere $S^2_\infty$. Each point is contained in some weak positive omnileaf and some weak negative omnileaf, so $\tOm^+$ and $\tOm^-$ are decompositions of $\tM$. By construction, the weak omnileaf decompositions $\tOm^\pm$ are preserved by both the flow $\tPhi$ (which fixes each weak omnileaf) and the deck action $\pi_1(M) \acts \tM$ (whose elements may permute weak omnileaves). These observations hold as well for the decompositions into weak leaves $\tcL^\pm$.

It follows that the decompositions $\tOm^\pm$ and $\tcL^\pm$ project to $\Phi$-invariant partitions of the manifold $M$, but we will generally work with them upstairs in $\tM$.

\subsubsection{Flowspaces, omnileaves, and leaves}
Using uniform quasigeodesity, Calegari showed that any quasigeodesic flow $\Phi$ on a closed hyperbolic $3$-manifold $M$ is product covered \cite[Thm. 3.12]{Calegari}, so the flowspace $P$ is a plane. The topology of the weak omnileaves and leaves is easier to understand after projecting to $P$. 

\begin{notation}
    Given a point $p \in P$ we denote the corresponding flowline by 
    \[ \flow{p} := \nu^{-1}(p). \]
    Given a subset $A \subset P$ we denote the corresponding union of flowlines by
    \[ \flow{A} := \nu^{-1}(A). \]
See \Cref{figure:Flowspace}.
\end{notation}

\begin{figure}[ht]
	\centering
    \begin{overpic}{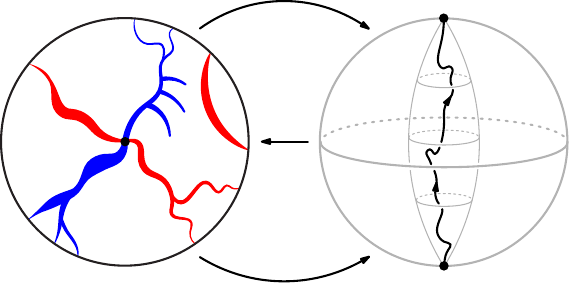}
    \small
        \put (48, 46){$e^+$}
        \put (48, 27){$\nu$}
        \put (48, 2){$e^-$}
        
        \put (19.5,26.5) {$p$}
        \put (73, 31) {$\flow{p}$}
        \put (74, 48) {$e^+(p)$}
        \put (74,-1) {$e^-(p)$}
    \end{overpic}
	\caption{Left: a point $p$ in the flowspace and the positive and negative omnileaves through $p$, which are equal to $(e^\pm)^{-1}(e^\pm(p))$. Note that omnileaves need not be connected. Right: the flowline $\flow{p}$, which stays a bounded distance from a geodesic and limits on the points $e^+(p)$ and $e^-(p)$ in $S^2_\infty$ in the forward and backward direction, respectively.}\label{figure:Flowspace}
\end{figure}
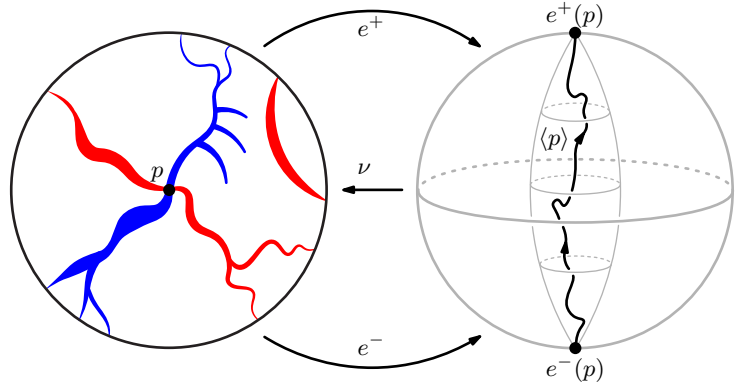

The endpoint maps $E^\pm: \tM \to S^2_\infty$ guaranteed by \Cref{proposition:QGFlowCharacterization} are constant along flowlines, so they descend to continuous endpoint maps 
\[ e^\pm: P \to S^2_\infty \]
that take each point $p$ to the positive/negative endpoint of the corresponding flowline $\flow{p}$. Since quasigeodesics have distinct endpoints, we have $e^+(p) \neq e^-(p)$ for each $p \in P$.

\begin{definition}\label{def:noncompactleaves}
    Let $\Phi$ be a quasigeodesic flow on a closed hyperbolic $3$-manifold $M$.

    The projection to $P$ of a positive or negative weak omnileaf or leaf will be called a positive or negative \emph{$P$-omnileaf} or \emph{$P$-leaf}.
    
    Equivalently: for each ideal point $z \in S^2_\infty$, the set $(e^+)^{-1}(z)$, when nonempty, is called a \emph{positive $P$-omnileaf}, and each component of $(e^+)^{-1}(z)$ is called a \emph{positive $P$-leaf}. The set $(e^-)^{-1}(z)$, when nonempty, is called a \emph{negative $P$-omnileaf}, and each component of $(e^-)^{-1}(z) \neq \emptyset$ is called a \emph{negative $P$-leaf}.

    Equivalently: a positive/negative $P$-omnileaf is a maximal subset $K \subset P$ with $e^\pm(K)$ a single point, while a positive/negative $P$-leaf is a maximal \emph{connected} subset $K \subset P$ with $e^\pm(K)$ a single point.

    Given a positive/negative $P$-leaf or $P$-omnileaf $K \subset P$, we will call the point $z = e^\pm(K)$ the \emph{root} of $K$. 
    
    The collections of all positive and negative $P$-omnileaves will be denoted
    \[
    \Om^\pm:=\{(e^\pm)^{-1}(z) \neq \emptyset \mid z \in S^2_\infty\},
    \]
    and positive and negative $P$-leaves will be denoted
    \[
    \cL^\pm:=\{\text{components of } (e^\pm)^{-1}(z) \mid z \in S^2_\infty\}.
    \]
    We will sometimes drop the ``$P$-" modifier for convenience when there is no risk of confusion.
\end{definition}

As with $\tOm^\pm$ and $\tcL^\pm$, the collections $\Om^\pm$ and $\cL^\pm$ are decompositions---in this case, of $P$. By construction, each of these four decompositions is preserved by the flowspace action $\pi_1(M) \acts P$.

The positive and negative leaf decompositions $\cL^\pm$ coming from a quasigeodesic flow are analogous to the stable and unstable singular foliations coming from a pseudo-Anosov flow, albeit with considerably more complicated topology. However, they share the following properties with a transverse pair of singular foliations. They are \emph{unbounded decompositions} that \emph{intersect compactly}:

\begin{lemma}[name={\cite[Lemma 4.8 \& 5.8]{Calegari}}]\label{lemma:UnboundedIntersectingCompactly}
    Each positive/negative leaf is an unbounded subset of $P$.

    If $K$ is a positive leaf and $L$ is a negative leaf, then $K \cap L$ is compact.
\end{lemma}

\subsubsection{Universal circles for quasigeodesic flows}
Calegari showed in \cite{Calegari} that the topological ends of the positive and negative leaves can be collated to construct a \emph{universal circle} $S^1$. In \cite{Frankel_qgflows}, Frankel showed that this universal circle $S^1$ can be used to compactify the flowspace $P$ to a closed disc $\bP = P \sqcup S^1_u$. The flowspace action $\pi_1(M) \acts P$ preserves the decompositions into positive and negative leaves, permuting their ends, so it extends to a faithful, orientation preserving action $\pi_1(M) \acts \bP$ on this \emph{compactified flowspace}.

\begin{notation}
    If $A$ is a subset of $P$ or $\bP$, we will write $\overline{A}$ for the closure of $A$ in $\bP$.
    
    If $A$ is a subset of $\bP$, we will use the notation
    \[ \partial A := A \cap S^1_u. \]
    
    If $K \in \cL^\pm$ is a positive/negative leaf, the points in $\del \overline{K}$ are called \emph{ends} of $K$.
\end{notation}

\begin{remark}
    This usage of the word ``end'' differs slightly from that of \cite{Calegari}, \cite{Frankel_qgflows}, and \cite{Frankel_spherefilling}, where it refers to a Freudenthal end. In fact, $\del \overline{K}$ is the closure of the image of $K$'s Freudenthal ends (see \cite{Frankel_qgflows}, Lemma~7.8).
\end{remark}

The compactified flowspace has the following properties:
\begin{itemize}
	\item For each positive or negative leaf $K \in \cL^\pm$, the set of ends $\del \overline{K}$ is totally disconnected (\cite[Lem. 4.2]{Frankel_closing}).
	
	\item The set
	$ \bigcup_{K \in \cL^+} \del \overline{K} $
	of all \emph{positive ends} is dense in $S^1_u$, as is the set
	$ \bigcup_{L \in \cL^-} \del \overline{L} $
	of all \emph{negative ends}. This is true by the construction of $S^1_u$, which proceeds roughly as follows: define a circular order on the Freudenthal ends of leaves, take the order completion, and take a natural quotient (see \cite{Frankel_qgflows}).
\end{itemize}

In \cite{Frankel_spherefilling}, Frankel showed that the endpoint maps $e^\pm$ extend continuously to $\pi_1$-equivariant maps 
\[ 
\be^\pm: \bP \to S^2_\infty 
\]
on the compactified flowspace. Furthermore, $\be^+$ agrees with $\be^-$ on the boundary circle, where it restricts to a map
\[ \be: S^1_u \to S^2_\infty.\]
The image of $\be$ is all of $S^2_\infty$,
so $\be$ is a $\pi_1(M)$-equivariant sphere-filling curve in $S^2_\infty$. This generalizes the Cannon--Thurston Theorem \cite{CannonThurston}, which produces such curves for suspension flows, and Fenley's construction for pseudo-Anosov flows with no perfect fits \cite{Fenley}.

\begin{lemma}\label{lemma:LeavesSharingEnds}
    Let $K$ and $L$ be two leaves in $P$, of the same or opposite sign. If $\partial \overline{K}$ intersects $\partial \overline{L}$ then they are rooted at the same point.
\end{lemma}
\begin{proof}
    Let $s \in \partial \overline{K} \cap \partial \overline{L}$. If $K$ is a positive leaf then its root is $e^+(K) = \be^+(\overline{K}) = \be^+(s)$, and if it is a negative leaf then its root is $e^-(K) = \be^-(\overline{K}) = \be^-(s)$. Similarly, the root of $L$ is either $e^+(L) = \be^+(s)$ or $e^-(L) = \be^-(s)$ depending on whether it is positive or negative. Either way, $\be^+(s) = \be^-(s)$ because the extended endpoint maps agree on $S^1_u$, so the root of $K$ is the same as the root of $L$. 
\end{proof}

This allows us to add an important additional property to \Cref{lemma:UnboundedIntersectingCompactly}:

\begin{lemma}\label{lemma:LeafEndsDisjoint}
    If $K^+ \in \cL^+$ and $K^- \in \cL^-$ are positive and negative leaves that intersect nontrivially then their ends $\partial \overline{K^+}$ and $\partial \overline{K^-}$ are disjoint.
\end{lemma}
\begin{proof}
    If $K^+ \cap K^- \neq \emptyset$, then the flowline corresponding to a point $p \in K^+ \cap K^-$ has forward and backward endpoints $e^+(p) = e^+(K^+)$ and $e^-(p) = e^-(K^-)$. If $\del\overline{K^+}$ were to intersect $\del\overline{K^-}$, then $e^+(K^+) = e^-(L^-)$ by the preceding lemma, and hence $e^+(p) = e^-(p)$. But this is impossible, since flowlines of quasigeodesic flows have distinct endpoints (\Cref{proposition:QGFlowCharacterization}).
\end{proof}

Although positive and negative leaves that intersect cannot share ends, it is possible for positive and negative leaves that are disjoint to share ends. This is akin to a ``perfect fit'' in a pseudo-Anosov flow (see e.g. \cite{Fenley}).

\subsubsection{Sprigs}
By \Cref{lemma:LeavesSharingEnds}, any two positive leaves that share an end are contained in the same positive omnileaf, and any two negative leaves that share an end are contained in the same negative omnileaf. In this section we will define another useful object called a $P$-sprig that collates leaves that share ends, but is finer than an omnileaf.

\begin{definition}\label{def:PomnileavesPsprigs}
    Let $\Phi$ be a quasigeodesic flow on a closed hyperbolic $3$-manifold $M$.
    
    For each ideal point $z \in S^2_\infty$, the set $(\be^+)^{-1}(z)$ is called a \emph{positive $\bP$-omnileaf}, and each component of $(\be^+)^{-1}(z)$ is called a \emph{positive $\bP$-sprig}. The set $(\be^-)^{-1}(z)$ is called a \emph{negative $\bP$-omnileaf}, and each component of $(\be^-)^{-1}(z)$ is called a \emph{negative $\bP$-sprig}.\

    The collections of all positive and negative $\bP$-omnileaves will be denoted
    \[
    \bOm^\pm:=\{(\be^\pm)^{-1}(z) \mid z \in S^2_\infty\},
    \]
    and the collections of all positive and negative $\bP$-sprigs will be denoted
    \[
    \bSp^\pm:=\{\text{components of } (\be^\pm)^{-1}(z) \mid z \in S^2_\infty\}.\qedhere
    \] 
\end{definition}

Note that $\bOm^\pm$ and $\bSp^\pm$ are decompositions of $\bP$ that are preserved by the action $\pi_1(M) \acts \bP$.

Since $\be^+$ is surjective, $(\be^+)^{-1}(z)$ is always nonempty. Given a positive $\bP$-omnileaf $\bK = (\be^+)^{-1}(z)$, $\bK \cap P$ is a $P$-omnileaf whenever it is nonempty. This is how we will \emph{define} $P$-sprigs:

\begin{definition}
    If $\bK$ is a positive/negative $\bP$-sprig (rooted at $z$) then $\bK \cap P$, when nonempty, is called a \emph{positive/negative $P$-sprig} (rooted at $z$). The collections of all positive and negative $P$-sprigs will be denoted
    \[
    \Sp^\pm:=\{\bK \cap P \neq \emptyset \mid \bK \in \bSp^\pm\}.\qedhere
    \]
\end{definition}

Note that $\Sp^\pm$ are decompositions of $P$ that are preserved by the action $\pi_1(M) \acts P$. They are coarser than $\cL^\pm$ and finer than $\Om^\pm$. In particular, each component of a positive/negative $P$-sprig is a positive/negative leaf, and each positive/negative leaf is contained in some $P$-sprig. In particular, any two positive/negative leaves that share an end are contained in the same $P$-sprig.

\begin{remark}
    As in \Cref{def:noncompactleaves}, we will sometimes drop the modifiers ``$P$-" or ``$\bP$-" from ``leaf/sprig/omnileaf" when there should be no chance of confusion. Note that we have not defined $\bP$-leaves; although there is a natural definition, we will not need to consider these objects.
\end{remark}

Note weak leaves and weak omnileaves are exactly of the form $\flow{K}$ for $K$ a $P$-leaf or $P$-omnileaf. We can use this to define weak sprigs.

\begin{definition}
    For each positive/negative $P$-sprig $K \in \Sp^\pm$, $\flow{K}$ is called a \emph{positive/negative weak sprig}. 
    
    The collections of all positive and negative weak sprigs will be denoted
    \[
    \tSp^\pm=\{\flow{K} \mid K \in \Sp^\pm\}. \qedhere
    \]
\end{definition}

The decompositions defined so far are collected for reference in  \Cref{ALTtable:decompositions}. Each subtable (leaf/sprig/omnileaf) has an additional row of ``strong" decompositions which will be explained in \Cref{section:strongdecompositions}.

\begin{proposition}\label{lemma:SprigEndsDisjoint}\label{prop:sprigsarespidery}
    Let $\Phi$ be a quasigeodesic flow on a closed hyperbolic $3$-manifold $M$. The associated sprig decompositions $\bSp^\pm$ are monotone, unbounded, upper-semicontinuous, and have following intersection property:
    \begin{enumerate}
        \item[($\cap_1$)] For each $s \in S$, $\bSp^+(s) \cap \bSp^-(s) = \{s\}$. 
    \end{enumerate}

\end{proposition}
That is, they form a ``proper muu pair'' as in \Cref{def:muuemuu} and \Cref{def:propermuu}.
\begin{proof}
 The decompositions $\bSp^\pm$ are monotone by construction and upper semicontinuous by \Cref{example:PointPreimages} and \Cref{lemma:monotonizationUSC}. They are unbounded because leaves are unbounded in $P$ (recall that a decomposition of a disc is unbounded if each element intersects the boundary circle).


We now verify ($\cap_1$). First we show that if two sprigs intersect in $P$, then their ends are disjoint; this is essentially the same argument as \Cref{lemma:LeafEndsDisjoint}. Note that if $\bK^+\in \bSp^+$ and $\bK^-\in \bSp^-$ intersect at a point $s\in S^1_u$, then $\be^+(\bK^+)=\be(s)=\be^-(\bK^-)$. If $\bK^+$ and $\bK^-$ additionally intersect at a point $p\in P$, then $e^+(p) = \be^+(\bK^+) = \be^-(\bK^-) = e^-(p)$, contradicting the fact that flowlines have distinct endpoints.   
The fact that no pair of positive and negative sprigs can contain more than one point of $S^1_u$ in their intersection is \cite[Lem. 5.6]{Frankel_closing}.
\end{proof}

\Cref{prop:sprigsarespidery} is the link between \Cref{part:circledisc} and \Cref{part:flowstraightening} of the paper---it allows us to apply the straightening map technology of \Cref{subsection:StraighteningMaps} to the sprig decompositions, which we do in \Cref{sec:straighteningtheflowspace}.

\begin{table}
\centering
\begin{tabular}{|l|c|}
    \hline
    
    Leaf decompositions & Where\\
    
    \hline
    $\begin{aligned}\cL^\pm &= \{\text{positive/negative leaves}\}\\ &= \{\text{components of } (e^\pm)^{-1}(z) \mid z \in S^2_\infty \} \end{aligned}$ 
    & P \rule{0pt}{5ex}\\
    
    \hline

     $\begin{aligned}\tcL^\pm &= \{\text{weak positive/negative leaves}\}\\ &= \{\flow{K} \mid K \in \cL^\pm \}\\ &= \{\text{components of } \wt K \mid \wt K \in \tOm^\pm \} \end{aligned}$& $\tM$ \rule{0pt}{7ex}\\

    \hline

    $\begin{aligned}\tcL^{\ppmm} &= \{\text{strong positive/negative leaves}\\ &= \{k \cap \flow{K} \neq \emptyset \mid k \in \tOm^\ppmm, K \in \cL^\pm\}\\ &= \{\text{components of } k \mid k \in \tOm\} \end{aligned}$& $\tM$ \rule{0pt}{7ex}\\

    \hline\hline

    Sprig decompositions&Where\\

    \hline

    $\begin{aligned}\Sp^\pm &= \{\text{positive/negative }P\text{-sprigs}\}\\ &= \{ \bK \cap P \mid \bK \in \bSp^\pm \}\end{aligned}$& P \rule{0pt}{5ex}\\

    \hline

    $\begin{aligned}\bSp^\pm &= \{\text{positive/negative }\bP\text{-sprigs}\}\\ &= \{\text{components of } \bK \mid \bK \in \bOm^\pm \}\end{aligned}$&$\bP$ \rule{0pt}{5ex}\\

    \hline

    $\begin{aligned}\tSp^\pm &= \{\text{weak positive/negative sprigs}\}\\ &= \{ \flow{K} \mid K \in \Sp^\pm \}\end{aligned}$& $\tM$ \rule{0pt}{5ex}\\

    \hline

    $\begin{aligned}\tSp^{\ppmm} &= \{\text{strong positive/negative sprigs}\}\\ &= \{k \cap \flow{K} \neq \emptyset\mid k \in \tOm^\ppmm, K \in \Sp^\pm\}\end{aligned}$& $\tM$ \rule{0pt}{5ex}\\[2ex]

    \hline\hline

    Omnileaf decompositions & Where\\

    \hline

    $\begin{aligned}\Om^\pm &= \{\text{positive/negative }P\text{-omnileaves}\}\\ &= \{(e^\pm)^{-1}(z) \neq \emptyset \mid z \in S^2_\infty \}\end{aligned}$& $P$ \rule{0pt}{5ex}\\

    \hline

    $\begin{aligned}\bOm^\pm &= \{\text{positive/negative }\bP\text{-omnileaves}\}\\ &= \{(\be^\pm)^{-1}(z) \neq \emptyset \mid z \in S^2_\infty \}\end{aligned}$& $\bP$ \rule{0pt}{5ex}\\

    \hline

    $\begin{aligned}\tOm^\pm &= \{\text{weak positive/negative omnileaves}\}\\ &= \{\flow{K} \mid K \in \Om^\pm \}\\ &= \{(E^\pm)^{-1}(z) \neq \emptyset \mid z \in S^2_\infty \}\end{aligned}$& $\tM$ \rule{0pt}{7ex}\\

    \hline

    $\begin{aligned}\tOm^{\ppmm} &= \{\text{strong positive/negative omnileaves}\}\\ &= \{R^{-1}(H^\pm) \neq \emptyset \mid H^\pm \text{a s/u horosphere}\}\end{aligned}$& $\tM$ \rule{0pt}{5ex}\\

    \hline

\end{tabular}

\caption{The leaf/sprig/omnileaf decompositions of a quasigeodesic flow. Strong decompositions will be defined in \Cref{section:strongdecompositions}.}
\label{ALTtable:decompositions}
\end{table}

\section{Straightening the flowspace}\label{sec:straighteningtheflowspace}

Fix a quasigeodesic flow $\Phi$ on a closed hyperbolic $3$-manifold $M$. This short section spells out how we apply the straightening maps of \Cref{sec:Straightening} to decompositions of the flowspace of $\Phi$.
Let
\[ \Lambda^\pm := \partial \bSp^\pm \]
be the restriction of the $\bP$-sprig decompositions of $\bP = P \sqcup S^1_u$ to the universal circle $S^1_u = \partial \bP$. 
Since $\bSp^\pm$ is a proper muu pair by \Cref{prop:sprigsarespidery}, \Cref{cor:whenEmuuRestrictsToEspecial} implies that $\Lambda^\pm$ is an especial pair (in fact, this uses only that $\bSp^\pm$ satisfy property \ref{it:boundaryefficient}).

By \Cref{proposition:effExtExistence} there is a disc $\bQ = Q \sqcup S^1_u$, whose boundary is canonically identified with $S^1_u$, supporting an emuu pair of decompositions $\bcF^\pm$ with $\partial \bcF^\pm = \Lambda^\pm = \partial \bSp^\pm$.
By \Cref{corollary:EfficientExtensionsAction}, the universal circle action $\pi_1(M) \acts S^1_u$, which preserves $\Lambda^\pm$, extends uniquely to an action $\pi_1(M) \acts \bQ$ that preserves $\bcF^\pm$.
By \Cref{prop:emuutofoliations}, $\cF^\pm := \Mon(\bcF^\pm \cap Q)$ is a transverse pair of singular foliations of $Q$. The action on $\bQ$ restricts to an action $\pi_1(M) \acts Q$ that preserves $\cF^\pm$. 

Since $\bSp^\pm$ is a proper muu pair,  \Cref{prop:properstraightening} furnishes a \emph{straightening map}
\[ s\colon P_\lk \to Q \]
that is surjective and equivariant, and can be extended to an equivariant surjection
\[ \bs\colon P_\lk \sqcup S^1_u \to \bQ \]
by taking $\bs|_{S^1_u} = \mathrm{id}_{S^1_u}$.
Here, $P_\lk$ is the \emph{linked subset}, which is the closed subset of $P$ defined by
\[
P_\lk := \{p \in P \mid \partial \bSp^+(p) \text{ and } \partial \bSp^-(p) \text{ are linked}\} \subset P.
\]
We define the related \emph{linked subsets}
\[ \tM_\lk := \flow{P_\lk} \subset \tM \]
and
 \[ M_\lk = \pi_{\tM}(\tM_\lk) \subset M \]
where $\pi_{\tM}: \tM \to M$ is the covering map.

Our goal, now, is to construct a new quasigeodesic flow $\Psi$ on $M$ with flowspace $\pi_1(M) \acts Q$ and positive and negative leaf decompositions $\cF^\pm$, which we will show is pseudo-Anosov. We will build this flow in \Cref{section:StraighteningFlow} after developing some useful technology in Sections \ref{section:comparisonmaps}, \ref{section:strongdecompositions}, and \ref{section:productstructures}. In \Cref{section:StraightenedFlowsPA} we will prove the new flow is pseudo-Anosov.

\section{Comparison maps}\label{section:comparisonmaps}
In this section we will show that any quasigeodesic flow on a closed hyperbolic $3$-manifold can be fruitfully compared to the geodesic flow on the unit tangent bundle of its universal cover. The ``comparison maps'' we produce will be used in \Cref{section:strongdecompositions} to construct \emph{strong positive and negative decompositions} analogous to the strong stable and unstable foliations of a pseudo-Anosov flow. These will play an important role in constructing our new ``straightened flow'' in \Cref{section:StraighteningFlow}, and in showing that it is pseudo-Anosov in \Cref{section:StraightenedFlowsPA}.

The unit tangent bundle $T^1 M$ of a Riemanian manifold $M$ comes with a natural flow, the \emph{geodesic flow}, which moves each $v \in T^1 \tM$ at unit speed along the corresponding geodesic. This lifts to the geodesic flow on $T^1 \tM$. When $M$ is hyperbolic the geodesic flow is Anosov.

\begin{definition}\label{def:comparisonMap}
    Let $\Phi$ be a quasigeodesic flow on a closed hyperbolic $3$-manifold $M$. A \emph{comparison map} is a continuous, $\pi_1(M)$-equivariant map
	\[ R\colon \tM \to T^1\tM \]
	that takes each oriented flowline to the oriented geodesic between its endpoints.
	A comparison map is said to be
    \begin{itemize}
        \item \emph{monotone} if it takes each flowline monotonically to the corresponding geodesic, and
        \item \emph{flow-equivariant} if it satisfies
	\[ R(x \cdot t) = R(x) \cdot t \text{ for all } x \in \tM \text{ and } t \in \bbR, \]
	where the flow on the left is $\tPhi$ and the flow on the right is the geodesic flow.\qedhere
    \end{itemize}
\end{definition}

\begin{lemma}\label{lemma:ComparisonBDandUC}
    Let $R$ be a comparison map for a quasigeodesic flow on a closed hyperbolic $3$-manifold $M$. Then $R$ is uniformly continuous and moves points a bounded distance. That is, $d(x, \pi(R(x)))$ is uniformly bounded, where $\pi\colon T^1 \tM \to \tM$ is the bundle projection.
\end{lemma}
\begin{proof}
 Since $\pi\circ R$ is equivariant, $d(x, \pi(R(x)))$ descends to a continuous function on $M$. Since $M$ is compact, this function is bounded, so $R$ moves points a bounded distance.

 That $R$ is uniformly continuous follows from the observation that it descends to a map $M\to T^1M$ between compact spaces.
\end{proof}

\subsection{Building a flow-equivariant comparison map}\label{subsection:monotoneComparison}

Fix a quasigeodesic flow $\Phi$ on a closed hyperbolic $3$-manifold $M$. We will construct a comparison map, modify it to arrange for monotonicity, and then reparametrize $\Phi$ so the new comparison map is flow-equivariant. 

Given a flowline $\alpha$ in $\tM$, let $\alpha^g$ denote the oriented geodesic from $E^-(\alpha)$ to $E^+(\alpha)$, and let $\rho_\alpha\colon \alpha \to \alpha^g$ denote the nearest point projection.

Define a map 
\[
G\colon \tM\to \tM
\]
that takes each flowline to the corresponding geodesic by nearest point projection, i.e. $G(x) = \rho_{x\cdot \bbR}(x)$. This is $\pi_1(M)$-equivariant since $\pi_1(M)$ act by isometries; and continuous, since positive and negative endpoints of flowlines vary continuously.

This has a natural lift to the unit tangent bundle, the map
\[
R\colon \tM \to T^1 \tM 
\]
defined by taking $R(x)$ to be the unit tangent vector based at $G(x)$ that is tangent to $(x \cdot \bbR)^g$ and points from $E^-(x)$ to $E^+(x)$. This too is $\pi_1(M)$-equivariant, so it is a comparison map.

Note that $R$ may fail to be monotone because an orbit may ``backtrack" with respect to its corresponding geodesic. However, because the orbits of $\Phi$ are uniform quasigeodesics (\cite[Lem. 3.10]{Calegari}), there is a uniform constant $C>0$ such that for any $x\in \tM$,
\[
R(x\cdot C)>R(x).
\]
where we are orienting the lifted geodesic $R(x \cdot \bbR)$ with the geodesic flow. In other words, the duration of backtracks is uniformly bounded by $C$.

This allows us to perform a standard averaging trick to turn $R$ into a monotone comparison map (compare e.g. \cite[\S 1]{Gromov_geodesic}, \cite[\S 5.3]{LandryMinskyTaylor}). We define
\[
R_\mathrm{av}(x)=\frac{1}{C}\int_0^{C} R(x \cdot t)\,dt
\]
where we are using the Riemann integral and identifying $R(x \cdot \bbR) = R(x) \cdot \bbR$ with $\bbR$ using the geodesic flow. 
This is monotone increasing along each flowline because for any $\epsilon>0$ we have 
\begin{align*}
 R_\mathrm{av}( x \cdot\epsilon)-R_\mathrm{av}(x)&=\frac{1}{C}\left(\int_{\epsilon}^{\epsilon + C} R(x \cdot t)\,dt-\int_{0}^{C} R(x \cdot t)\,dt \right) \\
 &=\frac{1}{C}\left(\int_{C}^{\epsilon +C} R(x \cdot t)\,dt-\int_{0}^{\epsilon} R(x \cdot t)\,dt \right) \\
 &=\frac{1}{C}\int_{0}^{\epsilon}(R(x \cdot (t + C))- R(x \cdot t))\,dt>0.   
\end{align*}
This averaging gives us a continuous, $\pi_1(M)$-equivariant map 
\[
R_\mathrm{av}\colon \tM\to T^1 \tM,
\]
which is monotone on flowlines. Continuity follows from the fact that $R_\mathrm{av}$ is continuous on each flowline and that averaging over a fixed interval is a continuous operation on $C(\bbR,\bbR)$ in the compact open topology.

Thus we have:
\begin{proposition}
    Every quasigeodesic flow on a closed hyperbolic $3$-manifold admits a monotone comparison map.
\end{proposition}

We can use $R_\mathrm{av}$ to reparametrize the orbits of $\tPhi$ as follows: given $x \in \tM$ and $t \in \bbR$, let $(\tPhi')^t(x)$ be the unique point in the $\tPhi$-orbit through $x$ that satisfies $R_\mathrm{av}((\tPhi')^t(x)) = \Theta^t(R_\mathrm{av}(x))$. Here $\Theta$ denotes the geodesic flow. This gives a reparametrization of $\Phi$ for which $R_\mathrm{av}$ is flow-equivariant.

\begin{proposition} \label{prop:comparisonmapexists}
    Every quasigeodesic flow on a closed hyperbolic $3$-manifold can be reparametrized so that it admits a flow-equivariant comparison map.
\end{proposition}

In the sequel we will use $R$ to denote a flow-equivariant comparison map.

\section{Strong decompositions} \label{section:strongdecompositions}
For this section let us fix a quasigeodesic flow $\Phi$ with a flow-equivariant comparison map $R\colon \tM \to T^1 \tM \simeq T^1\bbH^3$ as furnished by \Cref{prop:comparisonmapexists}. The geodesic flow on $T^1\tM$ is an Anosov flow. We will pull back its strong stable and unstable foliations to construct ``strong positive and negative decompositions'' for $\tPhi$ that behave like the strong stable and unstable foliations of a pseudo-Anosov flow.

\subsection{Weak omnileaves and the geodesic flow}
Let $\tTheta$ be the geodesic flow on $T^1\tM$. This is an Anosov flow whose weak stable/unstable foliations will be denoted $\tcW^{\su}$ and whose strong stable/unstable foliations will be denoted $\tcW^{\ssuu}$. Let
\[ \epsilon^\pm: T^1\tM \to S^2_\infty \]
be the maps that take each unit tangent vector $v$ to the forwards/backwards endpoint of the unique oriented geodesic to which it is tangent. We will say that a vector $v$ \emph{points towards} $\epsilon^+(v)$ and \emph{away from} $\epsilon^-(v)$.

Each weak stable/unstable leaf of $\tTheta$ consists of all vectors pointing toward/away from a single point. That is,
\[ \tcW^{\su} = \{ (\epsilon^\pm)^{-1}(z) \mid z \in S^2_\infty \}. \]
By definition, the comparison map $R$ takes each flowline to the oriented geodesic between its endpoints; equivalently, $\epsilon^+(R(x)) = E^+(x)$ and $\epsilon^-(R(x)) = E^-(x)$ for all $x \in \tM$. Thus we can think of the weak positive and negative omnileaf decompositions of $\tPhi$ as the pullbacks of the weak stable and unstable decompositions of the geodesic flow $\tTheta$. That is,
\[ \tOm^\pm = \{R^{-1}(W) \mid W \in \tcW^\su \}\]

\subsection{Strong omnileaves}
Given a horosphere $H \subset \bbH^3$ centered at a point $z \in S^2_\infty$ the set of all unit vectors in $T^1 \bbH^3$ that are normal to $H$ and point toward $z$ is the corresponding \emph{stable horosphere}. The set of all unit vectors normal to $H$ that point away from $z$ is the corresponding \emph{unstable horosphere}.

The stable and unstable horospheres are the strong stable and unstable leaves of the geodesic flow $\tTheta$. That is,
\[ \tcW^\ssuu = \{ \text{stable/unstable horospheres in } T^1 \bbH^3 \}. \]

\begin{definition}
    Let $\Phi$ be a quasigeodesic flow on a closed hyperbolic $3$-manifold $M$ and let $R\colon \tM \to T^1 \tM \simeq \bbH^3$ be a flow-equivariant comparison map.
    
    Given a stable/unstable horosphere $H^\pm \in \tcW^\ssuu$, the preimage $R^{-1}(H^\pm) \subset \tM$ is called a \emph{strong positive/negative omnileaf}. If $H^\pm$ is centered at $z \in S^2_\infty$, we say that this strong positive/negative omnileaf is \emph{rooted at $z$}.
    
    The collections of all strong positive and strong negative omnileaves will be denoted by
    \[ \tOm^\ppmm := \{ R^{-1}(H^+) \mid H^\pm \in \tcW^\ssuu \}.\qedhere\]
\end{definition}

Since $R$ is $\pi_1(M)$-equivariant, $\tOm^\ppmm$ are $\pi_1(M)$-invariant decompositions of $\tM$. They are also flow-invariant:

\begin{lemma}
    Let $k \in \tOm^\ppmm$ and $t \in \bbR$. Then $k \cdot t \in \tOm^\ppmm$. 
\end{lemma}
\begin{proof}
    Suppose $k \in \tOm^\pp$; the negative case is similar.
    
    Let $H^+$ be the stable horosphere corresponding to $k$. Under the geodesic flow, this flows after time $t$ to another horosphere $H^+ \cdot t$; let $k'$ be the corresponding strong positive omnileaf. We will show that $k \cdot t = k'$.
    
    If $x \in k$, then $R(x) \in H^+$, and since $R$ is flow-equivariant, $R(x \cdot t) = R(x) \cdot t\in H^+ \cdot t$, which implies that $x \cdot t \in k'$. Hence $k \cdot t \subset k'$.
    
    If $x' \in k'$, then $R(x') \in H^+ \cdot t$, so by flow-equivariance $R(x' \cdot (-t)) = R(x') \cdot (-t) \in H^+$, which implies that $x' \cdot (-t) \in k$. Then $x' = x \cdot t$ for $x := x' \cdot (-t) \in k$, and hence $k \cdot t \supset k'$. Together, $k \cdot t = k$ as desired. 
\end{proof}

Note that the decompositions into strong omnileaves depend on the choice of $R$.

A theme of this paper is that facts about horospheres in $T^1 \bbH^3$ can be translated, via the map $R$, to facts about strong omnileaves in $\tM$. Here is one such fact about horospheres that we will use going forward. 

\begin{lemma}\label{lemma:horospherefacts}
    Let $H^-$ be an unstable horosphere in $T^1\bbH^3$ centered at $z^-\in S^2_\infty$. For each natural number $i$, let $H_i^+$ be a positive horosphere centered at $z_i^+\in S^2_\infty$ that intersects $H^-$. 
    Suppose $z^+\in S^2_\infty-\{z^-\}$, and let $H^+$ be the unique stable horosphere centered at $z^+$ intersecting $H^-$.

    Suppose $z_i^+\to z^+$, and that $(a_i^-)$ is a sequence of points in $S^2-z^+$ converging to a point $a^-\in S^2_\infty-\{z^+\}$, such that $a_i^-\ne z_i^+$ for all $i$. 
    Let $x_i$ be the unique vector in $H_i^+$ with $\epsilon^-(x_i)=a_i^-$, and $x$ the unique vector in $H^+$ with $\epsilon^-(x)=a^-$. Then $(x_i)$ converges to $x$.
\end{lemma}

\subsection{The strong omnileaves inside a weak omnileaf}
Each strong omnileaf $k$ is contained in a corresponding weak omnileaf $\widetilde{K} \subset \tM$, which projects to a corresponding omnileaf $K \subset P$. We say that $k$ ``lies over'' $K$. If $z = E^\pm(k)$, then $\widetilde{K} = (E^\pm)^{-1}(z)$ and $K = \nu(\widetilde{K}) = (e^\pm)^{-1}(z)$.

To help remember the terminology, note that the modifiers ``strong'' and ``weak'' indicate that an omnileaf lives in $\wt M$.

The main goal of this subsection is to show that each weak omnileaf is decomposed by strong omnileaves as a product.

\begin{lemma}\label{lemma:StrongOmnileavesFlowlines}
    Let $k$ be a strong (positive or negative) omnileaf.
    \begin{enumerate}
        \item If $\widetilde{K} \supset k$ is the corresponding weak omnileaf then each flowline $\alpha \subset \widetilde{K}$ intersects $k$ in a point.
        \item The flowspace projection $\nu: \tM \to P$ restricts to a bijection $\nu|_k: k \to K$ with image the corresponding omnileaf $K$.
    \end{enumerate}
\end{lemma}
\begin{proof}
    The two statements are equivalent; we will prove the first for a strong positive omnileaf $k$.

    Let $\alpha \subset \widetilde{K}$ be a flowline. The comparison map takes $\alpha$ to the lifted geodesic from $w = E^-(\alpha)$ to $z = E^+(\alpha) = E^+(k)$. Each lifted geodesic in $T^1\tM$ that ends at $z$ intersects each stable horosphere centered at $z$ in a single point, including the one $H^+$ that defines $k$. Since $R$ is flow-equivariant the restriction $R|_\alpha$ is a homeomorphism onto this oriented geodesic, so there is a unique point $x \in \alpha$ with $R(x) \in H^+$, and this is the unique point of intersection $x = \alpha \cap k$.
\end{proof}

We will upgrade item (2) of \Cref{lemma:StrongOmnileavesFlowlines} in \Cref{lemma:uniqueStrongOmnileafIntersection}. We will use the following:

\begin{lemma}\label{lemma:SingleOmnileafConvergence}
    Let $k \in \tOm^\ppmm$ be a strong positive/negative omnileaf. If $x_1, x_2, \dots$ is a sequence of points in $h$ such that $\nu(x_i)$ converges in $P$, then $(x_i)$ converges in $\tM$. In particular, if $\lim \nu(x_i) = p$ then $\lim x_i = \flow{p}\cap k$.
\end{lemma}
\begin{proof}
    Suppose $k\in \tOm^\mm$; the other case is symmetric.
    Let $x_1, x_2, \dots \in k$ such that $p_i := \nu(x_i)$ converges to $p\in P$, and let $x = \flow{p} \cap k$.
    
    \textbf{Claim:} $R(x_i)$ converges to $R(x)$. 
    
    By definition, $k = R^{-1}(H^-)$ for some unstable horosphere $H^-$ centered at $E^-(h)=e^-(p)$.
    For each $i$, $R(x_i)$ is the unique vector in $H^-$ with $\epsilon^+(R(x_i)) = E^+(x_i) = e^+(p_i)$. Note that $e^+(p_i)$ converges to $e^+(p)$, and $e^+(p) \neq e^-(p) $ because endpoints of flowlines are distinct. Letting $H^+$ be the stable horosphere centered at $E^+(x)=e^+(p)$ that intersects $H^-$,
    we see that $R(x_i)$ converges to the vector $r$ in $H^+$ with $\epsilon^-(r) = e^-(p)$, which is $R(x)$ as claimed. (We've used a special case of \Cref{lemma:horospherefacts} where $H_i^+$ is the stable horosphere centered at $e^+(p_i)$ and $a_i=e^-(p_i)=e^-(p)$ for all $i$).

    \textbf{Claim:} $\lim x_i = x$. 
    
    Since $\lim p_i=p$, the $x_i$ can only accumulate on points in $\flow{p}$. Since in addition $R(x_i)$ converges to $R(x)$, $(x_i)$ can only accumulate on points in $\flow{p}$ that $R$ maps to $R(x)$. Since $x$ is the only such point, it is the only possible accumulation point of $(x_i)$. Finally, since $R$ moves points a bounded distance (\Cref{lemma:ComparisonBDandUC}) and $(R(x_i))$ converges, $(x_i)$ is eventually contained in a compact set. We conclude that $x = \lim x_i$.  
\end{proof}

\begin{proposition}\label{prop:StrongOmnileavesProject}
    Let $k$ be a strong positive/negative omnileaf. Then $\nu|_k$ is a homeomorphism onto its image, the corresponding omnileaf $K\subset P$.
\end{proposition}

\begin{proof}
    By \Cref{lemma:StrongOmnileavesFlowlines}, $\nu|_k: k \to K$ is a bijection. Its inverse is the section $\sigma: K \to k$ defined by $\sigma(p) = \flow{p} \cap k$. Then \Cref{lemma:SingleOmnileafConvergence} gives that for any sequence of points $p_1, p_2, \cdots \in K$ converging to a point $p$, the points $x_i = \sigma(p_i)$ converge to $\flow{p} \cap k = \sigma(p)$. Hence $\sigma$ is continuous, so $\nu|_k$ is a homeomorphism.
\end{proof}

\begin{lemma}\label{lemma:uniqueStrongOmnileafIntersection}
    Let $K, L\subset P$ be positive and negative omnileaves, respectively,  with $K \cap L \neq \emptyset$. Then each strong positive omnileaf over $K$ intersects a unique strong negative omnileaf over $L$.
\end{lemma}

\begin{proof}
For a unit vector $v$ pointing along the oriented geodesic from $e^-(L)$ to $e^+(K)$, let $H^{\pm}_v$ be the stable/unstable horosphere containing $v$.

Let $k$ be a strong positive omnileaf over $K$, and $l$ be a strong negative omnileaf over $L$. Since $K$ and $L$ intersect in $P$, we can choose $a,b\in\langle K\cap L\rangle$ such that $a\in k$ and $b\in l$. 
Then 
 \[
 k\cap l
 =R^{-1}(H^+_{R(a)})\cap R^{-1}(H^-_{R(b)})
 = R^{-1}(H^+_{R(a)}\cap H^-_{R(b)}).
 \]
The intersection $H^+_{R(a)}\cap H^-_{R(b)}$ is nonempty if and only if $R(a)=R(b)$, so the same is true for $k\cap l$ by the equation above. 
This shows that the unique negative omnileaf over $L$ intersected by $k$ is $R^{-1}(H^-_{R(a)})$.
\end{proof}

\subsection{Weak, strong, and ---; omnileaves, leaves, and sprigs}
\begin{definition}
    Let $k$ be a strong positive/negative omnileaf, and let $\widetilde{K} \supset k$ be the corresponding weak omnileaf.
    
    For each weak sprig $\widetilde{K}' \subset \widetilde{K}$, $k \cap \widetilde{K}'$ is called a \emph{strong positive/negative sprig}.
    
    For each weak leaf $\widetilde{K}'' \subset \widetilde{K}$, $k \cap \widetilde{K}''$ is called a \emph{strong positive/negative leaf}.
    
    These give decompositions of $\tM$ that we denote
    \begin{align*}
    \tSp^\ppmm &:=\{\text{strong positive/negative sprigs}\}, \text{ and}\\
    \tcL^\ppmm &:=\{\text{strong positive/negative leaves}\}.\qedhere
    \end{align*}
\end{definition}

We can also characterize strong leaves directly:

\begin{lemma}
    The strong positive/negative leaves are the connected components of strong positive/negative omnileaves.
\end{lemma}
\begin{proof}
    This follows from \Cref{prop:StrongOmnileavesProject}. Let $k$, $\wt K$, and $K$ be a strong omnileaf and corresponding weak omnileaf and omnileaf. The flowspace projection restricts to a homeomorphism $\nu|_k: k \to K$, and therefore induces a bijection between the components of $k$ and the components of $K$. 
    
    By definition, a strong leaf in $k$ is $k \cap \wt K''$ for a weak leaf  $\wt K'' \subset \wt K$, which is any set of the form $\wt K'' = \flow{K''}$ for a leaf $K'' \subset K$, which is any component of $K$. Therefore an equivalent definition of a strong leaf in $k$ is ``$k \cap \flow{K''}$ for any component $K''$ of $K$.'' Since $k \cap\flow{A} = (\nu|_k)^{-1}(A)$ for any subset $A \subset K$, an equivalent definition is ``$\nu|_k^{-1}(K'')$ for any component $K''$ of $K$.'' Since $\nu|_k$ induces a bijection between components, this is the same as ``a component of $k$.''
\end{proof}

A direct characterization of strong sprigs is less natural.

\section{Product structures}\label{section:productstructures}

    Let $\Phi$ be a quasigeodesic flow on closed hyperbolic $3$-manifold $M$. A \emph{product structure} over a subset $A \subset P$ is a homeomorphism 
    \[ \Pi: A \times \bbR \to \flow{A} \]
    such that
    \begin{itemize}
        \item $\Pi$ takes $\{p\} \times \bbR$ to $\flow{p}$ for each $p \in A$, and
        \item $\Pi$ conjugates the flow $\tPhi$ on $\flow{A}$ with the ``vertical flow'' on $A \times \bbR$, i.e. the flow defined by $(p, t) \cdot s = (p, t+s)$.
    \end{itemize}

    If $\Phi$ comes with a flow-equivariant comparison map $R$, we will say that \emph{strong positive sprigs are horizontal} in a product structure $\Pi$ over $A$ if for each positive $P$-sprig $K$, and each $y \in \bbR$, $\Pi((K \cap A) \times \{y\})$ is contained in a strong positive sprig.

Note that one could also ask for the weaker condition that strong positive \emph{leaves} are horizontal, or the stronger condition that strong positive \emph{omnileaves} are horizontal.

\begin{example}
    The suspension flow $\Phi$ of a pseudo-Anosov homeomorphism $\phi: \Sigma \to \Sigma$ has a natural product structure, defined over the whole flowspace, in which both strong positive and strong negative leaves are horizontal. This comes from identifying the flowspace with a lift $\wt \Sigma$ of the surface to the universal cover of the mapping torus, which we can identify with $\wt\Sigma\times \bbR$. Each strong positive/negative leaf lies in some fiber $\widetilde{\Sigma} \times \{t\}$, hence is horizontal.
    
    Since suspension flows have \emph{no perfect fits}, the omnileaves and leaves of $\Phi$ coincide. So, in fact, all strong positive and negative omnileaves are horizontal in this product structure. This is a very special case and we do not claim such a product structure always exists---in this paper, we will only produce product structures in which one of the strong sprig decompositions is horizontal at a time.
\end{example}

The following, which is the main result of this section, will be useful for completing the construction of the straightened flow.

\begin{theorem}\label{theorem:ProductStructure}
    Let $\Phi$ be a quasigeodesic flow on a closed hyperbolic $3$-manifold $M$ with a flow-equivariant comparison map $R$. Then there is a product structure
    \[ \Pi: P_{\lk} \times \bbR \to \tM_{\lk} \]
    over $P_\lk$ in which strong positive sprigs are horizontal.
\end{theorem}

We emphasize that it will be important to find a product structure in which strong positive \emph{sprigs} are horizontal, not just strong positive leaves.

The proof is technical, and the reader may wish to skip to \Cref{section:StraighteningFlow} on a first pass. The argument is similar to showing that a fiber bundle over a contractible CW complex is trivial.

\subsection{Product structures and sections}
The proof of \Cref{theorem:ProductStructure} will occupy the remainder of this section.

Note that a product structure $\Pi: A\times \bbR \to \flow{A} \subset \tM$ over a subset $A \subset P$ determines a corresponding continuous section of the flowspace projection $\nu: \tM \to P$ over $A$: the \emph{zero section} $\sigma: A \to \flow{A} \subset \tM$ defined by $\sigma(a) = \Pi(A, 0)$. Conversely, let us show that any continuous section over a closed subset determines a product structure:

\begin{proposition}\label{proposition:ProductStructureFromSection}
    Let $\Phi$ be a quasigeodesic flow on a closed hyperbolic $3$-manifold $M$.
    
    If $\sigma: A \to \tM$ is a section of the flowspace projection $\nu: \tM \to P$ defined over a closed subset $A$, then $\Pi(a, t) = \sigma(a) \cdot t$ defines a product structure
    \[ \Pi: A \times \bbR \to \flow{A}\]
    over $A$.

    If $\Phi$ is equipped with a flow-equivariant comparison map, then strong positive sprigs are horizontal under $\Pi$ if and only if $\sigma$ takes positive sprigs to strong positive sprigs.
\end{proposition}

Here, we say that $\sigma$ takes positive sprigs to strong positive sprigs if $\sigma(K \cap A)$ is contained in a strong positive sprig for each positive sprig $K \in \Sp^+$.

\begin{proof}
    It is immediate from the definition that $\Pi$ is a continuous bijection that conjugates the vertical flow to $\Phi$. To see that $\Pi$ is a product structure, it remains to show that $\Pi^{-1}$ is continuous. The claim about horizontality is straightforward.

    Note that 
\[
\Pi^{-1}(x)=(\nu(x),h(x)),
\]
where $h(x)$ is the time it takes to flow from $\sigma(\nu(x))$ to $x$. Hence it suffices to show that $h\colon \flow{A}\to \bbR$ is continuous. 

Suppose $x_1,x_2,\dots$ is a sequence of points in $\flow{A}$ converging to $x\in \flow{A}$. Let $y_i=\sigma(\nu(x_i))$ for all $i$ and $\sigma(\nu(x))=y$. Since $\nu$ and $\sigma$ are continuous, $\lim y_i=y$. Since $x_i\to x$ and $y_i\to y$, and the foliation of $\tM$ by flowlines is topologically conjugate to the foliation of $\bbR^3$ by vertical lines, the flow segments between $y_i$ and $x_i$ Hausdorff converge to the flow segment between $y$ and $x$. It follows that the parametrized lengths of these flow segments converge to $h(x)$, i.e. $\lim h(x_i)=h(x)$. This shows $h$ is continuous.
\end{proof}

\begin{example}
Let $K$ be a positive or negative omnileaf, let $k$ be any strong omnileaf over $K$, and let $\sigma\colon K\to \flow{K}$ be the section $\sigma(p)=\flow{p}\cap k$. Note that $\sigma$ is continuous by \Cref{lemma:SingleOmnileafConvergence}. Then \Cref{proposition:ProductStructureFromSection} shows that the flow on the corresponding weak leaf $\wt K$ is conjugate to the vertical flow on $K \times \bbR$ in such a way that strong omnileaves correspond to horizontal sets $K \times \{h\}$. 
\end{example}

For the remainder of this section, let us now fix a quasigeodesic flow $\Phi$ on a closed hyperbolic $3$-manifold $M$ with a flow-equivariant comparison map $R$. 
We will find a section $\sigma: P_\lk \to \tM_\lk$ taking positive sprigs into strong positive sprigs. Then \Cref{proposition:ProductStructureFromSection} will produce a product structure in which strong positive leaves are horizontal, proving \Cref{theorem:ProductStructure}.

\subsection{Saturating omnileaves}

The basic building block will be the ability to build appropriate sections over a set obtained by saturating a subset of a negative sprig by positive sprigs. The following lemma is a bit more general than what we need.

\begin{lemma}\label{lemma:SectionOverOmnileaves}
    Let $L$ be a negative omnileaf and let $A$ be the saturation of $L$ by positive omnileaves (that is, $A=\Om^+(L))$. Then there is a section over $A$ that takes each positive omnileaf to a strong positive omnileaf.
\end{lemma}

\begin{proof}
    Pick a strong negative omnileaf $l$ in $\flow{L}$, and let $X=\Om^\pp(l)$ be the saturation of $l$ by strong \emph{positive} omnileaves. Note that $\nu(X) = A$, and in fact $X$ intersects each flowline over $A$ in exactly one point. This is because for each positive omnileaf in $A$, there is exactly one strong positive omnileaf in $X$ (\Cref{lemma:uniqueStrongOmnileafIntersection}), and $\nu$ identifies the two bijectively (\Cref{lemma:StrongOmnileavesFlowlines}).
    Thus we can define a function 
    \begin{align*}
        \sigma: A &\to \flow{A}\\
         a&\mapsto \flow{a} \cap X,
    \end{align*}
    which takes each positive omnileaf in $A$ to a strong positive omnileaf. 
    To complete the lemma it remains to show that $\sigma$ is continuous. The argument is an extension of the one in \Cref{lemma:SingleOmnileafConvergence}, which implies that $\sigma$ is continuous over $L$.

    Let $a_1, a_2, \cdots \in A$ be a sequence of points that converge to a point $a  \in A$. For each $i \ge1$,   let $x_i = \sigma(a_i)$ and let $x = \sigma(a)$. We must show that $\lim x_i = x$.

    \textbf{Claim 1:} $R(x_i)$ converges to $R(x)$.

    Let $H^-_l$ be the unstable horosphere corresponding to the strong negative omnileaf $l = R^{-1}(H_l^-)$, which is centered at $z_l := e^-(L)$. For each $i$, let $H_i^+$ be the stable horosphere containing $R(x_i)$, and note that $H_i^+$ is centered at $e^+(a_i)$. Since by construction the strong positive omnileaf $R^{-1}(H^+_i)$ through $x_i$ intersects $l$, each $H^+_i$ intersects $H^-_l$. Similarly, let $H^+$ be the stable horosphere containing $R(x)$, which is centered at $e^+(a)$ and also intersects $H^-_l$.

    Since $a_i \to a$, it follows that $e^+(a_i) \to e^+(a)$. 
    Notice as well that $\epsilon^-(R(x_i)) = e^-(a_i) \to e^-(a) \neq e^+(a)$. Hence \Cref{lemma:horospherefacts} implies that $R(x_i)$ converges to the vector $r$ in $H^+$ with $\epsilon^-(r) = e^-(a)$, which is $R(x)$ as desired.

    \textbf{Claim 2:} $x_i$ converges to $x$.

    Since $R(x_i)$ converges to $R(x)$, and $R$ moves points a bounded distance (by \Cref{lemma:ComparisonBDandUC}), the $x_i$ are eventually contained in a compact set. Hence they must accumulate on a nontrivial subset of $R^{-1}(R(x))$. Since $a_i \to a$, the $x_i$ can only accumulate on $\flow{a}$. There is only one point in $\flow{a}$ that maps to $R(x)$, namely $x$. Hence $\lim x_i = x$. 
\end{proof}

\subsection{Finding horizontal product structures} \label{subsection:findingproductstructures}
For brevity, let us say that a section $\sigma: A \to \tM$ over a closed set $A \subset P$ is \emph{good} if for each positive $P$-sprig $K$, $\sigma(K)$ is contained in a strong positive sprig. To complete the proof of \Cref{theorem:ProductStructure} we need to find a good section over $P_\lk$.

We will build a good section over $P_\lk$ inductively using the structure of the decomposition space $\bSp^+$ of positive sprigs, and in particular the structure of separation intervals. 

We will need to be mindful of a subtlety here:  $\bSp^+$ is a decomposition of $\bP$, but our sections will be defined on subsets of $P$.
Given a subset $\cA \subset \bSp^+$, we will use the phrase \emph{a good section over $\cA$} to refer to a good section $|\cA|\cap P\to \tM$.

The construction has an inductive flavor. The basic piece is the following:

\begin{lemma}\label{lemma:SectionForCrossable} 
    Let $\cA \subset \bSp^+$, and suppose that there is some negative $P$-sprig that intersects every element of $\cA$. Then there is a good section over $\cA$.
\end{lemma}
\begin{proof}
    This is a special case of \Cref{lemma:SectionOverOmnileaves}.
\end{proof}

The following will enable the inductive piece:

\begin{lemma}\label{lemma:ConcatenatingSections}
	Let $\cA, \cB \subset \bSp^+$ be sets of positive sprigs such that $\cA \cap \cB$ is a single sprig. If there are good sections over $\cA$ and $\cB$ then there is a good section over $\cA \cup \cB$. This may be chosen to agree with any given good section over $\cA$.
\end{lemma}
\begin{proof}
    Let $\bK$ be the positive sprig in the intersection, and let $K=\bK\cap P$. Given good sections $\sigma_\cA: |\cA|\cap P \to \tM$ and $\sigma_\cB: |\cB|\cap P \to \tM$ we can find a time $T \in \bbR$ such that $\sigma_\cB(p) \cdot T = \sigma_\cA(p)$ for all $p \in K$ (it is possible that $K=\emptyset$, but then there is no problem). We see that $\sigma'_\cB(\cdot) = \sigma_\cB(\cdot) \cdot T$ defines a new good section over $|\cB|$ that agrees with $\sigma_\cA$ on $K$, so we can combine them to obtain a good section over $|\cA \cup \cB|$.
\end{proof}

Let $\bSp^+_\star \subset \bSp^+$ be the set of positive sprigs $K \in \bSp^+$ with $\partial K$ disconnected, and note that $P_\lk \subset |\bSp^+_\star|$. 
Note also that the interior of every separation interval is contained in $\bSp^+_\star$, since a sprig separates $P$ if and only its set of ends is disconnected.

\begin{proposition}\label{proposition:GoodSectionsForSegments}
	Let $\bK_\alpha, \bK_\omega\in \bSp^+_\star$. Then there is a good section over $[\bK_\alpha, \bK_\omega]$.
\end{proposition}

To prove this we will build a cover of such a segment by sets that admit good sections, and stitch these together with \Cref{lemma:ConcatenatingSections}.

\begin{proof}
    
    Fix some $\bK \in (\bK_\alpha, \bK_\omega]$. We will show that there is a good section over $[\bK', \bK]$ for some $\bK' \in [\bK_\alpha, \bK)$. Let $U_\alpha$ be the component of $\bP - \bK$ that contains $\bK_\alpha$. See \Cref{figure:PSoverSegment}. Then $U_\alpha \cap S^1_u = (k, k')$ for ends $k, k' \in \bK \cap S^1_u$. By \cite[Lemma 7.8]{Frankel_closing}, there is a negative leaf  $L$ whose closure separates $k$ from $k'$ and intersects $\bK\cap P$ nontrivially. 
    Using the nesting property, one can find $\bK'$ between $\bK_\alpha$ and $\bK$ such that every positive sprig in $[\bK', \bK]$ intersects $L$. 
    Then \Cref{lemma:SectionForCrossable} provides a good section over $[\bK', \bK]$.
	
	\begin{figure}[ht]
		\centering
		\begin{minipage}{.49\textwidth}
			\centering
			\begin{overpic}[scale=.95]{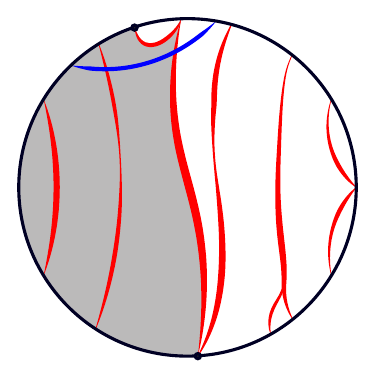}
				\tiny
				
				\put(15,40){\textcolor{red}{$\bK_\alpha$}}
				\put(82,40){\textcolor{red}{$\bK_\omega$}}
				
				\put(26,50){\textcolor{red}{$\bK'$}}
				\put(52.5,50){\textcolor{red}{$\bK$}}
				\put(75,50){\textcolor{red}{$\bK''$}}
				
				\put(35,85){\textcolor{blue}{$L$}}
				
				\put(51, 1){$k'$}
				\put(35, 95){$k$}
				
				\put(40,20){$U_\alpha$}
			\end{overpic}
		\end{minipage}
		\begin{minipage}{.49\textwidth}
			\centering
			\begin{overpic}[scale=.95]{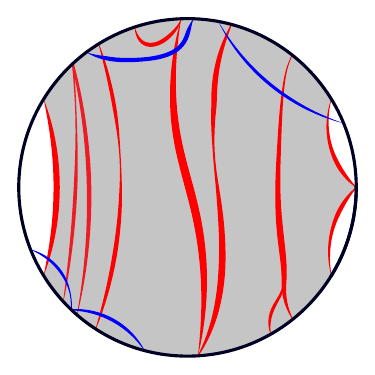}
				\tiny
				
				\put(7,48){\textcolor{red}{$\bK_\alpha$}}
				\put(87,49){\textcolor{red}{$\bK_\omega$}}				
		
			\end{overpic}
		\end{minipage}
		\caption{Product structures over segments.} \label{figure:PSoverSegment}
	\end{figure}
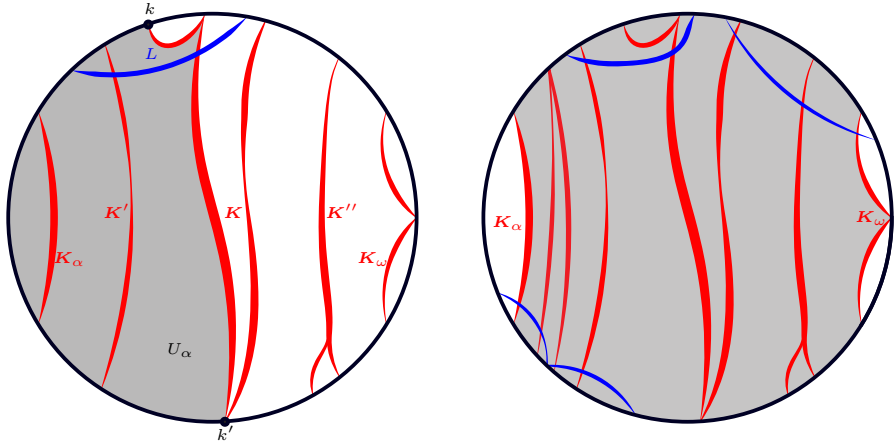
	
	A similar argument shows that for each $\bK \in [\bK_\alpha, \bK_\omega)$ one can find a good section over $[\bK, \bK'']$ for some $\bK'' \in (\bK, \bK_\omega]$. Also, for each $\bK \in (\bK_\alpha, \bK_\omega)$, one can combine these subsegments using \Cref{lemma:ConcatenatingSections} to find a subsegment $[\bK', \bK'']$ whose interior contains $\bK$, that admits a horizontal product structure.
	
	This shows that $[\bK_\alpha, \bK_\omega]$ has an open cover by separation intervals that admit horizontal product structures. Passing to a finite subcover and applying \Cref{lemma:ConcatenatingSections} again yields a horizontal product structure over $[\bK_\alpha, \bK_\omega]$.
\end{proof}

\begin{remark}
	The set of separation intervals endows $\bSp^+$ with the structure of an order tree as defined by Gabai-Kazez \cite{GabaiKazez}. Each segment is order-isomorphic to an interval in $\bbR$, but it is not clear that $\bSp^+$ is a countable union of segments---the key point is the possibility of uncountably many one-ended sprigs. As such we do not know whether $\bSp^+$ is an $\bbR$-order tree as defined by Gabai-Kazez. However, one can show that $\bSp^+_\star$ is an $\bbR$-order tree.
\end{remark}

\begin{lemma}[cf. \cite{GabaiKazez}, Proposition~3.1]\label{lem:ascendingunion}
	There is a sequence of subsets $T_1, T_2, \cdots \subset \bSp^+$ such that
	\begin{enumerate}[label=(\arabic*)]
		\item $T_1 = s_1$ is a separation interval,
		\item for each $i$, $T_i = T_{i - 1} \cup s_i$, where $s_i$ is a separation interval that intersects $T_{i-1}$ at exactly its initial point $\iota(s_i)$, 
		\item $\bSp^+_\star\subset\bigcup_{i=1}^\infty T_i$, and 
        \item for any $K, L \in \bSp^+_\star$, there is a finite $n$ such that $[K,L]\in T_n$.
	\end{enumerate}
\end{lemma}
\begin{proof}
	Let $\{p_i\}_{i=0}^\infty$ be a countable dense subset of 
    $\bP$. For each $i$, let $\bK_i = \bSp^+(p_i)$ be the positive sprig through $p_i$.
	
	Let $T_1 = s_1 = [\bK_0, \bK_1]$. This is closed and \emph{convex}, by which we mean it contains the separation interval between any two of its points.
	
	Assume that $T_{i-1}$ is closed and convex. Let $j(i) \geq i$ be the least index such $\bK_{j(i)}$ is not contained in $T_{i-1}$. Since $T_{i-1}$ is closed and convex, $T_{i-1} \cap [\bK_0, \bK_{j(i)}] = [\bK_0, \bL_i]$ for some $\bL_i$. Let $\sigma_i = [\bL_i, \bK_{j(i)}]$, and $T_i = T_{i-1} \cup \sigma_i$. One easily sees that this is closed and convex, and $\sigma_i \cap T_{i-1} = \bL_i = \iota(\sigma_i)$.
	
	To show (3), let $\bK \in \bSp^+_\star$. This has at least two complementary components in $\bP$ since $\del \bK$ is disconnected. Choose a $p_n$ such that $p_0$ and $p_n$ do not lie in the same complementary component of $\bK$. Then $\bK$ is contained in $[\bK_0, \bK_n]$, and hence in $\bigcup_{i=1}^\infty T_i$. 

    For (4), let $\bK, \bL\in \bcD_2^+$. 
    There exists an index $i(\bK)$ (resp. $i(\bL)$) so that $p_{i(\bK)}$ (resp. $p_{i(\bL)}$) lies in a complementary component of $\bK$ (resp. $\bL$) not containing $p_0$. Then $\bK\in [\bK_0, \bK_{i(\bK)}]$, and $\bL\in [\bK_0, \bK_{i(\bL)}]$. Given a positive sprig that separates $\bK$ from $\bL$, at least one of the following holds:
    \begin{itemize}
    \item the sprig separates 
    $\bK_{i(\bK)}$ from $\bK_0$,
    \item the sprig separates 
    $\bK_{i(\bL)}$ from $\bK_0$, or
    \item the sprig is equal to $\bK_0$.
    \end{itemize}
    This means we have $[\bK, \bL]\subset [\bK_0, \bK_{i(\bK)}]\cup [\bK_0, \bK_{i(\bL)}]$. In particular $[\bK, \bL]\subset T_{\max\{i(\bK),i(\bL)\}}$, establishing (4).
\end{proof}

\begin{corollary}\label{cor:goodsectionoverPlk}
	There is a good section over $P_\lk$. 
\end{corollary}
\begin{proof}
Using the notation from \Cref{lem:ascendingunion}, let $T=\bigcup_{i=1}^\infty T_i$. 
By \Cref{proposition:GoodSectionsForSegments}, there is a good section $\sigma_1$ over $|T_1|$. Assuming there is a good section $\sigma_i$ over $|T_i|$, we can use \Cref{lemma:ConcatenatingSections} to find a good section $\sigma_{i+1}$ over $|T_{i+1}|$ such that $\sigma_{i+1}|_{|T_i|}=\sigma_i$. By induction, we have good sections $\sigma_n\colon |T_n|\cap P\to \tM$ for all $n$, any two of which agree on the intersection of their domains. Thus we can define a function
\[
\sigma\colon |T|\cap P\to \tM
\]
by sending $x\in |T|$ to $\sigma(x) := \sigma_n(x)$, where $n$ is large enough so that $x\in |T_n|$. To finish the proof, it suffices to show that the restriction of $\sigma$ to $P_\lk$ is continuous.

Let $(p_n)_{n=1}^\infty$ be a sequence of points in $P_\lk$ converging to a point $p\in P_\lk$. We will show that $\lim_{n\to\infty} (\sigma(p_n))=\sigma(p)$. 

Let $\bK=\bcD^+(p)$, and $\bK_i=\bcD^+(p_i)$.
By \cite[Lemma 7.3]{Frankel_closing}, the sequence $(p_n)$ intersects at most finitely many complementary regions of $\bK$. 
Passing to a subsequence, we can assume that the sequence intersects at most one complementary region of $\bK$.
Given this assumption, we claim that all but finitely many of the $\bK_i$'s lie in the interval $[\bK, \bK_1]$.

Suppose otherwise. Then, by passing to a subsequence and relabeling, we can assume none of $\bK_1, \bK_2,\bK_3,\dots$ lie in $[\bK_1, \bK]$, and further that there is a single interval $I$ of $S^1_u-(\bK\cup \bK_1)$ such that  $\del\bK_i \subset \overline I$ for all $i$ (see \Cref{fig:contradiction}).
By upper semicontinuity (\Cref{theorem:USCConditions}), we have $\limsup\del \bK_i\subset \del \bK$, so in fact the $\del \bK_i$'s  Hausdorff converge to the endpoint $a$ of $I$ lying in $\del \bK$. 
Set $\bL_i=\bcD^-(p_i)$, and $\bL=\bcD^-(p)$. Since $p_i\in P_\lk$ for all $i$, the convex hull of $\del K_i$ in $S^1_u$ contains at least one point of $\del L_i$ for all $i$. Hence $a\in \limsup \del \bL_i\subset \del \bL$ (again using upper semicontinuity). Hence $\del \bK\cap \del \bL\ne \emptyset$, contradicting \Cref{lemma:SprigEndsDisjoint}.

Hence the $\bK_i$'s eventually lie in $[\bK,\bK_0]$. By \Cref{lem:ascendingunion}, we have $[\bK,\bK_0]\subset T_n$ for some $n$, and the result now follows from the continuity of $\sigma_n$.
\end{proof}

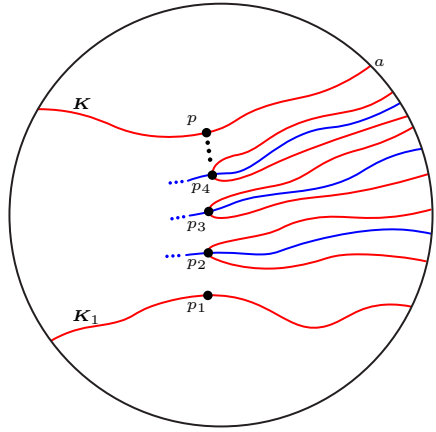
\begin{figure}
    \centering
    \begin{overpic}[]{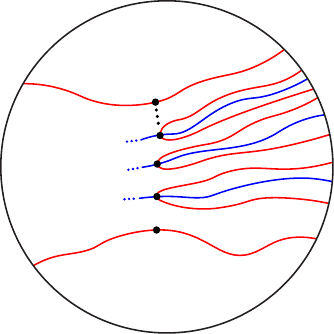}
				\tiny
				\put(42,72){$p$}
                \put(42,28){$p_1$}
				\put(42, 38){$p_2$}
                \put (42, 47){$p_3$}
                \put (43, 56){$p_4$}
                \put (86, 85){$a$}
                \put (15, 75) {$\bK$}
                \put (15, 25) {$\bK_1$}
	\end{overpic}
    \caption{Cartoon of a counterfactual situation in the proof of \Cref{cor:goodsectionoverPlk}; the blue leaves are the $\bL_i$.}
    \label{fig:contradiction}
\end{figure}

This completes the proof of \Cref{theorem:ProductStructure}.

 \section{Straightening a quasigeodesic flow}\label{section:StraighteningFlow}

In this section we show that the restriction of $\Phi$ to a certain subset $M_\lk\subset M$ is semiconjugate to a new quasigeodesic flow $\Psi$ on $M$, and show the semiconjugacy $M_\lk\to M$ is homotopic to the inclusion map. This construction will use positively horizontal product structures as well as straightening maps.

Many of the spaces and maps defined in this section fit into the diagram in \Cref{bigdiagram}, and the reader may find it helpful to consult that diagram while absorbing the definitions.

\subsection{Constructing the straightened flow} \label{subsec:buildingtheflow}

Fix a quasigeodesic flow $\Phi$ on a closed hyperbolic $3$-manifold $M$.
Recall that $\bSp^\pm$ denote the sprig decompositions of $\bP = P \sqcup S^1_u$. 
By \Cref{prop:comparisonmapexists} we can assume, after reparametrizing, that $\Phi$ admits a flow-equivariant comparison map $R$.

Let $Q$ be the corresponding straightened flowspace (recall the discussion in \Cref{sec:straighteningtheflowspace}), with singular foliations $\cF^\pm$, and straightening map
\[ s: P_{\lk} \to Q. \]
By \Cref{theorem:ProductStructure} we can fix a product structure
\[ \flow{P_{\lk}} \simeq P_{\lk} \times \bbR \]
over $P_{\lk}$ in which strong positive sprigs are horizontal. We define a map
\begin{align*}
	\tS\colon P_{\lk} \times \bbR 	&\to Q \times \bbR \\
 	(p, h) 						&\mapsto (s(p), h).
\end{align*}
Note that $\tS$ is closed map because $s$ is closed by \Cref{prop:properstraightening}, and hence $\tS$ is a quotient map.
Also, the $\pi_1(M)$-action on $\tM$ restricts to an action on $P_{\lk} \times \bbR$. We will use these facts to build a nice action of $\pi_1(M)$ on $Q\times \bbR$.

\begin{lemma}\label{lem:actionprops}
The action of any $g\in \pi_1(M)$ on $P_{\lk} \times \bbR$ sends point preimages of $\tS$ to point preimages of $\tS$. This induces an action of $\pi_1(M)$ on $Q\times \bbR$ by homeomorphisms.

\end{lemma}
\begin{proof}
    For each point $z=(q, t) \in Q \times \bbR$, its preimage can be written as $\tS^{-1}(z) = s^{-1}(q) \times \{t\}$. Here, $s^{-1}(q) = \bK \cap \bL$, where $\bK$ and $\bL$ are positive and negative sprigs, respectively. Since strong positive sprigs are horizontal, we have
    \[
    g\tS^{-1}(z)=(g\bK\cap g\bL)\times \{t'\}
    \]
    for a single $t'$. Note that $g\bK\cap g\bL$ is an entire point preimage of $s$, since it is the intersection of two sprigs whose ends link. By the definition of $\tS$, then, $g\tS^{-1}(z)$ is an entire point preimage of $\tS$.
    Since $\tS$ is a quotient map, the  action of $\pi_1(M)$ on $P_\lk\times \bbR$ by homeomorphisms descends to one on $Q\times \bbR$.
    \end{proof}

    Taking the quotient by this $\pi_1(M)$-action produces a covering of a closed $3$-manifold:

\begin{lemma}\label{lem:quotientisacoveringmap}
	The action of $\pi_1(M)$ on $Q \times \bbR$ is cocompact, free and properly discontinuous, and has Hausdorff quotient.
\end{lemma}

To be clear, by ``free and properly discontinuous" we mean that each element has a neighborhood disjoint from all its translates by nontrivial group elements.

\begin{proof}
Since $M$ is compact, there is a compact fundamental domain $A$ for the action $\pi_1(M)\acts \tM$. Since $\flow{P_{\lk}}$ is closed, the intersection $A_{\lk} := A \cap \flow{P_{\lk}}$ is compact. Also, the $\pi_1(M)$-orbit of every point in $\flow{P_{\lk}}$ intersects $A_{\lk}$. Then $\tS(A_{\lk})$ is a compact set that intersects the $\pi_1(M)$-orbit of every point in $Q \times \bbR$, so $\pi_1(M)$ acts cocompactly on $Q\times \bbR$.

    If the action is not free and properly discontinuous, then we can find a sequence of points $y_i \in Q \times \bbR$ and nontrivial elements $g_i \in \pi_1(M)$ such that $\lim y_i = y = \lim g_i \cdot y_i$. Then 
    \[
    \lim R_s(y_i) = R_s(y) = \lim R_s(g_i \cdot y_i) = \lim g_i R_s(y_i),
    \]
    contradicting the fact that $\pi_1(M) \acts T^1\bbH^3$ is free and properly discontinuous.

    If the quotient $\lquotient{\pi_1(M)}{Q \times \bbR}$ is not Hausdorff then one can find points $y, y' \in Q \times \bbR$ in distinct orbits and a sequence of points $y_i \in Q \times \bbR$ and pairwise distinct elements $g_i \in \pi_1(M)$ such that $\lim y_i = y$ and $\lim g_i \cdot y_i = y'$. Then 
    \[
    R_s(y) =\lim R_s(y_i)
    \]
    and
    \[
    R_s(y')=\lim R_s(g_i y_i) = \lim g_i R_s(y_i).
    \]
    If $R_s(y) = R_s(y')$ then this contradicts the fact that $\pi_1(M) \acts T^1\tM$ is free and properly discontinuous. If $R_s(y) \neq R_s(y')$ and these are in distinct orbits, then this contradicts the fact that $\lquotient{\pi_1(M)}{\tM}=T^1 M$ is Hausdorff. 
    
    Finally, if $R_s(y) \neq R_s(y')$ and these are in the same orbit then let $g$ be the element of $\pi_1(M)$ with $g \cdot R_s(y') = R_s(y)$ and let $h = gg_i$ for each $i$. Then
    \[
    \lim R_s(y_i) = R_s(y) = \lim h_i R_s(y_i),
    \]
    contradicting free and proper discontinuity once again.
\end{proof}

Now let
\[
\tM_s := Q \times \bbR \,\, \text{ and } \,\,\tM_\lk:= P_\lk\times \bbR.
\]
By \Cref{lem:quotientisacoveringmap}, the quotient
\[
M_s := \lquotient{\pi_1(M)}{\tM_s}
\]
is a closed $3$-manifold. Letting
\[
M_\lk:=\lquotient{\pi_1(M)}{P_\lk\times \bbR},
\]
The map $\tS\colon \tM_\lk\to \tM_s$ induces a map
\[
S\colon M_{\lk}\to M_s.
\]
The foliation of $M_s=Q \times \bbR$ by oriented vertical lines is $\pi_1(M)$-invariant, so it descends to an oriented foliation on $M_s$. To get a flow, we need to parametrize the leaves of this foliation. We will do this by pulling back the parametrization of the geodesic flow under a suitable map, which we construct in the next lemma. Before the statement, we set 
\[
R_\lk:=R|_{\tM_\lk}.
\]

\begin{lemma}
    The restricted comparison map $R_\lk$ induces a $\pi_1(M)$-equivariant map 
    \[
    R_s: M_s \to T^1\tM
    \]
    that sends vertical lines $\{q\}\times \bbR\subset Q\times \bbR=M_s$ monotonically to lifted geodesics.
\end{lemma}

\begin{proof}
    Let $z=(q,t)\in M_s=Q\times \bbR$, and write $\tS^{-1}(z)=(\bK\cap \bL)\times \{t\}$ for positive and negative sprigs $\bK$ and $\bL$. 
    Since strong positive sprigs are horizontal, $\tS^{-1}(y)$ is contained in a \emph{single} strong positive sprig. Hence  $R(\tS^{-1}(y))$ is a single point, so $R_\lk$  induces a continuous map $R_s: Q \times \bbR \to T^1\tM$ since $\tS$ is a quotient map. This map is equivariant because $\tS$ and $R$ are.

    The monotonicity statement follows from the fact that $R$ is monotonic along flowlines, the restriction of $\tS$ to a single flowline is a bijection.
\end{proof}

Just as in \Cref{subsection:monotoneComparison}, we can use the map $R_s$, together with the geodesic flow on $T^1\bbH^3$, to parametrize the vertical foliation of $Q\times \bbR$ and obtain a flow $\tPhi_s$ on $\tM_s$ that descends to a flow $\Phi_s$ on $M_s$.

\begin{lemma}\label{lem:semiconjugacy1}
The map $S$ is a semiconjugacy from $\Phi_\lk$ to $\Phi_s$, i.e. for all $x\in M_\lk$, $t\in \bbR$,
\[
\Phi_\lk^t(x)\mapsto \Phi^t_s(S(x)).
\]
\end{lemma}
\begin{proof}
    It suffices to show that $\tS(\tPhi_\lk^t(x))=\tPhi_s^t(x)$ for all $x\in \tM, t\in \bbR$. 
    This is a straightforward consequence of the fact that $\tPhi_s$ was parametrized using the map $R_s$, which was induced by the map $R_L$. Indeed, fixing $x\in M_\lk$ and $t\in \bbR$, we need $\tS(\tPhi_\lk^t(x))=\tPhi^t_s(\tS(x))$. We have
    \[
    R_s(\tS(\tPhi^t_\lk(x)))=R_\lk(\tPhi^t_\lk(x))=\Theta^t(R_\lk(x))=\Theta^t(R_s(\tS(x))=R_s(\tPhi_s^t(\tS(x))).
    \]
    To conclude, note that $R_s$ maps $\tPhi_s$-flowlines monotonically to lifts of geodesics.
\end{proof}

The manifold $M_s$ is hyperbolic by geometrization \cite{Perelman1, Perelman2, Perelman3}. By construction we have a natural identification of $\pi_1(M_s)$ with $\pi_1(M)$, so Mostow Rigidity furnishes a unique isometry 
\[MR\colon M_s\to M\]
inducing this identification.

\begin{definition}
We denote the flow on $M$ obtained by pushing forward $\Phi_s$ using $MR$ by $\Psi$. That is, $\Psi^t(x)=MR(\Phi_s^t(x))$.
\end{definition}

\begin{remark}
No appeal to geometrization is necessary if $M$ is Haken, in which case a result of Waldhausen \cite[Corollary 6.5]{Waldhausen} implies $M_s$ is homeomorphic to $M$ via a map inducing the identification of fundamental groups.
\end{remark}

We denote the restrictions of $\Phi$ and $\tPhi$ to $M_\lk$ and $\tM_\lk$ respectively by
\[
\Phi_\lk:= \Phi|_{M_\lk} \,\,\text{ and }\,\,\tPhi_\lk:= \tPhi|_{M_\lk}.
\]
The map $MR\circ S$ carries oriented flowlines of $\Phi_\lk$ to oriented flowlines of $\Psi$. In fact, by the way we parametrized $\tPhi_s$, it takes \emph{parametrized} flowlines to parametrized flowlines:

\begin{lemma}\label{lem:semiconjugacy2}
The map $MR\circ S$ is a semiconjugacy from $\Phi_\lk$ to $\Psi$, i.e. for all $x\in M_\lk$, $t\in \bbR$,
\[
\Phi_\lk^t(x)\mapsto \Psi^t((MR\circ S)(x)).
\]
\end{lemma}

\begin{proof}
    This follows immediately from \Cref{lem:semiconjugacy1}.
\end{proof}

We have now constructed the flow and semiconjugacy from \Cref{theorem_main}. What remains is to show that:
\begin{enumerate}
    \item the semiconjugacy is homotopic to inclusion,
    \item $\Psi$ is quasigeodesic, and
    \item $\Psi$ is pseudo-Anosov.
\end{enumerate}
The first two points will be handled in \Cref{sec:homotopyquasigeodesity}, and the pseudo-Anosov property in \Cref{section:StraightenedFlowsPA}.

\subsection{Homotopy considerations and quasigeodesity of $\Psi$} \label{sec:homotopyquasigeodesity}

Let $\pi$ denote the projection $T^1\tM\to \tM$. Continuing to let $R\colon M\to T^1\tM$ denote our flow-equivariant comparison map for $\Phi$, there is a natural map $\wt G_\Phi\colon \tM\to \tM$ defined by $\wt G_\Phi:=\pi\circ R$. This map appeared in \Cref{lemma:ComparisonBDandUC}, where we noted that $d(x, \wt G_\Phi(x))$ is uniformly bounded. It follows that $\wt G_\Phi$ can be extended to $S^2_\infty$ by the identity map. In particular $\wt G_\Phi$ is surjective.

\begin{definition}\label{def:geodesificationmap}
     We call $\wt G_\Phi$ the \emph{geodesification map} for $\wt\Phi$.  
    Since $\wt G_\Phi$ is equivariant, it induces a surjection $G_\Phi\colon M\to M$ that sends each $\Phi$-orbit to the corresponding geodesic. We call $G_\Phi$ the \emph{geodesification map} for $\Phi$. Note that both geodesifications depend on $R$. 
\end{definition}

We remark that the geodesification map for $\Phi$ cannot be injective, since Zeghib proved that no closed hyperbolic manifold is foliated by geodesics \cite{Zeghib}.

\begin{lemma}
The geodesification map $G_\Phi$, defined above, is homotopic to the identity on $M$.
\end{lemma}

\begin{proof}
As noted above, the particular lift $\wt G_\Phi$ of $G_\Phi$ continuously extends to $\partial_\infty \bbH^3$ by the identity. This in itself may already satisfy the reader, but for completeness we explain why this property suffices to prove the lemma, essentially following ideas in \cite[\S1]{HandelThurston}.

For convenience, we use $\tG$ and $G$ to denote $\tG_\Phi$ and $G_\Phi$ for the remainder of this proof.

First, by a homotopy of $G$ we may assume that $\wt G$ fixes some $\wt p\in \wt M$ and still moves points a uniformly bounded distance. Next, we claim that $\wt G$ commutes with each deck transformation of $\wt M$. It suffices to prove this for primitive deck transformations. Let $g$ be a primitive deck transformation, and let $\omega\in S^2_\infty$ be the terminal point of $g$'s geodesic axis.
For all $x\in \wt M$, we have
\[
\lim_{n\to \infty} (\wt G g \wt G^{-1})^n(x)=\wt G(\lim_{n\to\infty} g^n(\wt G^{-1}(x)))=\wt G(\omega)=\omega.
\]
Hence the axes of $g$ and $\wt G g\wt G^{-1}$ have the same terminal point, so they are equal. Since $g$ is primitive, $\wt G g\wt G^{-1}=g^k$ for some $k$. Thus $g= (\wt G^{-1} g\wt G)^k$ for some $k$, so $k=1$ (again because $g$ is primitive). This shows $\wt G$ commutes with $g$ as desired.

This implies that for any deck tranformation $h$, we have 
\[
\wt G(h(\wt p))=h(\wt G(\wt p))=h(\wt p).
\]
Hence $G$ induces the identity map on $\pi_1(M,p)$, where $p$ is the image of $\wt p$ under the covering projection. Since $M$ is a $K(\pi,1)$, the lemma follows.
\end{proof}

We now define a few more maps:
\begin{itemize}
\item The identification of $\pi_1(M)$ with $\pi_1(M_s)$ allows us to identify the Gromov boundaries of each with $S^2_\infty$. Let $\wt{MR}\colon \wt M_s\to \wt M$ be the lift of $MR$ that fixes $S^2_\infty$ pointwise.
\item We can define a comparison map $R_\Psi\colon \wt M\to T^1 \wt M$ by $R_\Psi=R_s\circ (\wt {MR})^{-1}$, with associated geodesifications $\wt G_\Psi\colon \tM\to \tM$ and $G_\Psi\colon M\to M$.
\end{itemize}

These maps fit into the diagram in \Cref{bigdiagram}. 
We note that all of the four-sided and 3-sided complementary regions commute by the definitions of the various maps. The next proof will show in particular that the entire diagram commutes.

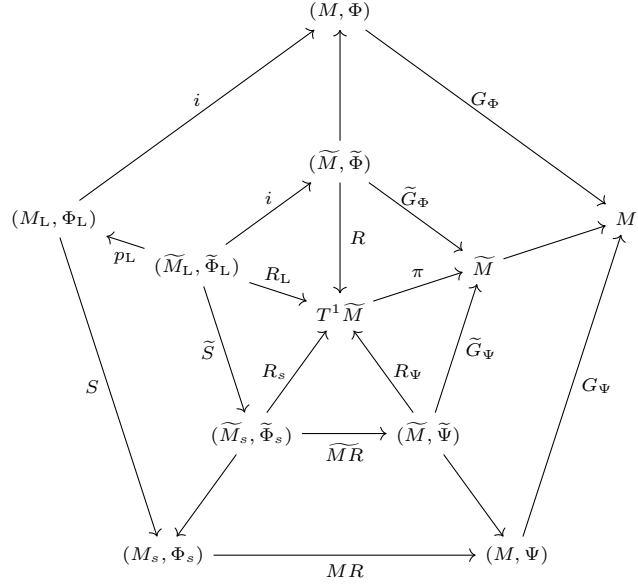
\begin{figure}
\begin{center}
    \begin{tikzpicture}[scale=2]
\node (O) at (0,0){\scriptsize$T^1\tM$};
\node (A2) at (90:2){\scriptsize$(M,\Phi)$};
\node (A1) at (90:1){\scriptsize$(\tM,\wt\Phi)$};
\node (B2) at (162:2){\scriptsize$(M_\lk, \Phi_\lk)$};
\node (B1) at (162:1){\scriptsize$(\tM_\lk, \wt\Phi_\lk)$};
\node (C2) at (234:2){\scriptsize$(M_s, \Phi_s)$};
\node (C1) at (234:1){\scriptsize$(\tM_s, \wt\Phi_s)$};
\node (D2) at (306:2){\scriptsize$(M, \Psi)$};
\node (D1) at (306:1){\scriptsize$(\tM, \wt\Psi)$};
\node (E1) at (18:1){\scriptsize$\tM$};
\node (E2) at (18:2){\scriptsize$M$};
\draw
(A1) edge[->] node[right] {\scriptsize$R$} (O)
(B1) edge[->] node[above] {\scriptsize$R_\lk$} (O)
(C1) edge[->] node[left] {\scriptsize$R_s$} (O)
(D1) edge[->] node[right] {\scriptsize$R_\Psi$} (O)
(O) edge[->] node[above] {\scriptsize$\pi$} (E1)
(A1) edge[->] node[right] {} (A2)
(B1) edge[->] node[below] {\scriptsize$p_{\lk}$} (B2)
(C1) edge[->] node[left] {} (C2)
(D1) edge[->] node[left] {} (D2)
(E1) edge[->] node[below] {} (E2)
(B1) edge[->] node[above] {\scriptsize$i$} (A1)
(B1) edge[->] node[left] {\scriptsize$\wt S$} (C1)
(C1) edge[->] node[below] {\scriptsize$\wt{MR}$} (D1)
(D1) edge[->] node[right] {\scriptsize$\wt G_\Psi$} (E1)
(A1) edge[->] node[above] {\scriptsize$\wt G_\Phi$} (E1)
(B2) edge[->] node[above] {\scriptsize$i$} (A2)
(B2) edge[->] node[left] {\scriptsize$S$} (C2)
(C2) edge[->] node[below] {\scriptsize$MR$} (D2)
(D2) edge[->] node[right] {\scriptsize$G_\Psi$} (E2)
(A2) edge[->] node[above] {\scriptsize$G_\Phi$} (E2);
\end{tikzpicture}
\end{center}
\caption{The maps defined in this section fit into the above diagram, which is referred to in the proof of \Cref{prop:homotopictoinclusion}. By a slight abuse we use $i$ to denote the inclusions $M_\lk\hookrightarrow M$ and $\tM_\lk \hookrightarrow \tM$. Each outer spoke is covering map or the restriction of one.}
\label{bigdiagram}
\end{figure}

\begin{proposition}\label{prop:homotopictoinclusion}
The semiconjugacy from $\Phi_\lk$ to $\Psi$ is homotopic to inclusion $i\colon M_\lk\hookrightarrow M$, i.e.
\[
MR\circ S\simeq i.
\]
\end{proposition}

\begin{proof}
    Since the triangles and squares of the diagram in \Cref{bigdiagram} commute, we have 
    \[
    G_\Psi\circ MR\circ S\circ p_\lk=G_\Phi\circ i\circ p_\lk
    \]
    where $p_\lk\colon \wt M_\lk\to M_\lk$ is the restriction of the covering projection $\wt M\to M$. Since $p_\lk$ is surjective, we have 
    \[
    G_\Psi\circ MR\circ S=G_\Phi\circ i.
    \]
    Since $G_\Phi$ and $G_\Psi$ are homotopic to the identity, we have $MR\circ S\simeq i$.
\end{proof}

We conclude the section by showing:

\begin{proposition}\label{prop:flowisquasigeodesic}
$\Psi$ is quasigeodesic.
\end{proposition}
\begin{proof}
    \Cref{prop:homotopictoinclusion} implies that each flowline of $\wt\Psi$ has well-defined and distinct endpoints. It therefore suffices, by \Cref{proposition:QGFlowCharacterization}, to show that the positive and negative endpoint maps for $\tPsi$ are continuous. But these are just the maps induced by the endpoint maps for $\tPhi$ under the quotient map $\wt {MR}\circ \wt S$, so they are continuous.
\end{proof}

\section{Straightened flows are pseudo-Anosov}\label{section:StraightenedFlowsPA}

In this section we complete the proof of our main theorem, \Cref{theorem_main}.

In \Cref{section:StraighteningFlow} we constructed a quasigeodesic flow $\Psi$ on $M$ which is strongly redolent of a pseudo-Anosov flow. Most notably, the sprig decompositions of its compactified flowspace restrict to a transverse pair of singular foliations of the flowspace, so the lifted flow $\tPsi$ on $\tM$ has 2-dimensional $\tPsi$-invariant singular foliations which are preserved by the $\pi_1(M)$-action. Downstairs in $M$, this means that $\Psi$ preserves a pair of transverse singular foliations (the positive and negative decompositions).

What remains is to prove that $\Psi$ is genuinely pseudo-Anosov. To do this, we will build a metric compatible with the topology of $M$ with respect to which it is easy to show that points in strong positive/negative leaves are forward/backward asymptotic. This itself verifies all the properties in the definition of a topological pseudo-Anosov flow except the existence of a Markov partition. Rather than directly constructing one, we will use our results about forward and backward asymptoticity to directly prove that $\Psi$ is expansive (\Cref{defn:expansive}); it is known that all expansive flows are topologically pseudo-Anosov (see \Cref{prop:ifexpansivethenpA}). 

We conclude the section, and the paper, by proving that the strong stable and unstable decompositions of $\Psi$ are uniformly exponentially contracted/expanded in forward time.

\subsection{Setup for the section}

In this section $\Psi$ will denote a quasigeodesic flow on a closed hyperbolic 3-manifold $M$ such that the associated positive/negative leaf decompositions of the flowspace form a pair of transverse singular foliations that intersect efficiently.

Given such a flow, one can construct a flow-equivariant comparison map $R$ as in \Cref{section:comparisonmaps}, which yields strong stable and unstable decompositions. As in \Cref{section:productstructures} we can find a product structure on $\tM$ so that all strong positive leaves are horizontal.

This allows us to cover $\tM$ by flowboxes $A\times I$ (see \ref{it:weakcontraction} of \Cref{sec:pAflows}) so that each strong positive leaf that intersects a flow box does so in a connected subset of a horizontal slice $A\times\{t\}$. We call such a flowbox \emph{positively adapted}. This is, in particular, a system of charts that shows that the strong positive decomposition is a 1-dimensional singular foliation of $\tM$. Similarly, we can define \emph{negatively adapted} flowboxes and conclude the strong negative leaf decompositions of $\tM$ are singular foliations. Since both are flow-invariant and $\pi_1(M)$-invariant, they project to $\Psi$-invariant 1-dimensional singular foliations of $M$.

\subsection{The polygonal metric}

Let us set some notation. For $x,y\in \tM$, let  $\td_\bbH(x,y)$ denote the distance between $x$ and $y$ in the hyperbolic metric; similarly let $d_\bbH$ denote hyperbolic distance downstairs in $M$.
For a metric on $T^1\tM$, we use the canonical (Sasaki) metric, which is preserved by the action of $\pi_1(M)$ on $T^1\tM$. This metric has the property that geodesics in $\tM$ lift to geodesics in $T^1\tM$, and stable/unstable horospheres are uniformly exponentially contracted/expanded by the geodesic flow in forward time.

Given a subset $A \subset \tM$, we denote the diameter of $R(A)$ in $T^1\tM$ by 
\[
\diam_R(A).
\]
If $A$ lies in a single strong positive/negative sprig, then $R(A)$ lies in a single stable/unstable horosphere, and we write 
\[\diam^\pm(A)\] 
for the diameter of $R(A)$ in the induced path metric on this horosphere.

Any two points $x, y$ in a strong positive or negative leaf $k \in \cF^\ppmm$ are connected by a unique oriented sub-arc of $k$. We call this a \emph{strong positive} or \emph{strong negative segment}, and denote it by $k[x, y]$. A \emph{flow segment} is a compact arc $x \cdot [t_0, t_1]$ in a flowline. This may be oriented to agree or disagree with the flow, depending on whether $t_0 \leq t_1$ or $t_0 \geq t_1$.

The initial and terminal points of a flow segment or strong positive or negative segment $\sigma$ will be denoted by $\iota(\sigma) = p$ and $\tau(\sigma) = q$.

\begin{definition}
	A \emph{polygonal path} from $x\in\tM$ to $y\in \tM$ is a finite concatenation
	\[ 
    \alpha = \sigma_0 \star \sigma_1 \star \cdots \star \sigma_n 
    \]
	of flow segments, strong positive segments, and strong negative segments, where $x = \iota(\sigma_0)$, $\tau(\sigma_i) = \iota(\sigma_{i+1})$ for each $i$, and $\tau(\sigma_n) = y$.
\end{definition}

\begin{lemma}
	Any two points in $\tM$ can be joined by a polygonal path.
\end{lemma}
\begin{proof}
	First note that any two points in $P$ may be joined by a polygonal path, i.e. a finite formal concatenation of positive and negative segments. Indeed, the set of points that may be joined to a chosen basepoint by a polygonal path is both closed and open in any chart, and hence in $P$.
	
	Let $x$ and $y$ be points in $\tM$. Then a polygonal path from $\nu(x)$ to $\nu(y)$ in $P$ lifts to a unique polygonal path, with no flow segments, from $x$ to some point $y' \in y \cdot \bbR$. Concatenating with a flow segment yields a polygonal path from $x$ to $y$.
\end{proof}

 \begin{definition}
 	Define the \emph{polygonal length} of a positive/negative segment $k[x, y]$ by 
     \[
     \ell_\diamond(k[x,y]) := \diam^\pm(k[x, y]),
     \]
     and of a flow segment $x \cdot [t_0, t_1]$ by 
     \[
     \ell_\diamond(x \cdot [t_0, t_1])=\diam_R(x \cdot [t_0, t_1]).
     \] 
     (Note that the above equals $|t_1-t_0|$ since $R$ is flow-equivariant).
     Define the \emph{polygonal length} of a polygonal path $\alpha = \sigma_0 \star \cdots \star \sigma_n$ by
 	\[ 
     \ell_\diamond(\alpha) := \sum_{i = 0}^n \ell_\diamond(\sigma_i). 
     \]
 	The \emph{distance} between points $x, y \in \tM$ is defined by
 	\[ 
     \td(x, y) = \inf \{\ell_\diamond(\alpha) \mid \alpha \text{ polygonal path from } x \text{ to } y \}. \qedhere\]
\end{definition}

\begin{lemma}\label{lem:compareRdiampoly}
    For any polygonal arc $\alpha$, $\diam_R(\alpha)\le \ell_\diamond(\alpha)$.
\end{lemma}

\begin{proof}
     In any metric space, if sets $A$ and $B$ intersect nontrivially then $\diam(A\cup B)\le \diam (A)+\diam (B)$. Hence if $\alpha=\sigma_1\star\cdots\star \sigma_n$ is a polygonal path, then 
     \[
     \diam_R(\alpha) 
     \le
     \sum_{i=1}^n\diam_R(\sigma_i) 
     \le \sum_{i=1}^n\ell_\diamond(\sigma_i)=\ell_\diamond(\alpha).
     \]
     For the second inequality above, we have simply used that the distance between two points in the same stable or unstable horosphere, with respect to the path metric on that horosphere, is bounded below by their distance in $T^1\tM$.
  \end{proof}

In particular if $\alpha_1, \alpha_2, \dots$ is a sequence of polygonal arcs with $\ell_\diamond(\alpha_i) \to 0$, then $\diam_R(\alpha_i) \to 0$. 

\begin{lemma}\label{lemma:Rtd}
	If $A \subset \tM$ is a connected set such that $\diam(R(A)) = 0$, then $A$ is a single point. 
\end{lemma}
\begin{proof}
    It suffices to show that for each $v \in T^1\tM$, the preimage $R^{-1}(v)$ is totally disconnected.

	The flowlines that intersect $R^{-1}(v)$ correspond to the set of points $p \in Q$ with $e^+(p) = z$ and $e^-(p) = w$, where $z$ and $w$ are the endpoints of the geodesic through $v$ in $S^2_\infty$. This set is totally disconnected, since each component is contained in the intersection of a positive and negative leaf (i.e. the intersection of a component of $e^+(p)$ with a component of $e^-(p)$), and each positive leaf intersects each negative leaf in at most one point.

    Since $R$ is flow-equivariant, $R^{-1}(v)$ intersects each flowline in at most one point. The lemma follows.
\end{proof}

\begin{lemma}\label{lemma:compatiblemetrics}
	$\td$ is a metric that is compatible with the topology on $\tM$.
\end{lemma}
\begin{proof}
	It is immediate that $\td$ is symmetric and satisfies the triangle inequality, and that $\td(x, x) = 0$ for all $x \in \tM$.

    Next we claim that $x_i \in \tM$ converges to $x\in \tM$ if and only if $\lim \td(x, x_i) = 0$. In the special case when $(x_i)$ is a constant sequence, this shows that $d(x,y)=0$ implies $x=y$, so $\td$ is in fact a metric. More generally, it shows that $\td$ is compatible with the topology of $\tM$.

	\begin{itemize}
	\item[($\Rightarrow$)]First suppose that $\lim x_i = x$.
	
	Let $\wt K$ be the weak positive leaf through $x$. 
    Since $\lim x_i = x$, the weak negative leaf $\wt L_i$ through $x_i$ eventually intersects $\wt K$. Let $k$ be the strong positive leaf through $x$, and let $l_i$ be the strong positive leaf over $L_i$ that intersects $k$. This intersects the orbit $x_i \cdot \bbR$ at some point $y_i=x_i \cdot t_i$. Build a polygonal path $\alpha_i = \sigma^+_i \star \sigma^-_i \star \sigma_i$, where $\sigma^+_i$ is the segment in $k$ from $x$ to $k \cap l_i$, $\sigma^-_i$ is the segment in $l_i$ from $k \cap l_i$ to $y_i$, and $\sigma_i$ is the flow segment from $y_i$ to $x_i$.

    It is clear that the $\sigma_i^+$ converge to $x$. Working in a product structure near $x$ for which $k$ and its saturation by negative leaves are horizontal, we see that the segments $\sigma_i^-$ converge to $x$; note this implies $y_i\to x$. Since the flowlines of $\tPhi$ foliate $\tM$, the $\sigma_i$ must accumulate on a subset of the orbit through $x$. But if $x_i'$ is any point in $\sigma_i$, the sequence $x_1', x_2',\dots$ can only accumulate on $x$ since $x_i,y_i\to x$. We conclude that the $\sigma_i$ converge to $x$.
    
    Since $\sigma^+_i$, $\sigma^-_i$, and $\sigma_i$ all converge to $x$ by above, we have $\lim R(\sigma^\pm_i) = \lim R(\sigma_i) = R(x)$. Hence $\diam^\pm(\sigma^\pm_i)\to 0$,
    so $\lim \td(x, x_i) = 0$.

	\item[($\Leftarrow)$] Conversely, suppose that $\lim \td( x, x_i) = 0$.

    Choose polygonal paths $\alpha_i$ from $x$ to $x_i$ such that $\ell_\diamond(\alpha_i)<\td(x, x_i)+\frac{1}{i}$, so $\lim \ell_\diamond(\alpha_i)=0$. Then $\diam_R(\alpha_i)$ goes to zero by \Cref{lem:compareRdiampoly}, so $\lim  R(\alpha_i) = R(x)$.

    Suppose $x_i$ does not converge to $x$ in $\tM$. Then up to taking a subsequence, we can find an $\epsilon>0$ so that $x_i\notin \overline N_\epsilon(x)$ for all $i$, where $\overline N_\epsilon(x)$ is the closed ball of radius $\epsilon$ about $x$ in the hyperbolic metric. Let $\beta_i$ be the largest initial subpath of $\alpha_i$ contained in $\overline N_\epsilon(x)$. By taking a further subsequence, we can arrange for the $\beta_i$ to Hausdorff-converge to a set $\beta$ by \Cref{lem:subseqhauscon}. This set is connected by \Cref{lem:connectedlimit}, and meets $\del \overline N_\epsilon (x)$. However, applying \Cref{lemma:continuoushausdorfflim} to $R|_{\overline N_\epsilon(x)}$ we have 
    \[
    R(\beta)=\lim(R(\beta_i))=R(x).
    \]
    Hence $\beta=\{x\}$ by \Cref{lemma:Rtd}, a contradiction. Hence $\lim x_i=x$.\qedhere
    \end{itemize}
\end{proof}

Observe that $\td$ is invariant under deck transformations, so it descends to a metric $d$ on $M$ defined by setting 
\[
d(x,y)=\min\{\td(\wt x, \wt y)\mid \text{$\wt x$ lift of $x$, $\wt y$ lift of $y$}\}.
\]
The metric $d$ is uniformly equivalent to the hyperbolic metric on $M$ (i.e. the maps $(M, d_\bbH)\to (M, d)$ and $(M, d)\to (M,d_\bbH)$ are uniformly continuous) since $M$ is compact. Hence $\wt d$ is uniformly equivalent to the hyperbolic metric $\td_\bbH$ on $\wt M$. In particular, for any $\epsilon>0$, there exists $\delta>0$ so that whenever $\td(a,b)<\delta$, then $\td_\bbH(a,b)<\epsilon$ .

With this observation it is easy to show that strong stable/unstable leaves are \emph{topologically} contracted/expanded by the flow:

\begin{proposition}\label{proposition:WeakContraction}
	If $x, y \in \tM$ are contained in a single strong positive (or negative) leaf, then $\td(x \cdot t, y \cdot t) \to 0$ as $t \to \infty$ (resp. $-\infty$).
\end{proposition}

\begin{proof}
    Suppose without loss of generality that $x$ and $y$ are contained in the same strong positive leaf $k \in \cF^\pp$, and let $c = k[x, y]$ be the strong positive segment from $x$ to $y$. Then 
    \[
    0\le\lim_{t\to\infty}\td(x\cdot t, y\cdot t)\le \lim_{t\to\infty}\diam^+ (c \cdot t).
    \]
    Since the geodesic flow contracts distance in stable horospheres, the righthand limit is 0.
\end{proof}

We remark that until this point, we had not established that leaves of the 2-dimensional $\Psi$-invariant foliations contain at most one singular orbit.

\Cref{proposition:WeakContraction} establishes that $\Psi$ possesses the weak contraction property \ref{it:weakcontraction} from \Cref{sec:pAflows}. Next we will show it possesses property \ref{it:markovpartition}, i.e. has a Markov partition, by way of showing it is expansive.

\subsection{Expansiveness}

Before defining expansiveness, we make a definition:

\begin{definition}\label{definition:fellowtravelutr}
If $F$ is a flow on a  metric space $X$, we say that two points $x,y\in X$ \emph{ $\delta$-fellow travel up to reparameterization} if there exists $\delta>0$, and an increasing homeomorphism $h\colon \bbR\to \bbR$ with $h(0)=0$, such that 
\[
d(F^t(x), F^{h(t)}(y))<\delta \text{ for all $t\in \bbR$.}
\]
Note that since $h^{-1}$ is also an increasing homeomorphism fixing 0, the roles of $x$ and $y$ here are symmetric. 

If $\alpha$ and $\beta$ are orbits respectively containing points $x$ and $y$ that $\delta$-fellow travel up to reparameterization, we say that the orbits $\alpha$ and $\beta$ themselves $\delta$-fellow travel up to reparameterization.
We also use this terminology with forward/backward half orbits.
\end{definition}

The following definition is due to Bowen--Walters (see \cite[Thm. 3]{BowenWalters}):

\begin{definition}\label{defn:expansive}
    Let $X$ be a compact metric space. A nonsingular flow $F$ on $X$ is \emph{expansive} if it satisfies the following property:
    \begin{enumerate}[label=(E\arabic*)]
    \item for all $\epsilon>0$, there exists $\delta>0$ such that if $x, y\in X$ $\delta$-fellow travel up to reparametrization, then $y=F^t(x)$ for some $t\in (-\epsilon,\epsilon)$.\qedhere
    \end{enumerate}
\end{definition}

Expansiveness is independent of metric, and is a conjugacy invariant (\cite[Cor. 4]{BowenWalters}). 
We will use the following alternative form of the definition.

\begin{lemma}
    Let $F$ be a nonsingular flow on a compact metric space $X$. Then $F$ is expansive if and only if it satisfies the following property:
    \begin{enumerate}[label=\emph{(E\arabic*)}, start=2]
    \item for all $\epsilon>0$, there exists $\delta>0$ such that if $x,y\in X$ $\delta$-fellow travel up to reparameterization, then $y$ lies in the same orbit as $x$ and the flow segment from $y$ to $x$ lies in $N_\epsilon(x)$.
    \end{enumerate}
\end{lemma}
\begin{proof}
    The proof follows arguments in \cite{BowenWalters}, but we include it since the statement is not in that paper explicitly.

    Let $t_0$ be shortest positive period of an orbit of $\Phi$ (if $F$ has no periodic orbit, set $t_0=1$). We claim that  for any $t\in (0,t_0)$, there exists $\eta>0$ such that $d(x,F^t(x))>\eta$ for every $x\in X$. If not, there exists $t\in (0,t_0)$ and a sequence  $(x_i)$ in $X$ such that $d(x_i, F^t(x_i))\to 0$. Up to taking a subsequence, $(x_i)\to x$ for some $x\in X$, so $d(x, F^t(x))=\lim_{n\to\infty}d(x_n, F^t(x_n))=0$, a contradiction since $t<t_0$.

    Now, suppose that (E2) holds and let $0<\epsilon<T_0$. By the claim we just proved, there exists $\eta$ such that $d(x,F^\epsilon(x))>\eta$ for all $x\in X$. Use (E2) to choose a $\delta$ for $\eta$. Note that if the flow segment from $x$ to $y$ is contained in $N_\eta(x)$, then $y\in F^{(-\epsilon,\epsilon)}(x)$.

    Conversely, suppose that (E1) holds and let $\epsilon>0$. Since $X$ is compact, we can choose $\epsilon'$ so that $F^{(-\epsilon',\epsilon')}(x)\subset N_\epsilon(x)$ for all $x\in X$. Using (E1) to choose a $\delta$ for $\epsilon'$ gives (E2).
\end{proof}

In particular, if $F$ is \emph{not} expansive, then there exists some $\epsilon>0$ such that for all $\delta>0$, there exist $x$ and $y$ that $\delta$-fellow travel up to reparameterization such that either $x$ and $y$ lie on distinct orbits, or the flow segment from $y$ to $x$ is not contained in $N_\epsilon(x)$.

We will need a couple of  basic facts about flowboxes:

\begin{lemma}\label{lemma:oneplaque}
    No weak positive or weak negative leaf of $\tPsi$ intersects a flowbox in more than one plaque.
\end{lemma}

\begin{proof}
    Projecting such a configuration to the flowspace would force a negative and positive leaf to intersect in more than one point, a contradiction.
\end{proof}

\begin{lemma}\label{lem:separationconstant}
There exist $\eta, B>0$ such that for all $x\in \tM$, $N_\eta (x)$ is contained in a positively adapted flowbox of diameter less than $B$. The same is true replacing ``positive" by ``negative."
\end{lemma}

\begin{proof}
    Using a product structure for $\tM$ in which strong positive leaves are horizontal we can find, for every point $x\in M$, a positively adapted flowbox containing $x$ in its interior. By compactness we may pass to a finite subcollection whose interiors cover $M$. Choose $B$ to be larger than the maximum diameter in this finite collection of flowboxes, and choose $\eta$ small enough that $N_\eta(x)$ is contained in one of the finite collection. Now take all lifts of this collection of flowboxes to $\tM$.
    For the negative case, the argument is similar.
\end{proof}

The following is a useful criterion for expansiveness in our setting.

\begin{lemma}\label{expcrit}
    Suppose that $\tPsi$ satisfies the following: there exists $\delta>0$ such that if $x,y\in \tM$ $\delta$-fellow travel up to reparameterization, then $x$ and $y$ lie in the same flowline.

    Then $\Psi$ is expansive.
\end{lemma}

The key to this lemma is that the lifts to $\wt M$ of long flow segments in $M$ have endpoints that are far apart.

\begin{proof}
First we make the following claim: for all $\epsilon>0$, there exists $\eta>0$ such that if $x$ and $y$ lie in the same flowline of $\tPsi$ and $d(x,y)<\eta$, then the flow segment from $x$ to $y$ lies in $N_\epsilon(x)$.

Indeed, by compactness, we can cover $M$ by flowboxes of diameter $<\epsilon$. Choose $\eta>0$ small enough so that for all $x\in M$, $N_\eta(x)$ is contained in a flowbox. Lifting these flowboxes to $\tM$ we have the property that every $\eta$-neighborhood in $\tM$ is contained in a lifted flowbox of diameter $<\epsilon$. Now if $x$ and $y$ are in the same flowline and $y\in N_\eta(x)$, then $x$ and $y$ lie in a flowbox $A\times I$ of diameter $<\epsilon$, which must contain the flow segment from $x$ to $y$ since flowlines have connected intersection with flowboxes in $\tM$. Since $\diam(A\times I)<\epsilon$, we have $A\times I\subset N_\epsilon(x)$. This proves the claim.

Now take any $\epsilon>0$, and use the claim to choose a suitable $\eta$. Let $\delta$ be the constant from the lemma statement, and let $\delta_0$ be the minimum of $\delta$, $\eta$, and $\mathrm{inj}(M)$, where $\mathrm{inj}(M)$ is the injectivity radius of $M$.

Suppose that $x,y\in M$ $\delta_0$-fellow travel up to reparameterization. Since $\delta_0<\mathrm{inj}(M)$, we can choose lifts $\wt x$ and $\wt y$ of $x$ and $y$ to $M$ that $\delta_0$-fellow travel up to reparameterization. By the assumption, $\wt x$ and $\wt y$ must on the same flowline since $\delta_0<\delta$. Since $\delta_0<\eta$, the flow segment from $\wt x$ to $\wt y$ lies in $N_\epsilon (\wt x)$ by the claim. Downstairs in $M$, this implies the flow segment from $x$ to $y$ lies in $N_\epsilon(x)$, so $\Psi$ is expansive.
\end{proof}

\begin{lemma}\label{lem:HalfOrbitFT}
Any two forward half-orbits 
that $\eta$-fellow travel up to reparametrization must lie on distinct weak negative leaves. Any two backwards half-orbits that $\eta$-fellow travel up to reparametrization must lie on distinct weak positive leaves. Here $\eta$ is the constant in \Cref{lem:separationconstant}.
\end{lemma}
\begin{proof}
    Let $\alpha$ and $\beta$ be distinct forward half-orbits in the same weak negative leaf that $\eta$-fellow travel up to reparametrization. Each point $x \in \alpha$ is contained in negatively adapted flow box of diameter at most $B$ that contains $N_\eta(x)$. This intersects $\beta$, so it follows that there is a strong negative segment $k[x, y]$ of diameter at most $B$ from $x$ to a point $y \in \beta$. This applies for every point $x \in \alpha$, so every orbit that intersects the segment $k[x, y]$ must have the same forward endpoint. Since $k[x,y]$ is connected, it must lie in a single weak positive leaf. However, $k[x,y]$ intersects any weak positive leaf in at most one point by \Cref{lemma:oneplaque}, a contradiction.
    The argument for backwards half-orbits in the same weak positive leaf is similar.
\end{proof}

\begin{proposition}
$\Psi$ is expansive.
\end{proposition}
\begin{proof}
    Suppose for a contradiction that $\Psi$ is not expansive. 

    Let $(\alpha_1,\beta_1),(\alpha_2,\beta_2),\dots$ be a sequence of pairs of distinct $\tPhi$-flowlines such that $\alpha_i$ and $\beta_i$ $\delta_i$-fellow travel up to reparameterization, where $\delta_i\to 0$. Since $M$ is compact, we can assume after applying deck transformations that the sequences $(\alpha_i)$ and $(\beta_i)$ both limit (in the Kuratowski sense) on a single flowline $\zeta$. By \Cref{lem:HalfOrbitFT} we can assume after passing to a tail that $\alpha_i$ and $\beta_i$ lie on distinct positive leaves. After possibly passing to a further tail one can assume that the (distinct) weak positive leaves through $\alpha_i$ and $\beta_i$ intersect the weak negative leaf through $\zeta$ in distinct orbits $\alpha'_i$, $\beta'_i$.

    Now fix $I$ sufficiently large so that $\delta_I < \eta$. By \Cref{proposition:WeakContraction}, there are forward half orbits of $\alpha'_I$ and $\alpha_I$ that fellow travel arbitrarily closely and forward half orbits of $\beta'_I$ and $\beta_I$ that fellow travel arbitrarily closely. Then using the triangle inequality we can choose forward half orbits of $\alpha'$ and $\beta'$ that $\eta$-fellow travel up to reparametrization. This contradicts \Cref{lem:HalfOrbitFT}, since these lie in the same weak negative leaf.
\end{proof}

\begin{proposition}[Inaba-Matsumoto, Oka, Brunella, Iakovoglou]\label{prop:ifexpansivethenpA}
    Let $F$ be an expansive flow on a closed 3-manifold $M$. Then $F$ is topologically pseudo-Anosov. 
\end{proposition}
\begin{proof}
    The fact that expansive flows possess singular foliations satisfying property \ref{it:weakcontraction} in \Cref{sec:pAflows} is due to Inaba-Matsumoto \cite{InabaMatsumoto} and Oka \cite{Oka}. Brunella proved in \cite{Brunella} that transitive expansive flows admit Markov partitions, and the non-transitive case was proven by Iakovoglou \cite{Iakovoglou}, so property \ref{it:markovpartition} holds as well.
\end{proof}

We remark that the singular foliations  that $\Psi$ possesses by way of being expansive must agree with the projections $\cL^\pm$ of $\tcL^\pm$ to $M$. Leaves of each correspond to maximal connected flow-saturated subsets in which all orbits are forward (backward) asymptotic.

We immediately obtain:

\begin{corollary}\label{corollary:PsiispA}
    $\Psi$, together with the weak positive/negative foliations $\cL^\pm$, admits a Markov partition.
\end{corollary}

\subsection{Uniform exponential contraction/expansion}

To round out our discussion of $\Psi$'s nice properties, we will now upgrade \Cref{proposition:WeakContraction} by showing that strong positive and negative leaves are uniformly exponentially contracted and expanded; see \Cref{theorem:UniformExponentialContractionExpansion} below. We will focus on uniform exponential contraction of strong positive leaves, since the result for negative leaves is similar.

The Sasaki metric on $T^1 \tM$ is well adapted to the geodesic flow: there exists $\lambda>1$ such that if $A$ is a compact subset of a stable/unstable horosphere $S^\pm$, then 
\[
\diam_{S^\pm}(A\cdot t)=\lambda^{\mp t}\diam_{S^\pm}(A)
\]
for all $t\in \bbR$, where $\diam_{S^\pm}$ denotes diameter in the induced path metric on $S^\pm$. (In fact $\lambda=e$, as pointed out to us by M. Stover).
This implies:
 \begin{observation}
 	If $A$ is a compact subset of a single strong positive/negative leaf, then 
 	\[ 
    \diam^\pm(A \cdot t) = \lambda^{\mp t} \diam^\pm(A) 
    \]
 	for all $t$.
 \end{observation}

In particular if $\sigma = k[x, y]$ is a segment in a strong positive leaf, then $\Psi$ uniformly contracts $\diam^+(\sigma)$ in the sense that  $\diam^+(\sigma \cdot t) = \lambda^{-t} \diam^+(\sigma)$. 
To see that $\Psi$ uniformly contracts distance in strong positive leaves, we will show that $\diam^+(\sigma)$ and $\td(x, y)$ are uniformly comparable up to some scale:

\begin{proposition}[Bounded Distortion]\label{proposition:BoundedDistortion}
	There are constants $C > 0$ and $c > 0$ such that, if $\sigma = k[x, y]$ is a strong segment with $\td(x, y) \leq C$, then
	\[ c \diam^+(\sigma) \leq \td(x, y) \leq \diam^+(\sigma). \]
\end{proposition}

We will begin by understanding the difference in diameter of nearby ``parallel'' segments.

\begin{definition}
    Two segments $\mu_0, \mu_1 \subset P$ of positive leaves are \emph{parallel} if each negative leaf that intersects $\mu_0$ also intersects $\mu_1$ and vice versa.
	
    Two segments $\sigma_0, \sigma_1 \subset \tM$ of strong positive leaves are \emph{parallel} if the corresponding positive leaf segments $\nu(\sigma_0)$ and $\nu(\sigma_1)$ are parallel.
\end{definition} 

Consider parallel positive leaf segments $\sigma$ and $\sigma'$. For each point $x \in \sigma$, there is a unique polygonal path $\alpha_x$ that starts at $x$, traverses a single strong negative segment, then a single flow segment, and ends at a point $x' \in \sigma'$. We call $\tau_x$ the \emph{tine} at $x$. The family of tines $\{\tau_x\}_{x \in \sigma}$ from $\sigma$ to $\sigma'$ is called the \emph{rake} from $\sigma$ to $\sigma'$. See \Cref{figure:Rake}.

We call the maximal polygonal length of these tines the \emph{rake distance} 
\[d_\rake(\sigma, \sigma') = \sup_{x \in \sigma} \ell_\diamond(\tau_x).\]
This is not generally symmetric, but it is still useful: $\ell_\diamond(\tau_x)$ is continuous in $x \in \sigma$, 
so the rake distance between compact segments is finite, and it bounds the Hausdorff distance (with respect to $\td$) above, i.e. $d_H(\sigma, \sigma') \leq d_{\rake}(\sigma,\sigma')$.

\begin{figure}[ht]
	\centering
	\includegraphics{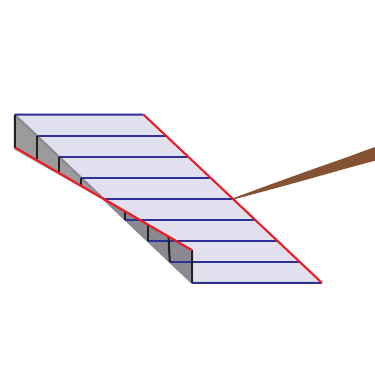}
	\caption{A rake.} \label{figure:Rake}
\end{figure}

The next lemma says that given a strong positive leaf segment, nearby strong positive leaf segments have roughtly the same $\diam^+(\cdot)$.

\begin{lemma}\label{lemma:ParallelDistortion}
	There are constants $\epsilon > 0$ and $b > 0$ such that, if $\sigma$ is a strong positive leaf segment with $\diam^+(\sigma) \leq 1$, and $\sigma'$ is a parallel segment with
	\[ d_\rake(\sigma, \sigma') < \epsilon, \]
	then
	\[ b \diam^+(\sigma) \leq \diam^+(\sigma'). \]
\end{lemma}
\begin{proof}
	Otherwise, we can find sequences of parallel strong positive leaf segments $\sigma_i$ and $\sigma'_i$ such that
	\[ \diam^+(\sigma_i) \leq 1 \text{ for all } i, \]
	\[ d_{\rake}(\sigma_i, \sigma'_i) \to 0 \text{ as } i \to \infty, \]
	and 
	\[ \frac{\diam^+(\sigma'_i)}{\diam^+(\sigma_i)} \to 0 \text{ as } i \to \infty. \]
		
	For each $i$ there is a time $t_i \leq 0$ such that $\diam^+(\sigma_i \cdot t_i) = 1$. Replace $\sigma_i$ by $\sigma_i \cdot t_i$ and $\sigma'_i$ by $\sigma'_i \cdot t_i$. This scales the numerator and denominator by the same factor, so we still have $\frac{\diam^+(\sigma'_i)}{\diam^+(\sigma_i)} \to 0$. Flowing a rake for negative time shrinks its strong negative segments and preserves the lengths of its orbit segments, so we still have $d_{\rake}(\sigma_i, \sigma'_i) \to 0$.
	
	The $\sigma_i$ are segments of uniformly bounded diameter in the strong positive singular foliation. After composing with elements of $\pi_1(M)$ and taking a subsequence, we can assume that they Hausdorff-limit on a strong positive leaf segment $\sigma = \lim_H \sigma_i$. The Hausdorff distance from $\sigma_i$ to $\sigma'_i$ is bounded above by $d_{\rake}(\sigma_i, \sigma'_i)$, which tends to $0$, so $\lim_H\sigma'_i = \sigma$ as well. But then $R(\sigma_i)$ and $R(\sigma_i')$ Hausdorff-limit on $R(\sigma)$, so 
    \[ \lim_{i \to \infty} \diam^+(\sigma'_i) = \diam^+(\sigma) = \lim_{i \to \infty} \diam^+(\sigma_i) = 1,\]
    so we have $\frac{\diam^+(\sigma'_i)}{\diam^+(\sigma_i)} \to 1$, a contradiction.
\end{proof}

\begin{proposition}
    There is a constant $c > 0$ such that
    \[ c \diam^+(k[x, y]) \leq \td(x, y) \leq \diam^+(k[x, y]) \]
    for every strong positive leaf segment $k[x, y]$ with $\diam^+(k[x, y]) \leq 1$.
\end{proposition}
\begin{proof}
    The righthand inequality is obvious, since $k[x, y]$ is itself a polygonal path from $x$ to $y$ with $\ell_\diamond(k[x, y]) = \diam^+(k[x, y])$.
	
    If the lemma were false, then we could find strong positive leaf segments $\sigma_i = k_i[x_i, y_i]$ with
    \[ \diam^+(\sigma_i) \leq 1 \text{ for all } i \]
    and
    \[ \frac{\td(x_i, y_i)}{\diam^+(\sigma_i)} \to 0 \text{ as } i \to \infty. \]
    Fix polygonal paths $\alpha_i$ from $x_i$ to $y_i$ such that 
    \[ \frac{\ell_\diamond(\alpha_i)}{\diam^+(\sigma_i)} \to 0. \]
	
    After applying deck transformations we can assume that the $\sigma_i = k_i[x_i, y_i]$ converge to a strong positive leaf segment $k[x, y]$, where $x = \lim_{i \to \infty} x_i$ and $y = \lim_{i \to \infty} y_i$. Since the denominator in the preceding limits is bounded above, the numerators must go to $0$; this implies that the $\td$-diameters of the $\alpha_i$ go to 0. Hence $k[x,y]=x=y$ is a point and the $\alpha_i$ converge to $x$.

    Let $B_i$ be a sequence of negatively adapted flow boxes around $x$ with diameters tending to $0$. Let $\epsilon_i$ be the maximal $\diam^+(\cdot)$ of the strong negative leaf segments in $B_i$ and let $t_i$ be  $\diam_R(\cdot)$ of the flow segments in $B_i$; both $\epsilon_i$ and $t_i$ tend to $0$. After taking subsequences we can assume that $\sigma_i \subset B_i$ for all $i$.

    For each $i$, we claim that can subdivide $\sigma_i$ into segments
    \[ \sigma_i = \sigma_i^1 \star \cdots \star\sigma_i^{n(i)}, \]
    such that each $\sigma_i^j$ is parallel to a strong positive subsegment $\mu_i^j$ of the polygonal path $\alpha_i$, and such that the
    \[ \mu_i^1, \cdots, \mu_i^{n(i)} \subset \alpha_i \]
    are disjoint. To see this, observe that the weak negative plaque in $B_i$ through each point $x \in \sigma_i$ intersects $\alpha_i$ at either an interior point of a strong positive leaf segment in $\alpha_i$ or in endpoints of two strong positive leaf segments. Compactness of $\sigma_i$ yields the existence of the desired subdivision.

    Since this is taking place in $B_i$, the rake distance between $\sigma_i^j$ and $\mu_i^j$ is bounded above by $\epsilon_i+t_i$, which tends to $0$.
    When $i$ is sufficiently large, this is less than the $\epsilon$ in Lemma~\ref{lemma:ParallelDistortion}, so $\diam^+(\mu_i^j) \geq b \diam^+(\sigma_i^j)$ for all $j$. Adding these for $j = 1, 2, \cdots, n(i)$ yields 
    \[ \ell_\diamond(\alpha_i) \geq \sum_{j = 1}^{n(i)} \diam^+(\mu_i^j) \geq b \sum_{j = 1}^{n(i)} \diam^+(\sigma_i^j) \geq b \diam^+(\sigma_i), \]
    which contradicts $\frac{\ell_\diamond(\alpha_i)}{\diam^+(\sigma_i)} \to 0$.
\end{proof} 

\begin{lemma}\label{lem:constantC}
	There is a constant $C > 0$ such that if $x, y$ are points in a strong leaf $k$ with $\td(x, y) \leq C$, then $\diam^+(k[x, y]) \leq 1$.
\end{lemma}
\begin{proof}
	Otherwise, we could find a sequence of strong positive leaf segments $\sigma_i = k_i[x_i, y_i]$ with 
	\[ \diam^+(\sigma_i) > 1 \text{ for all } i \]
	and
	\[ \td(x_i, y_i) \to 0 \text{ as } i \to \infty. \]
	
	After translating by elements of $\pi_1(M)$ and taking a subsequence, we can assume that $x_i$ converges to a point $x$. Since $\td_\diamond(x_i, y_i)$ goes to zero, $y_i$ also converges to $x$. Since a positively adapted flowbox at $x$ intersects each strong positive leaf in a single plaque by \Cref{lemma:oneplaque}, this means that $\sigma_i$ converges to $x$, which contradicts $\diam^+(\sigma_i) > 1$.
\end{proof}

The preceding lemma and proposition combine to prove \Cref{proposition:BoundedDistortion}. 

\begin{theorem}[Uniform exponential contraction/expansion]\label{theorem:UniformExponentialContractionExpansion}
	There are constants $C, a > 0, \lambda>1$ such that:
	\begin{itemize}
		\item If $x, y \in \tM$ are contained in the same strong positive leaf, and $\td(x, y) \leq C$, then 
		\[ \td(x \cdot t, y \cdot t) \leq a \lambda^{-t} \td(x, y) \text{ for all } t \geq 0. \]
		
		\item If $x, y \in \tM$ are contained in the same strong negative leaf, and $\td(x, y) \leq C$, then 
		\[ \td_\diamond(x \cdot t, y \cdot t) \leq a \lambda^t \td_\diamond(x, y) \text{ for all } t \leq 0.\]
	\end{itemize}
\end{theorem}
\begin{proof}
	Let $C$ be the constant from \Cref{lem:constantC}, and let $x$ and $y$ be points in a strong positive leaf $k$ such that $\td(x,y)\le C$. Then
	\[ \td(x \cdot t, y \cdot t) \leq \diam^+(k[x \cdot t, y \cdot t]) = e^{-t} \diam^+(k[x, y]), \]
	where the inequality comes from the fact that $k[x \cdot t, y \cdot t]$ is itself a polygonal path from $x \cdot t$ to $y \cdot t$, and the equality comes from the fact that $R(k[x \cdot t, y \cdot t]) = \Theta_t(R(k[x, y])$, where $\Theta$ is the geodesic flow. By Proposition~\ref{proposition:BoundedDistortion}, we have 
	\[ \diam^+(k[x, y]) \leq \frac{1}{c} \td_\diamond(x, y). \]
	Combining, we have
	\[ \td_\diamond(x \cdot t, y \cdot t) \leq \frac{1}{c} e^{-t} \diam^+([x, y]), \]
	so we can take $a=\frac{1}{c}$ and $\lambda=e$.
	A similar argument works for points in a strong negative leaf.
\end{proof}

\subsection{Proof of the main theorem}
We have now proven all the elements of \Cref{theorem_main}. For the reader's convenience, we recall the statement here and assemble the proof.

\maintheorem*

\begin{proof}
    We define $M_\lk$ and $\Psi$ as in \Cref{subsec:buildingtheflow}. Note that this involves  reparametrizing $\Phi$ so that it admits a flow-equivariant comparison map. By \Cref{lem:semiconjugacy2} there is a semiconjugacy $M_\lk\to M$ carrying $\Phi_\lk=\Phi|_{M_\lk}$ onto $\Psi$. In \Cref{prop:homotopictoinclusion} we proved that this semiconjugacy is homotopic to the inclusion map $M_\lk\hookrightarrow M$.

    By \Cref{prop:flowisquasigeodesic}, $\Psi$ is quasigeodesic.

    What remains is to verify $\Psi$ is pseudo-Anosov. We will prove the slightly stronger statement that $\Psi$ satisfies \ref{it:weakcontraction}, \ref{it:strongcontraction}, and \ref{it:markovpartition} from \Cref{sec:pAflows}.
    By \Cref{proposition:WeakContraction}, $\Psi$ satisfies the weak contraction property \ref{it:weakcontraction}.
    By \Cref{theorem:UniformExponentialContractionExpansion}, $\Psi$ satisfies the strong contraction property \ref{it:strongcontraction}.
    By \Cref{corollary:PsiispA}, $\Psi$ admits a Markov partition, hence satisfies property \ref{it:markovpartition}.
\end{proof}

\bibliographystyle{alphaurl}
\bibliography{bibliography}

\begin{thebibliography}{BFFP24}

\bibitem[AT24]{AgolTsang}
Ian Agol and Chi~Cheuk Tsang.
\newblock Dynamics of veering triangulations: infinitesimal components of their
  flow graphs and applications.
\newblock {\em Algebraic and Geometric Topology}, 24(6):3401--3453, 2024.

\bibitem[AT25a]{AlfieriTsang1}
Antonio Alfieri and Chi~Cheuk Tsang.
\newblock Heegaard {F}loer theory and pseudo-{A}nosov flows {I}: {G}enerators
  and categorification of the zeta function.
\newblock {\em ArXiv preprint 2504.15420}, 2025.

\bibitem[AT25b]{AlfieriTsang2}
Antonio Alfieri and Chi~Cheuk Tsang.
\newblock Heegaard {F}loer theory and pseudo-{A}nosov flows {II}:
  {D}ifferential and {F}ried pants.
\newblock {\em ArXiv preprint 2504.15420}, 2025.

\bibitem[BBM24]{BarthelmeBonattiMann_nontransitive}
Thomas Barthelmé, Christian Bonatti, and Kathryn Mann.
\newblock Non-transitive {A}nosov flows.
\newblock ArXiv 2411.03586, 2024.

\bibitem[BFFP23]{BFFP2}
Thomas Barthelm\'e, S\'ergio~R. Fenley, Steven Frankel, and Rafael Potrie.
\newblock Partially hyperbolic diffeomorphisms homotopic to the identity in
  dimension 3, {II}: {B}ranching foliations.
\newblock {\em Geom. Topol.}, 27(8):3095--3181, 2023.
\newblock \href {https://doi.org/10.2140/gt.2023.27.3095}
  {\path{doi:10.2140/gt.2023.27.3095}}.

\bibitem[BFFP24]{BFFP1}
Thomas Barthelm\'e, Sergio~R. Fenley, Steven Frankel, and Rafael Potrie.
\newblock Partially hyperbolic diffeomorphisms homotopic to the identity in
  dimension 3, {P}art {I}: {T}he dynamically coherent case.
\newblock {\em Ann. Sci. \'Ec. Norm. Sup\'er. (4)}, 57(2):293--349, 2024.

\bibitem[BFM25]{BarthelmeFrankelMann}
Thomas Barthelmé, Steven Frankel, and Kathryn Mann.
\newblock Orbit equivalences of pseudo-anosov flows.
\newblock {\em Inventiones Mathematicae}, 240(3):1119--1192, 2025.

\bibitem[BH99]{BridsonHaefliger}
Martin Bridson and André Haefliger.
\newblock {\em Metric spaces of non-positive curvature}.
\newblock Number 319 in Grundlehren der mathematischen Wissenschaften.
  Springer-Verlag, Heidelberg, 1999.

\bibitem[Bin57]{Bing_Dogbone}
R.~H. Bing.
\newblock A decomposition of $\operatorname{E^3}$ into points and tame arcs
  such that the decomposition space is topologically different from
  $\operatorname{E^3}$.
\newblock {\em Annals of Mathematics}, 65(3):484--500, 1957.

\bibitem[Bru95]{Brunella}
Marco Brunella.
\newblock Surfaces of section for expansive flows on three-manifolds.
\newblock {\em Journal of the Mathematical Society of Japan}, 47(3):491--501,
  1995.

\bibitem[BW72]{BowenWalters}
Rufus Bowen and Peter Walters.
\newblock Expansive one-parameter flows.
\newblock {\em Journal of Differential Equations}, 12:180--193, 1972.

\bibitem[Cal00]{Calegari_Rcovered}
Danny Calegari.
\newblock The geometry of $\mathbb{R}$-covered foliations.
\newblock {\em Geometry and Topology}, 4:457--515, 2000.

\bibitem[Cal06]{Calegari}
Danny Calegari.
\newblock {Universal circles for quasigeodesic flows}.
\newblock {\em Geometry and Topology}, 10(4):2271 -- 2298, 2006.

\bibitem[Cal11]{Calegari_blog}
Danny Calegari.
\newblock Quasigeodesic flows on hyperbolic 3-manifolds.
\newblock {\em Blog post available at
  https://lamington.wordpress.com/2011/12/20/quasigeodesic-flows-on-hyperbolic-3-manifolds/},
  2011.

\bibitem[CT07]{CannonThurston}
J.W. Cannon and W.P. Thurston.
\newblock Group invariant {P}eano curves.
\newblock {\em Geometry and Topology}, 11:1315--1355, 2007.

\bibitem[Fen02]{Fenley_Rcovered}
S{\'e}rgio Fenley.
\newblock Foliations, topology and geometry of 3-manifolds:
  $\mathbb{R}$-covered foliations and transverse pseudo-{A}nosov flows.
\newblock {\em Commentarii Mathematici Helvetici}, 77:415--490, 2002.

\bibitem[Fen09]{Fenley_continuousextension}
S{\'e}rgio Fenley.
\newblock Geometryof foliations and flows i: {A}lmost transverse
  pseudo-{A}nosov flows and asymptotic behavior of foliations.
\newblock {\em Journal of Differential Geometry}, 81:1--89, 2009.

\bibitem[Fen12]{Fenley}
S{\'e}rgio Fenley.
\newblock Ideal boundaries of pseudo-{A}nosov flows and uniform convergence
  groups with connections and applications to large scale geometry.
\newblock {\em Geometry and Topology}, 16(1):1--110, 2012.

\bibitem[Fen13]{Fenley_Rcoveredrigidity}
S{\'e}rgio Fenley.
\newblock Rigidity of pseudo-{A}nosov flows transverse to $\mathbb{R}$-covered
  foliations.
\newblock {\em Commentarii Mathematici Helvetici}, 88:643--676, 2013.

\bibitem[FM01]{FenleyMosher}
S{\'e}rgio Fenley and Lee Mosher.
\newblock Quasigeodesic flows in hyperbolic 3-manifolds.
\newblock {\em Topology}, 40(3):503--537, 2001.

\bibitem[Fra13]{Frankel_qgflows}
Steven Frankel.
\newblock Quasigeodesic flows and {M}{\"o}bius-like groups.
\newblock {\em Journal of Differential Geometry}, 93(3):401--429, 2013.

\bibitem[Fra15]{Frankel_spherefilling}
Steven Frankel.
\newblock Quasigeodesic flows and sphere-filling curves.
\newblock {\em Geometry and Topology}, 19(3):1249--1262, 2015.

\bibitem[Fra18]{Frankel_closing}
Steven Frankel.
\newblock Coarse hyperbolicity and closed orbits for quasigeodesic flows.
\newblock {\em Annals of Mathematics}, 188(1):1--48, 2018.

\bibitem[Fri79]{Fried_fiberedfaces}
David Fried.
\newblock Fibrations over ${S}^1$ with pseudo-{A}nosov monodromy.
\newblock In {\em Travaux de {T}hurston sur les surfaces (translated to
  {English} by Djun Kim and Dan Margalit}, number 66-67 in Ast\'erisque, pages
  251--266. Soci\'et\'e math\'ematique de France, 1979.

\bibitem[Fri82]{Fried}
David Fried.
\newblock The geometry of cross sections to flows.
\newblock {\em Topology}, 21(4):353--371, 1982.

\bibitem[Gab83]{Gabai}
David Gabai.
\newblock Foliations and the topology of 3-manifolds.
\newblock {\em Journal of Differential Geometry}, 18:445--503, 1983.

\bibitem[GK97]{GabaiKazez}
David Gabai and William Kazez.
\newblock Order trees and laminations of the plane.
\newblock {\em Mathematical Research Letters}, 4:603--616, 1997.

\bibitem[GMP25]{GMP}
Elena Gomes, Santiago Martinchich, and Rafael Potrie.
\newblock Invariant sets for homeomorphisms of hyperbolic 3-manifolds, 2025.
\newblock URL: \url{https://arxiv.org/abs/2504.03425}, \href
  {http://arxiv.org/abs/2504.03425} {\path{arXiv:2504.03425}}.

\bibitem[Gro87]{Gromov}
Mikhail Gromov.
\newblock Hyperbolic groups.
\newblock In S.~M. Gersten, editor, {\em Essays in Group Theory}, pages
  75--263. Springer New York, New York, NY, 1987.

\bibitem[Gro00]{Gromov_geodesic}
Mikhail Gromov.
\newblock Three remarks on geodesic dynamics and fundamental group.
\newblock {\em L'Enseignement Mathématique}, 46:391--402, 2000.

\bibitem[Han85]{Handel}
Michael Handel.
\newblock Global shadowing of pseudo-{A}nosov homeomorphisms.
\newblock {\em Ergodic Theory and Dynamical Systems}, 5(3):373–377, 1985.

\bibitem[HT85]{HandelThurston}
Michael Handel and William~P. Thurston.
\newblock New proofs of some results of {N}ielsen.
\newblock {\em Advances in Mathematics}, 56:173--191, 1985.

\bibitem[HY61]{HockingYoung}
John Hocking and Gail Young.
\newblock {\em Topology}.
\newblock Addison-Wesley Series in Mathematics. Addison-Wesley Pub. Co,
  Reading, MA, 1961.

\bibitem[Iak25]{Iakovoglou}
Ioannis Iakovoglou.
\newblock Markov partitions for non-transitive expansive flows.
\newblock {\em Comptes Rendus Mathématique}, 363:437--444, 2025.

\bibitem[IM90]{InabaMatsumoto}
Takashi Inaba and Shigenori Matsumoto.
\newblock Nonsingular expansive flows on 3-manifolds.
\newblock {\em Japanese Journal of Mathematics}, 16(2):329--340, 1990.

\bibitem[Kap01]{Kapovich}
Michael Kapovich.
\newblock {\em Hyperbolic manifolds and discrete groups}.
\newblock Progress in Mathematics. Birkhäuser, Boston, MA, 2001.

\bibitem[Kur68]{Kuratowskiv2}
Kazimierz Kuratowski.
\newblock {\em Topology}, volume~2.
\newblock Academic Press, New York, new edition, revised and augmented edition,
  1968.

\bibitem[LMT23]{LandryMinskyTaylor}
Michael~P. Landry, Yair~N. Minsky, and Samuel~J. Taylor.
\newblock Flows, growth rates, and the veering polynomial.
\newblock {\em Ergodic Theory and Dynamical Systems}, 43(9):3026--3107, 2023.

\bibitem[LMT25]{LandryMinskyTaylor_transversesurfaces}
Michael~P. Landry, Yair~N. Minsky, and Samuel~J. Taylor.
\newblock Flows, growth rates, and the veering polynomial.
\newblock {\em Journal für die reine und angewandte Mathematik (Crelles
  Journal) DOI: https://doi.org/10.1515/crelle-2025-0083}, 2025.

\bibitem[LT25]{LandryTsang}
Michael~P. Landry and Chi~Cheuk Tsang.
\newblock Endperiodic maps, splitting sequences, and branched surfaces.
\newblock {\em Geometry and Topology(to appear)}, pages 1--144, 2025.

\bibitem[Moo25]{Moore}
Robert~L. Moore.
\newblock Concerning upper semi-continuous collections of continua.
\newblock {\em Transactions of the American Mathematical Society}, 27:416--428,
  1925.

\bibitem[Mos92a]{Mosher_norm1}
Lee Mosher.
\newblock Dynamical systems and the homology norm of a 3-manifold {I}:
  {E}fficient intersection of surfaces and flows.
\newblock {\em Duke Mathematical Journal}, 65(3):416--428, 1992.

\bibitem[Mos92b]{Mosher_norm2}
Lee Mosher.
\newblock Dynamical systems and the homology norm of a 3-manifold {II}.
\newblock {\em Inventiones Mathematicae}, 107:243--281, 1992.

\bibitem[Mos92c]{Mosher_examples}
Lee Mosher.
\newblock Examples of quasi-geodesic flows on hyperbolic {$3$}-manifolds.
\newblock In {\em Topology '90 ({C}olumbus, {OH}, 1990)}, volume~1 of {\em Ohio
  State Univ. Math. Res. Inst. Publ.}, pages 227--241. de Gruyter, Berlin,
  1992.

\bibitem[Mos96]{Mosher}
Lee Mosher.
\newblock Laminations and flows transverse to finite-depth foliations, {P}art
  {I}.
\newblock Preprint, 1996.

\bibitem[Mun00]{Munkres}
James~R. Munkres.
\newblock {\em Topology}.
\newblock Prentice Hall, Upper Saddle River, NJ, second edition, 2000.

\bibitem[MZ37]{MontgomeryZippin}
Deane Montgomery and Leo Zippin.
\newblock Translation groups of three-space.
\newblock {\em American Journal of Mathematics}, 59(1):121--128, 1937.

\bibitem[Oka90]{Oka}
Masatoshi Oka.
\newblock Singular foliations on cross-sections of expansive flows on
  3-manifolds.
\newblock {\em Osaka Journal of Mathematics}, 27:863--883, 1990.

\bibitem[Per02]{Perelman1}
Grigori Perelman.
\newblock The entropy formula for the {R}icci flow and its geometric
  applications.
\newblock arXiv 0211159, 2002.

\bibitem[Per03a]{Perelman3}
Grigori Perelman.
\newblock Finite extinction time for the solutions to the {R}icci flow on
  certain three-manifolds.
\newblock arXiv 0303109, 2003.

\bibitem[Per03b]{Perelman2}
Grigori Perelman.
\newblock Ricci flow with surgery on three-manifolds.
\newblock arXiv 0303109, 2003.

\bibitem[Rat73]{Ratner}
Marina Ratner.
\newblock Markov partitions for {A}nosov flows on $n$-dimensional manifolds.
\newblock {\em Israel Journal of Mathematics}, 15:92--114, 1973.

\bibitem[Sha21]{Shannon}
Mario Shannon.
\newblock Hyperbolic models for transitive topological {A}nosov flows in
  dimension three.
\newblock arXiv 2108.12000, 2021.

\bibitem[Wal68]{Waldhausen}
Friedhelm Waldhausen.
\newblock On irreducible 3-manifolds which are sufficiently large.
\newblock {\em Annals of Mathematics}, 87(1):56--88, 1968.

\bibitem[Whi33]{Whitney}
Hassler Whitney.
\newblock Regular families of curves.
\newblock {\em Annals of Mathematics}, 34(2):244--270, 1933.

\bibitem[Wil63]{Wilder}
Raymond Wilder.
\newblock {\em Topology of Manifolds}, volume~32 of {\em American Mathematical
  Society Colloquium Publications}.
\newblock American Mathematical Society, Providence, RI, 1963.

\bibitem[Zeg93]{Zeghib}
Abdelghani Zeghib.
\newblock Sur les feuilletages géodésiques continus des variétés
  hyperboliques.
\newblock {\em Inventiones mathematicae}, 114:193--206, 1993.

\end{thebibliography}

\end{document}